\documentclass[11pt]{article}

\usepackage[margin=1.25in]{geometry}

\usepackage{amsmath, amssymb, amsthm, authblk, rotating, hyperref, longtable}
\usepackage[section]{placeins}
\usepackage[mathscr]{euscript}
\usepackage{stmaryrd}

\usepackage[sc]{mathpazo}
\usepackage[T1]{fontenc}

\title{Derived Equivalences of K3 Surfaces and Twined Elliptic Genera
}

\author[1]{John F. R. Duncan\thanks{E-mail: \texttt{john.duncan@emory.edu}}}
\author[2]{Sander Mack-Crane\thanks{E-mail: \texttt{sander@berkeley.edu}}}
\affil[1]{Department of Mathematics and Computer Science, Emory University, Atlanta, GA 30322, U.S.A.}
\affil[2]{Department of Mathematics, University of California at Berkeley, Berkeley, CA 94720, U.S.A.}
\date{}

\newcommand{\wt}[1]{\widetilde{#1}}
\newcommand{\wh}[1]{\widehat{#1}}

\newcommand{\tw}{\mathrm{tw}}
\renewcommand{\a}{\mathfrak{a}}

\renewcommand{\aa}{\hat{\mathfrak{a}}}
\newcommand{\CC}{\mathbb{C}}
\newcommand{\e}{\mathfrak{e}}
\newcommand{\HH}{\mathbb{H}}
\newcommand{\LL}{\Lambda}

\newcommand{\PP}{\mathbb{P}}
\newcommand{\QQ}{\mathbb{Q}}
\newcommand{\RR}{\mathbb{R}}
\newcommand{\VU}{{\vsn}}
\newcommand{\VV}{{\vsn_\tw}}
\newcommand{\Jo}{J(0)}
\newcommand{\jo}{\jmath}
\newcommand{\Lo}{L(0)}
\newcommand{\Lm}{L(m)}
\newcommand{\Ln}{L(n)}
\newcommand{\vv}{\mathbf{v}}
\newcommand{\zz}{\mathfrak{z}}
\newcommand{\w}{\omega}
\newcommand{\wc}{\bar{\omega}}
\newcommand{\ZZ}{\mathbb{Z}}
\newcommand{\ii}{\textbf{i}}
\newcommand{\up}{\bigtriangleup}
\newcommand{\dn}{\bigtriangledown}

\newcommand{\Aut}{\operatorname{Aut}}
\newcommand{\be}{\begin{equation}}
\newcommand{\Cliff}{\operatorname{Cliff}}
\newcommand{\Co}{\textsl{Co}}
\newcommand{\coa}{G}	
\newcommand{\CM}{\operatorname{CM}}
\newcommand{\Db}{\operatorname{D}^{\rm b}}
\newcommand{\ee}{\end{equation}}
\newcommand{\End}{\operatorname{End}}
\newcommand{\Id}{\operatorname{Id}}
\newcommand{\tr}{\operatorname{tr}}

\newcommand{\rk}{\operatorname{rk}}
\newcommand{\II}{I\!I}
\newcommand{\lab}{\langle}
\newcommand{\mlc}{\widetilde{H}(X,\ZZ)\otimes_{\ZZ}{\CC}}
\newcommand{\mlcoo}{\widetilde{H}^{1,1}(X,\ZZ)\otimes_{\ZZ}{\CC}}
\newcommand{\mlr}{\widetilde{H}(X,\ZZ)\otimes_{\ZZ}\RR}
\newcommand{\mlroo}{\widetilde{H}^{1,1}(X,\ZZ)\otimes_{\ZZ}{\RR}}
\newcommand{\mlz}{\widetilde{H}(X,\ZZ)}
\newcommand{\mlzoo}{\widetilde{H}^{1,1}(X,\ZZ)}

\newcommand{\Spin}{\operatorname{Spin}}
\newcommand{\SO}{\operatorname{SO}}
\newcommand{\SL}{\operatorname{SL}}
\newcommand{\U}{\operatorname{U}}
\newcommand{\rab}{\rangle}
\newcommand{\rank}{\operatorname{rank}}
\newcommand{\Span}{\operatorname{Span}}
\newcommand{\Stab}{\operatorname{Stab}}
\newcommand{\Stabo}{\operatorname{Stab}^{\circ}}

\newcommand{\vn}{V^{\natural}}
\newcommand{\vsn}{V^{s\natural}}
\newcommand{\vsnt}{V^{s\natural}_\tw}
\newcommand{\vir}{\mathcal{V}}
\newcommand{\MM}{\mathbb{M}}

\newtheorem{theorem}{Theorem}[section]
\newtheorem*{stheorem}{Theorem}
\newtheorem{proposition}[theorem]{Proposition}
\newtheorem{lemma}[theorem]{Lemma}
\newtheorem{conjecture}[theorem]{Conjecture}
\newtheorem*{sconjecture}{Conjecture}

\numberwithin{equation}{section}

\begin{document}
\maketitle

\begin{abstract}
We use the unique canonically-twisted module over a certain distinguished super vertex operator algebra---the moonshine module for Conway's group---to attach a weak Jacobi form of weight zero and index one to any symplectic derived equivalence of a projective complex K3 surface that fixes a stability condition in the distinguished space identified by Bridgeland. According to work of Huybrechts, following Gaberdiel--Hohenegger--Volpato, any such derived equivalence determines a conjugacy class in Conway's group, the automorphism group of the Leech lattice. Conway's group acts naturally on the module we consider.

In physics the data of a projective complex K3 surface together with a suitable stability condition determines a supersymmetric non-linear sigma model, and supersymmetry preserving automorphisms of such an object may be used to define twinings of the K3 elliptic genus. Our construction recovers the K3 sigma model twining genera precisely in all available examples. In particular, the identity symmetry recovers the usual K3 elliptic genus, and this signals a connection to Mathieu moonshine. A generalization of our construction recovers a number of the Jacobi forms arising in umbral moonshine.

We demonstrate a concrete connection to supersymmetric non-linear K3 sigma models by establishing an isomorphism between the twisted module we consider and the vector space underlying a particular sigma model attached to a certain distinguished K3 surface. 
\end{abstract}

\clearpage

\tableofcontents

\section{Introduction}\label{sec:intro}

The main result of this paper is a construction which attaches weak Jacobi forms to suitably defined autoequivalences of the bounded derived category of coherent sheaves on a complex projective K3 surface.

The origins of our method extend back to the monstrous moonshine phenomenon, initiated by the observations of McKay and Thompson \cite{Tho_FinGpsModFns,Tho_NmrlgyMonsEllModFn}, Conway--Norton \cite{MR554399}, and Queen \cite{MR628715}. The more recent Mathieu moonshine observation of Eguchi--Ooguri--Tachikawa \cite{Eguchi2010}, and its extension \cite{UM,UMNL,mumcor} to Niemeier lattice root systems is also closely related.

Our results also have physical significance. As we explain presently, they suggest that a certain distinguished super vertex operator algebra is a universal object for supersymmetric non-linear K3 sigma models. This represents a new role for vertex algebra in physics: rather than serving as the ``chiral half'' of a particular, holomorphically factorizable super conformal field theory, the super vertex operator algebra in question is, evidently, simultaneously related to a diverse family of super conformal field theories.

\subsection{Monstrous Moonshine}\label{sec:intro:mm}

To explain the connection to monstrous moonshine recall that an isomorphism of Riemann surfaces
\begin{gather}\label{eqn:intro-Tg}
	T_g:\Gamma_g\backslash\wh{\HH}\xrightarrow{\sim} \wh{\CC}
\end{gather}
is attached to each conjugacy class $[g]$ in the monster group $\MM$ by the work \cite{MR554399} of Conway--Norton. Here $\wh{\CC}:=\CC\cup\{\infty\}\simeq \PP^1$ denotes the Riemann sphere, we set
\begin{gather}
\wh{\HH}:=\HH\cup\QQ\cup\{\infty\},
\end{gather}
for $\HH=\{\tau\in\CC\mid\Im(\tau)>0\}$, and $\Gamma_g$ is a discrete subgroup of $\SL_2(\RR)$ that is commensurable with the modular group $\SL_2(\ZZ)$. Actually, $\Gamma_g$ always lies between some $\Gamma_0(N)$ and its normalizer in $\SL_2(\RR)$ (cf. \S\ref{sec:autfms} for $\Gamma_0(N)$ and its normalizer), and the subgroup of upper-triangular matrices in $\Gamma_g$ is generated by $\pm\left(\begin{smallmatrix}1&1\\0&1\end{smallmatrix}\right)$ for all $g\in \MM$. Further, $T_g$ maps the $\Gamma_g$-orbit containing $\infty\in\wh{\HH}$ to $\infty\in\wh{\CC}$. So $T_g$ (restricted to $\HH$) admits a Fourier series expansion 
\begin{gather}
	T_g(\tau)=\sum_{n\geq -1}c_g(n)q^n
\end{gather}
for some $c_g(n)\in \CC$, where $q=e^{2\pi \ii \tau}$. (We choose a square root of $-1$ in $\CC$ and denote it $\ii$.) Moreover, $c_g(-1)=1$ and $c_g(0)=0$ for all $g\in \MM$.

For $g=e$ the identity in $\MM$ we have $\Gamma_e=\SL_2(\ZZ)$, so $T_e(\tau)$ is almost the classical complex elliptic $j$-invariant,
\begin{gather}
	\begin{split}
	T_e(\tau)
	&=j(\tau)-744\\
	&=q^{-1} + 196884 q + 21493760 q^2 + 864299970 q^3 + 20245856256 q^4+\ldots
	\end{split}
\end{gather}
McKay's original moonshine observation is that $196884=1+196883$, where $196883$ is the dimension of the first non-trivial irreducible representation of $\MM$. Thompson extended this \cite{Tho_NmrlgyMonsEllModFn} and posited the existence of a graded infinite-dimensional $\MM$-module 
\begin{gather}
\vn=\bigoplus_{n\geq -1}\vn_n
\end{gather} 
such that $J(\tau):=j(\tau)-744$ is the generating function of the dimensions of its homogeneous subspaces. The $T_g$ of Conway--Norton \cite{MR554399} are explicit predictions for what the graded traces of elements $g\in \MM$ on $\vn$ should be, 
\begin{gather}\label{eqn:intro-Tgvn}
	T_g(\tau)=\sum_{n\geq -1}(\tr_{\vn_n}g) q^n.
\end{gather}

The $\MM$-module $\vn$ was constructed concretely by Frenkel--Lepowsky--Meurman \cite{FLMPNAS,FLMBerk,FLM}. The identities (\ref{eqn:intro-Tgvn}) were established for all $g\in \MM$ by Borcherds \cite{borcherds_monstrous}.

We refer to \cite{mnstmlts} for a recent review of moonshine, including a much fuller description of the above developments, and many more references.

The most obvious connection between moonshine and this article starts with the {\em multiplicative moonshine} observation of Conway--Norton (cf. \S9 of \cite{MR554399}), considered in detail by Queen \cite{MR628715}, which attaches analogues of the $T_g$ of (\ref{eqn:intro-Tg}) to elements of the Conway group $\Co_0$, a $2$-fold cover of the sporadic simple group $\Co_1$. 

The Conway group may be realized explicitly as the automorphism group of the Leech lattice $\LL$,
\begin{gather}
	\Co_0=\Aut(\LL).
\end{gather} 
We have $\LL\otimes_\ZZ\CC\simeq \CC^{24}$ so $\Co_0$ comes equipped with a $24$-dimensional representation over $\CC$. (Cf. \S\ref{sec:lat} for more on the $\Co_k$ and $\LL$.) Choose $g\in \Co_0$ and let $\varepsilon_1,\ldots,\varepsilon_{24}$ be the associated eigenvalues. Queen confirmed \cite{MR628715} (cf. also \cite{MR780666}) that 
\begin{gather}
	t_g(\tau):=q^{-1}\prod_{n>0}\prod_{i=1}^{24}(1-\varepsilon_iq^{2n-1})
\end{gather}
defines an isomorphism of Riemann surfaces $\Gamma_g\backslash\wh{\HH}\to\wh{\CC}$, as in (\ref{eqn:intro-Tg}), for some discrete group $\Gamma_g<\SL_2(\RR)$, for any $g\in \Co_0$. 

Note that, in contrast to $T_g$, the constant term in the Fourier expansion of $t_g$ is $-\chi_g$, where
\begin{gather}\label{eqn:intro-chig}
\chi_g:=\sum_i\varepsilon_i,
\end{gather} 
and this value is generally non-vanishing. Define $T^s_g$, for $g\in \Co_0$, by setting
\begin{gather}
	T^s_g(\tau):=t_g(\tau/2)+\chi_g=q^{-1/2}+O(q^{1/2}).
\end{gather}
In our earlier work \cite{vacogm} we obtained the Conway group analogues of the results of Frenkel--Lepowsky--Meurman and Borcherds on monstrous moonshine mentioned above, constructing a graded infinite-dimensional $\Co_0$-module 
\begin{gather}
\vsn=\bigoplus_{n\geq -1}\vsn_{n/2},
\end{gather}
and showing that
\begin{gather}
	T^s_g(\tau)=\sum_{n\geq -1}(\tr_{\vsn_{n/2}}g)q^{n/2}
\end{gather}
for $g\in \Co_0$.

More than this, and in direct analogy with $\vn$, the $\Co_0$-module $\vsn$ comes equipped with a distinguished super vertex operator algebra structure. (Cf. \S\ref{sec:va} for a recap on vertex algebra, and \S\ref{sec:dist} for the construction and characterization of $\vsn$.) 

The reader familiar with vertex algebra will no doubt also be aware of the modularity results on trace functions attached to vertex operator algebras (cf. \cite{Zhu_ModInv,Dong2000,MR2046807}) and super vertex operator algebras (cf. \cite{DonZha_MdlrOrbVOSA,MR3077918,MR3205090}). The results of \cite{Dong2000}, for example, go a long way to explaining why the right hand side of (\ref{eqn:intro-Tgvn}) should define a holomorphic function on $\HH$ that is invariant for some congruence subgroup of $\SL_2(\ZZ)$. Interestingly, there is as yet no conceptual understanding of why (\ref{eqn:intro-Tgvn}) should actually satisfy the much stronger condition of defining an isomorphism as in (\ref{eqn:intro-Tg}), but see \cite{DunFre_RSMG}, or \S6 of the review \cite{mnstmlts}, for a conjectural proposal to establish a theory that would achieve this.

\subsection{K3 Surfaces and Jacobi Forms}
\label{sec:intro:k3}

In this article we use the unique (up to equivalence) canonically-twisted $\vsn$-module to attach a Jacobi form $\phi_g$ (cf. (\ref{eqn:tg-phig})) to a suitable derived autoequivalence $g$ of a complex projective K3 surface $X$. More precisely, we prove the following result in \S\ref{sec:tg}.
\begin{stheorem}[\ref{theorem:jacobi form}]
Let $X$ be a projective complex K3 surface and let $\sigma$ be a stability condition in Bridgeland's space. If $g$ is a symplectic autoequivalence of the derived category of coherent sheaves on $X$ that preserves $\sigma$ then $\phi_g$ is a weak Jacobi form of weight $0$, index $1$, and some level.
\end{stheorem}

Jacobi forms (cf. \cite{eichler_zagier} or \S\ref{sec:autfms}) are $2$-variable analogues of modular forms, admitting transformation formulas under a group of the form $\SL_2(\ZZ)\ltimes \ZZ^2$ (or a finite index subgroup thereof), that are modeled on those of the classical Jacobi theta functions $\vartheta_i(\tau,z)$ (cf. (\ref{equation:theta1})). Jacobi forms also appear as Fourier coefficients of Siegel modular forms (cf. \cite{feingold_frenkel}). Roughly, a Jacobi form has level if it is required to transform only under some congruence subgroup $\Gamma<\SL_2(\ZZ)$, and the term weak refers to certain growth conditions at the cusps of $\Gamma$. 

See \S\ref{sec:de} for a brief review of K3 surfaces, their symplectic derived autoequivalences, and stability conditions, and see \S\ref{sec:va} for the notion of canonically-twisted module over a super vertex operator algebra.

Note that the appearance of Jacobi forms in vertex algebra goes back to the work of Kac--Peterson \cite{MR750341} (cf. also \cite{MR1104219}) on basic representations of affine Lie algebras. (In particular, it actually predates Borcherds' introduction of the notion of vertex algebra in \cite{Bor_PNAS}.) More general results were established recently in \cite{MR2925472}, by applying earlier work \cite{MR1738180} of Miyamoto. Cf. also \cite{MR2201000}. Vertex algebraic constructions were used to attach Jacobi forms to conjugacy classes in the sporadic simple group of Rudvalis in \cite{Duncan2008,Duncan2008a}.

What is the meaning of the functions $\phi_g$ of Theorem \ref{theorem:jacobi form}? One answer to this question is furnished by physics. More specifically, Theorem \ref{theorem:jacobi form} can be interpreted as a statement about supersymmetric non-linear sigma models on K3 surfaces. For as explained by Huybrechts in \cite{2013arXiv1309.6528H}, the analyses of \cite{MR1479699,MR1416354,NahWen} (cf. \cite{GHV} for a concise account) suggest the conjecture that the pairs $(X,\sigma)$ with $X$ and $\sigma$ as in Theorem \ref{theorem:jacobi form} are in natural correspondence with the supersymmetric non-linear sigma models on complex projective K3 surfaces.

\subsection{Sigma Models}

Witten introduced \cite{MR970278,MR885560} a construction which attaches a weak Jacobi form for $\SL_2(\ZZ)$ to any supersymmetric non-linear sigma model, called the {\em elliptic genus} of the model in question. It turns out (cf. (\ref{eqn:tg-K3EG})) that $\phi_g$ is exactly the K3 elliptic genus when $g=e$ is the identity autoequivalence. (In particular, both $\phi_e$ and the K3 elliptic genus are independent of the choices of $X$ and $\sigma$.)

Generalizing this, the analysis of \cite{GHV} suggests that we can expect to obtain a Jacobi form with level, called a {\em twined elliptic genus}, from any supersymmetry preserving automorphism of a supersymmetric non-linear sigma model. In terms of the pairs $(X,\sigma)$, such automorphisms should correspond to symplectic autoequivalences that preserve $\sigma$. Thus it is natural to compare the $\phi_g$ to twined elliptic genera of supersymmetric non-linear K3 sigma models.

Unfortunately it is generally a difficult matter to compute twined K3 elliptic genera, for the Hilbert spaces attached to supersymmetric non-linear K3 sigma models are only known in a few special cases. However, it has been shown recently by Gaberdiel--Hohenegger--Volpato \cite{GHV} (cf. also \cite{2013arXiv1309.6528H}) that any group of supersymmetry preserving automorphisms of such a model can be embedded in the Conway group $\Co_0$. (Actually, $\Co_0$ here can be replaced by $\Co_1$, but it seems to be more natural to regard $\Co_0$ as the operative group.) More specifically (and subject to some assumptions about the moduli space of K3 sigma models), the groups of supersymmetry preserving automorphisms of K3 sigma models are exactly the subgroups of $\Co_0$ that pointwise fix a $4$-dimensional subspace of $\LL\otimes_\ZZ\RR$, according to \cite{GHV}.

Thus there is hope that a suitably defined $\Co_0$-module may be used to recover all the twined K3 elliptic genera, bypassing the explicit construction of super conformal field theories attached to K3 sigma models. The present work furnishes strong evidence that this is indeed the case, and that $\vsnt$ is precisely the $\Co_0$-module to consider. Indeed, about half of the conjugacy classes in $\Co_0$ that fix a $4$-space in $\LL\otimes_\ZZ\RR$ appear in the explicit computations of \cite{GHV,Gaberdiel:2012um,2014arXiv1403.2410V}, and we find precise agreement with the $\phi_g$, defined via $\vsnt$, in every case.

This leads us to the following conjecture, indicating one precise sense in which $\vsn$ may serve as a universal object for K3 sigma models. 
\begin{sconjecture}[\ref{conj:sigmod}]
The twined elliptic genus attached to any supersymmetry preserving automorphism of a supersymmetric non-linear K3 sigma model coincides with $\phi_g$ for some $g\in \Co_0$ fixing a $4$-space in $\LL\otimes_\ZZ\RR$.
\end{sconjecture}
It will be interesting to see if $\vsn$ cannot ultimately shed light on more subtle aspects of K3 sigma models, beyond their twined elliptic genera. 

It is at first surprising that the central charge of $\vsn$ is twice that of the super conformal field theories attached to K3 sigma models, i.e. $12$ rather than $6$. In \S\ref{sec:sigmod} we give an explanation for this discrepancy by demonstrating an isomorphism of Virasoro-modules between $\vsn$ and the Neveu--Schwarz sector of the super conformal field theory attached to a particular, distinguished K3 sigma model, which has been considered earlier in \cite{Wen2002,2013arXiv1309.4127G}. See Proposition \ref{prop:sigmod-Virmodisom}. Note that we naturally obtain a Virasoro-module structure of central charge $12$ on the sigma model by taking the diagonal copy of the Virasoro algebra, within the two commuting copies that act on left and right movers, respectively.

For concreteness we have chosen to formulate our main results in terms of derived categories of coherent sheaves, and their stability conditions, rather than sigma models. We refer the reader to \cite{MR2405589,MR2483930,MR1957548} for introductory expositions of the deep connection between these notions.

\subsection{Mathieu Moonshine}

It is a striking fact that the K3 elliptic genus is involved in another, more recently discovered moonshine phenomenon, which is, at first glance, seemingly unrelated to the monster. Namely, it was observed by Eguchi--Ooguri--Tachikawa \cite{Eguchi2010} that if the K3 elliptic genus $Z_{\rm K3}=\phi_e$ is written in the form
\begin{gather}
	Z_{\rm K3}(\tau,z)=24\mu(\tau,z) \frac{\vartheta_1(\tau,z)^2}{\eta(\tau)^3}+H^{(2)}(\tau) \frac{\vartheta_1(\tau,z)^2}{\eta(\tau)^3},
\end{gather}
where $\eta$ is the Dedekind eta function (cf. (\ref{eqn:autfms-eta})), $\vartheta_1$ is the usual Jacobi theta function (cf. (\ref{equation:theta1})), and $\mu$ denotes the Appell--Lerch sum
\begin{gather}
	\mu(\tau,z):=\frac{\ii y^{1/2}}{\vartheta_1(\tau,z)}\sum_{n\in\ZZ}(-1)^n\frac{y^nq^{n(n+1)/2}}{1-y q^n},
\end{gather}
where $q=e^{2\pi \ii\tau}$ and $y=e^{2\pi \ii z}$, then $q^{1/8}H^{(2)}(\tau)$ is a power series in $q$ with integer coefficients,
\begin{gather}\label{eqn:um-HFouexp}
	H^{(2)}(\tau)=-2q^{-1/8}+90q^{7/8}+462q^{15/8}+1540q^{23/8}+4554q^{31/8}+11592q^{39/8}+\ldots,
\end{gather}
and the coefficient of each non-polar term appearing in (\ref{eqn:um-HFouexp}) is twice the dimension of an irreducible representation of the largest sporadic simple group of Mathieu, $M_{24}$. (The meaning of the superscript $^{(2)}$ will be elucidated presently.)

Inspired by the monstrous antecedent, it is natural to conjecture the existence of a graded infinite-dimensional $M_{24}$-module
\begin{gather}\label{eqn:intro-K}
	\check K^{(2)}=\bigoplus_{n>0}\check K^{(2)}_{n-1/8},
\end{gather}
such that $H^{(2)}(\tau)=-2q^{-1/8}+\sum_{n>0}\dim (\check K^{(2)}_{n-1/8}) q^{n-1/8}$, and investigate the series
\begin{gather}\label{eqn:intro-HgtrK}
	H^{(2)}_g(\tau):=-2q^{-1/8}+\sum_{n>0}(\tr_{\check K^{(2)}_{n-1/8}}g)q^{n-1/8}
\end{gather}
for $g\in M_{24}$.

The work of Cheng \cite{MR2793423}, Eguchi--Hikami \cite{Eguchi2010a}, and Gaberdiel--Hohenegger--Volpato \cite{Gaberdiel2010a, Gaberdiel2010} determined precise candidates for the (\ref{eqn:intro-HgtrK}), and found, moreover, that if $\chi_g$ denotes the number of fixed points of $g\in M_{24}$ in its unique (up to equivalence) non-trivial permutation action on $24$ points, then 
\begin{gather}\label{eqn:intro-Zg}
	Z^{(2)}_g(\tau,z):=\chi_g\mu(\tau,z)\frac{\vartheta_1(\tau,z)^2}{\eta(\tau)^3}+H^{(2)}_g(\tau) \frac{\vartheta_1(\tau,z)^2}{\eta(\tau)^3}
\end{gather}
is a weak Jacobi form of weight $0$ and index $1$, with some level depending on $g$.

The group $M_{24}$ appears naturally as a subgroup of $\Co_0$, in such a way that the definition of $\chi_g$ just given coincides with (\ref{eqn:intro-chig}) for $g\in M_{24}<\Co_0$, so we may compare the $\phi_g$ of Theorem \ref{theorem:jacobi form} to the weak Jacobi forms $Z^{(2)}_g$ of Mathieu moonshine. Interestingly, $\phi_g=Z^{(2)}_g$ for $g$ in all but $7$ of the $26$ conjugacy classes of $M_{24}$. Cf. Table \ref{table:4spacecoin}. The conjugacy classes of $M_{24}$ for which $\phi_g\neq Z^{(2)}_g$ are those named $3B$, $6B$, $12B$, $21A$, $21B$, $23A$ and $23B$ in \cite{atlas}. Note that $\phi_g$ is not even defined for $g$ in any of the last $5$ of these, since such elements of $\Co_0$ do not pointwise fix a $4$-space in $\LL\otimes_\ZZ\RR$.

Regarding the $M_{24}$-module $\check K^{(2)}$ of (\ref{eqn:intro-K}), Gannon has proven \cite{Gannon:2012ck} that the candidate $H^{(2)}_g$ determined in \cite{MR2793423,Eguchi2010a,Gaberdiel2010a,Gaberdiel2010} are indeed the graded trace functions attached to a graded $M_{24}$-module, but there is, as yet, no analogue for $\check K^{(2)}$ of the vertex algebraic constructions of $\vn$ or $\vsn$. The fact that $\phi_g$ recovers $Z^{(2)}_g$ for so many $g\in M_{24}$ suggests that $\vsn$ may play an important role in determining such a concrete construction.

See \cite{MR2985326} for a detailed review of Mathieu moonshine, including explicit descriptions of the $H^{(2)}_g$. The $H^{(2)}_g$ are examples of mock modular forms of weight $1/2$, a notion which has arisen fairly recently, thanks to the foundational work of Zwegers \cite{zwegers} on Ramanujan's mock theta functions \cite{MR947735,MR2280843}, contemporaneous work \cite{BruFun_TwoGmtThtLfts} of Bruinier--Funke on harmonic Maass forms, and subsequent contributions by Bringmann--Ono \cite{BringmannOno2006} and Zagier \cite{zagier_mock}. We refer to \cite{zagier_mock} and \cite{Ono_unearthing} for introductory accounts of mock modular forms. The $H^{(2)}_g$ for $g\in M_{24}$ have been constructed uniformly in \cite{Cheng2011}, and related results appear in \cite{jacteg}.

\subsection{Umbral Moonshine}

The superscripts in $\check K^{(2)}$, $H^{(2)}_g$ and $Z^{(2)}_g$ indicate that Mathieu moonshine is but one 
case of a more generally defined theory. Indeed, the observations of \cite{Eguchi2010} were extended in \cite{UM,UMNL} (cf. also \cite{mumcor}), to an association of (vector-valued) mock modular forms $H^{(\ell)}_g=(H^{(\ell)}_{g,r})$ to conjugacy classes $[g]$ in finite groups $G^{(\ell)}$ (with $G^{(\ell)}=M_{24}$ for $\ell=2$), for certain symbols $\ell$, called {\em lambencies}. The resulting collection of relationships between finite groups and mock modular forms is now known as umbral moonshine.

The lambencies of \cite{UMNL} are in correspondence with the $23$ (non-empty) simply-laced root systems that arise in even self-dual positive-definite lattices of rank $24$. These are the so-called {Niemeier root systems} (cf. \S\ref{sec:lat}). For example, if $n$ is a divisor of $24$ and $k=24/n$, then $\ell=n+1$ corresponds to the union of $k$ copies of the $A_n$ root system, denoted $A_n^k$. In particular, $\ell=2$ corresponds to $A_1^{24}$. 

The group $G^{(\ell)}$ is, by definition, the outer automorphism group of the self-dual lattice $N^{(\ell)}$ whose Niemeier root system corresponds to $\ell$. That is, $G^{(\ell)}:=\Aut(N^{(\ell)})/W^{(\ell)}$ where $W^{(\ell)}$ is the normal subgroup of $\Aut(N^{(\ell)})$ generated by reflections in root vectors. Note that all of these groups $G^{(\ell)}$ embed in $\Co_0$.
 
According to the McKay correspondence \cite{McKay_Corr,MR661802}, the irreducible simply-laced root systems are in correspondence with certain surface singularities called du Val singularities (cf. e.g. \cite{MR543555}). Thus the governing role of simply-laced root systems in umbral moonshine suggests a geometric interpretation involving non-smooth K3 surfaces equipped with configurations of du Val singularities. Evidence in support of this idea is developed in \cite{2014arXiv1406.0619C}. A number of the weak Jacobi forms $\phi_g$ constructed here appear also in \cite{2014arXiv1406.0619C}.
 
As in the case that $\ell=2$, the particular properties of the mock modular forms $H^{(\ell)}_g$ support the existence of graded infinite-dimensional $G^{(\ell)}$-modules
\begin{gather}\label{eqn:intro-Kell}
	\check K^{(\ell)}=\bigoplus_{r\in I^{(\ell)}}\bigoplus_{\substack{n\in\ZZ\\r^2-4mn\leq0}}\check K^{(\ell)}_{r,n-r^2/4m}
\end{gather}
such that 
\begin{gather}\label{eqn:intro-HelltrKell}
	H^{(\ell)}_{g,r}(\tau)=-2q^{-1/4m}\delta_{r,1}+\sum_{n}(\tr_{\check K^{(\ell)}_{r,n-r^2/4m}}g)q^{n-r^2/4m}
\end{gather}
for $g\in G^{(\ell)}$ and $r\in I^{(\ell)}$, where $m$ is a certain positive integer depending on $\ell$, and $I^{(\ell)}$ is a certain subset of $\{1,\ldots,m-1\}$. (We refer the reader to \cite{UMNL,mumcor} or \S9 of \cite{mnstmlts} for a fuller discussion of $\check K^{(\ell)}$ and its relation to $G^{(\ell)}$ and the $H^{(\ell)}_g$.) 

The existence of  $G^{(\ell)}$-modules $\check K^{(\ell)}$ satisfying (\ref{eqn:intro-HelltrKell}) is one of the main conjectures of umbral moonshine, and has now been proven \cite{umrec} for all Niemeier root systems. A concrete, vertex algebraic construction of $\check K^{(\ell)}$ has been established recently \cite{mod3e8} for the special case that $\ell=30+6,10,15$, which is the lambency corresponding to the root system $E_8^3$. Vertex algebraic constructions of $G^{(\ell)}$-modules closely related to the $\check K^{(\ell)}$ appear in \cite{umvan4} for the lambencies corresponding to $A_3^8$, $A_4^6$, $A_6^4$ and $A_{12}^2$, and in \cite{umvan2} for $D_6^4$, $D_8^3$, $D_{12}^2$ and $D_{24}$.

In the case of $A_{\ell-1}^k$, where $k(\ell-1)=24$, the mock modular forms $H^{(\ell)}_g$, together with certain characters $\chi_g$ and $\bar{\chi}_g$ of $G^{(\ell)}$, can be used to define weak Jacobi forms $Z^{(\ell)}_g$ of weight 0 and index $\ell-1$ (and some level depending on $g$), in a natural way. (See \S4 of \cite{UM} for the details of this construction.)

In \S\ref{sec:um} we present a natural generalization of the construction of $\phi_g$ in \S\ref{sec:tg}, and in so doing attach a weak Jacobi form $\phi^{(\ell)}_g$, of weight $0$ and index $\ell-1$, to any element $g\in \Co_0$ that fixes a $2d$-dimensional subspace of $\LL\otimes_\ZZ\RR$, where $d=2(\ell-1)$. Interestingly, many of the $Z^{(\ell)}_g$ of umbral moonshine are realized as (scalar multiples of) $\phi^{(\ell)}_g$ for suitable $g\in \Co_0$. Thus we have evidence that $\vsn$ may be an important device for realizing a number of the $\check K^{(\ell)}$ explicitly. The particular coincidences between $Z^{(\ell)}_g$ and $\phi^{(\ell)}_g$ are recorded in \S\ref{sec:tables-coins}, Tables \ref{table:4spacecoin} through \ref{table:24spacecoin}.

Surprisingly, $\vsn$ can be used to attach mock modular forms to conjugacy classes in finite groups beyond those arising as $G^{(\ell)}$ for some lambency $\ell$. Indeed, in \cite{2014arXiv1406.5502C} the canonically-twisted $\vsn$-module $\vsnt$ is used to attach $2$-vector-valued mock modular forms of weight $1/2$ to conjugacy classes in any subgroup of $\Co_0$ fixing a $2$-space in $\LL\otimes_\ZZ\RR$. In this way mock modular forms are attached to the conjugacy classes of the sporadic Mathieu groups $M_{23}$ and $M_{12}$, McLaughlin's sporadic group $McL$, and the sporadic group $HS$ of Higman and Higman--Sims (cf. \cite{MR0227269,MR0240193,MR0248215}). An association of mock modular forms to conjugacy classes in subgroups of $\Co_0$ fixing $3$-spaces in $\LL\otimes_\ZZ\RR$ is also considered in \cite{2014arXiv1406.5502C}. Consequently, mock modular forms (of a different kind) are attached to $M_{22}$ and $M_{11}$. See \cite{2015arXiv150307219C} (and its prequel \cite{2014arXiv1412.2804B}) for an extension of this method to subgroups of $\Co_0$ that fix a line in $\LL\otimes_\ZZ\RR$. This analysis associates mock modular forms (of yet another variety) to the sporadic groups $\Co_2$, $\Co_3$ and $M_{24}$.

It is evident from the above mentioned results that $\vsn$ should play an important role in umbral moonshine. Thus the close relationship between Conway moonshine and monstrous moonshine serves to motivate the possibility that monstrous and umbral moonshine are related in a deep and direct way, potentially sharing a common origin. The results of \cite{2014arXiv1403.3712O} also motivate this point of view. See the introduction to \cite{UM} for related discussion.

\subsection{Organization}

We now describe the structure of the paper. Since the main result involves a number of different topics, not typically seen together in a single work, we begin with a number of brief preliminary sections. We recall some basic facts about even self-dual lattices, and also discuss the Conway group in \S\ref{sec:lat}. We then recall modular forms, Jacobi forms and certain special examples of such functions in \S\ref{sec:autfms}. We review the main results from \cite{2013arXiv1309.6528H} on derived equivalences of K3 surfaces in \S\ref{sec:de}. In \S\ref{sec:va} we recall basic definitions in vertex algebra theory, and give a brief description of the Clifford module (a.k.a. free fermion) super vertex algebra construction in \S\ref{sec:cliff}.

The results of our earlier work \cite{vacogm} play an important role here, and we review these next, recalling some useful formulas relating to spin modules in \S\ref{sec:spin}, and the construction of the distinguished super vertex operator algebra $\vsn$ in \S\ref{sec:dist}.

In \S\ref{sec:tg} we establish our main result: a mechanism which attaches a weak Jacobi form to any symplectic derived equivalence of a K3 surface that fixes a suitable stability condition. See Theorem \ref{theorem:jacobi form}. 

We also formulate a conjecture relating the Jacobi forms so arising to twined elliptic genera of K3 sigma models. In brief, all the examples of twined K3 elliptic genera available in the literature are recovered from our construction. This suggests that the super vertex operator algebra $\vsn$ serves as a universal object for K3 sigma models. 

The construction of \S\ref{sec:tg} easily generalizes so as to recover a number of the weak Jacobi forms of umbral moonshine. We discuss this in detail in \S\ref{sec:um}. 

We give some deeper evidence for the conjectural relationship between $\vsn$ and K3 sigma models in \S\ref{sec:sigmod}, by exhibiting an isomorphism of graded vector spaces between $\vsn$ and the super conformal field theory arising from a certain distinguished K3 sigma model.

We present data necessary for the computation of all the Jacobi forms appearing in this work in \S\ref{sec:tables-comps}. We record coincidences between these Jacobi forms and other functions appearing in the context of K3 sigma models, and umbral moonshine, in \S\ref{sec:tables-coins}.

As mentioned earlier, we choose a square root of $-1$ in $\CC$ and denote it by $\ii$. We also set $e(x):=e^{2\pi \ii x}$.

\section{Lattices}\label{sec:lat}

An {\em integral lattice} is a free $\ZZ$-module of finite rank, $L\simeq \ZZ^n$, equipped with a symmetric bilinear form $\langle\cdot\,,\cdot\rangle:L\otimes_\ZZ L:\to \ZZ$. An excellent general reference for lattices is \cite{MR1662447}. Given a field $k$ of characteristic zero, the bilinear form $\langle\cdot\,,\cdot\rangle$ extends naturally to the $n$-dimensional vector space $L\otimes_\ZZ k$ over $k$. The {\em signature} of $L$ is the pair $(r,s)$ where $r$ is the maximal dimension of a positive-definite subspace of $L\otimes_\ZZ\RR$, and $s$ is the maximal dimension of a negative-definite subspace of $L\otimes_\ZZ\RR$. Call $n$ the {\em rank} of $L$, and say $L$ is {\em non-degenerate} if $n=r+s$. Say $L$ is {\em positive-definite} if $s=0$, and {\em negative-definite} if $r=0$. Say $L$ is {\em indefinite} if $rs\neq 0$.

Define the {\em dual} of $L$ by setting
\begin{gather}
	L^*:=\left\{\gamma\in L\otimes_\ZZ\QQ\mid \langle\lambda,\gamma\rangle\in\ZZ,\text{ for all $\lambda\in L$}\right\}.
\end{gather}
Certainly $L^*$ contains $L$. Say that $L$ is {\em self-dual} if $L^*=L$. Observe that a self-dual lattice is necessarily non-degenerate.

Given $\lambda\in L$ call $\langle\lambda,\lambda\rangle$ the {\em square-length} of $\lambda$. 
A lattice $L$ is called {\em even} if all of its square-lengths are even integers. The set of vectors of square-length $\pm 2$ in an even lattice is called its {\em root system}.

Write $\RR^{r,s}$ for a real vector space of dimension $n=r+s$, with elements denoted $x=(x_1,\ldots,x_n)$, equipped with the bilinear form
\begin{gather}
	\langle x,y\rangle=\sum_{i=1}^r x_iy_i - \sum_{i=r+1}^{n} x_iy_i.
\end{gather}
Write $I_{r,s}$ for the lattice in $\RR^{r,s}$ composed of vectors $x=(x_i)$ with integer coordinates. Observe that $I_{r,s}$ is self-dual. One often writes $\ZZ^n$ for $I_{n,0}$.
Assuming $r=s\mod 8$ define an even self-dual lattice $\II_{r,s}$ in $\RR^{r,s}$ by setting
\begin{gather}\label{eqn:lat-IIrsexp}
	\II_{r,s}:=
		\left\{
	x\in \ZZ^{n}\cup(\ZZ+\tfrac12)^n\subset\RR^{r,s}\mid \sum_{i=1}^rx_i=\sum_{i=r+1}^{n}x_i\mod 2
		\right\}.
\end{gather}
(Cf. e.g. \cite{MR640949}.) Taking $r=8$ and $s=0$ we obtain the $E_8$ {\em root lattice}, commonly denoted $E_8$. It is the unique (up to isomorphism) even self-dual lattice of signature $(8,0)$. The lattice $\II_{1,1}$ is often denoted $U$, and sometimes called the {\em hyperbolic plane}.

We refer to Chapter V of \cite{Ser_CrsArth} for a proof of the following fundamental result. 
\begin{theorem}\label{thm:lat-indevsduniq}
Suppose that $L$ is a non-even self-dual integral lattice with signature $(r,s)$. If $rs\neq 0$ then $L\simeq I_{r,s}$.
Suppose that $L$ is an even self-dual lattice with signature $(r,s)$. Then $r=s\mod 8$. If $rs\neq 0$ then $L\simeq \II_{r,s}$. 
\end{theorem}

Note that the right hand side of (\ref{eqn:lat-IIrsexp}) defines an integral self-dual lattice so long as $r+s=0\mod 4$. The lattice of rank $n=0\mod 4$ obtained by taking $r=n$ and $s=0$ in (\ref{eqn:lat-IIrsexp}) is called the {\em spin lattice} of rank $n$, and we denote it $D_n^+$. The {\em $D_n$ root lattice} is the intersection $I_{n,0}\cap \II_{n,0}$ (for any positive $n$), and is an even lattice of index $2$ in $D_n^+$. 
\begin{gather}\label{eqn:lat-Dn}
	D_n:=
			\left\{
	x\in \ZZ^{n}\subset\RR^{n,0}\mid \sum_{i=1}^nx_i=0\mod 2
		\right\}.
\end{gather}

We have isomorphisms $D_4^+\simeq \ZZ^4$ and $D_8^+\simeq E_8$, and $D_{12}^+$ is the unique (up to isomorphism) self-dual lattice of signature $(12,0)$ having no vectors with square-length $1$. The lattices $D_{16}^+$ and $E_8^{\oplus 2}$ are the only even self-dual lattices of signature $(16,0)$.

According to Theorem \ref{thm:lat-indevsduniq} we have 
\begin{gather}\label{eqn:lat-LUII}
	L\oplus U\simeq \II_{25,1}
\end{gather} 
for $L=E_8^{\oplus 3}$. But there are in fact $24$ choices for $L$ (up to isomorphism) that solve (\ref{eqn:lat-LUII}), according to Niemeier's classification \cite{Nie_DefQdtFrm24} of even self-dual definite lattices of rank $24$. (Cf. also \cite{MR558941} and Chapter 16 of \cite{MR1662447}.) Distinguished amongst these is the {\em Leech lattice}, named for its discoverer (cf. \cite{Lee_SphPkgHgrSpc,Lee_SphPkgs}) and denoted here by $\Lambda$, which is the unique even self-dual lattice of signature $(24,0)$ with an empty root system (i.e. no vectors of square-length $2$).

The uniqueness of the Leech lattice was proven by Conway \cite{Con_ChrLeeLat}. Conway also investigated its automorphism group \cite{MR0237634,MR0248216} and discovered three new sporadic simple groups in the course of this, $\Co_1$, $\Co_2$ and $\Co_3$. Define the {\em Conway group} by setting
\begin{gather}\label{eqn:lat-Co0}
	\Co_0:=\Aut(\LL).
\end{gather} 
Then $\Co_0$ is not simple, for its center is non-trivial, generated by $-\Id$. But the quotient group
\begin{gather}\label{eqn:lat-Co1}
\Co_1:=\Co_0/\{\pm \Id\}
\end{gather}
is simple, and is the largest sporadic simple Conway group. The groups $\Co_2$ and $\Co_3$ may be realized as the stabilizers in $\Co_0$ of vectors in $\LL$ with square-length equal to $4$ or $6$, respectively.

Given a lattice $L$ with signature $(r,s)$, write $L(-1)$ for the lattice of signature $(s,r)$ obtained by multiplying the bilinear form on $L$ by $-1$. Then, for example, if $L$ is even self-dual with signature $(k,16+k)$ for some positive integer $k$, we have
\begin{gather}\label{eqn:lat-II319}
	L\simeq \II_{k,16+k}\simeq E_8(-1)^{\oplus 2}\oplus U^{\oplus k},
\end{gather}
as a consequence of Theorem \ref{thm:lat-indevsduniq}.

We call $\Lambda(-1)$ the {\em negative-definite Leech lattice}.

An {\em embedding} of lattices $K\to L$ is an embedding of abelian groups $\iota:K\to L$ such that $\langle \lambda,\mu\rangle_K=\langle\iota(\lambda),\iota(\mu)\rangle_L$ for $\lambda,\mu\in K$. A {\em primitive embedding} is an embedding $\iota:K\to L$ such that the quotient group $L/\iota(K)$ is torsion-free.

\section{Modular Forms}\label{sec:autfms}

Here we recall some basic facts about modular forms and Jacobi forms. 
For $\tau\in\HH$ and $z\in\CC$ we use the notation $q:=e(\tau)$ and $y:=e(z)$, where $e(x):=e^{2\pi \ii x}$.

Recall (cf. e.g. \cite{MR1193029,Shimura}) that a holomorphic function $f:\HH\to \CC$ is a called an {\em unrestricted modular form} of weight $k$ for a group $\Gamma<\SL_2(\RR)$ if 
\begin{gather}\label{eqn:autfms-unrestmodfm}
	f\left(\frac{a\tau+b}{c\tau+d}\right)\frac{1}{(c\tau+d)^k}=f(\tau)
\end{gather}
for all $\left(\begin{smallmatrix}a&b\\c&d\end{smallmatrix}\right)\in \Gamma$. Assume for simplicity that $\Gamma$ is commensurable with $\SL_2(\ZZ)$. Then the action of $\Gamma$ on $\HH$ extends naturally to $\wh{\QQ}:=\QQ\cup\{\infty\}$, and the orbits of $\Gamma$ on $\wh{\QQ}$ are called its {\em cusps}. The orbit containing $\infty$ is called the {\em infinite cusp} of $\Gamma$, and the {\em modular group} $\SL_2(\ZZ)$ has only the infinite cusp.

Say that $f$ as in (\ref{eqn:autfms-unrestmodfm}) is a {\em weakly holomorphic modular form} if it has at most exponential growth at cusps. This amounts to the condition that if $\sigma\in \SL_2(\ZZ)$ then $f(\sigma\tau)$ admits a Laurent expansion in $q^{1/w}$, for some positive integer $w$. If the Laurent expansions of the $f(\sigma\tau)$ are actually Taylor series in $q^{1/w}$, so that $f(\sigma\tau)=O(1)$ as $\Im(\tau)\to \infty$, then we say that $f$ is a {\em modular form}. A {\em cusp form} satisfies $f(\sigma\tau)\to 0$ as $\Im(\tau)\to \infty$, for every $\sigma\in \SL_2(\ZZ)$.

Write $M_{k}(\Gamma)$ for the space of modular forms of weight $k$ for $\Gamma$. Write $S_k(\Gamma)$ for the subspace of cusp forms.

The \emph{Eisenstein series} $E_k$ are a family of modular forms for the full modular group $SL_2(\ZZ)$. For $k$ even and greater than 2, the Eisenstein series $E_k$ is
defined by
\be\label{equation:eisenstein series}
E_k(\tau):=\sum_{\substack{m,n\in\ZZ \\ (m,n)\neq(0,0)}}\frac{1}{(m\tau+n)^k},
\ee
and admits a Fourier expansion
\be\label{equation:eisenstein q-series}
E_k(\tau)=1+\frac{2}{\zeta(1-k)}\sum_{n>0}\sigma_{k-1}(n)q^n,
\ee
where $\sigma_s(n):=\sum_{d\mid n} d^s$. When $k=2$ the series in (\ref{equation:eisenstein series}) is not absolutely convergent, but converges conditionally to (\ref{equation:eisenstein q-series}), and we define $E_2$ by (\ref{equation:eisenstein q-series}). The conditional convergence prevents $E_2$ from being a modular form, but it is a {\em quasi-modular form}, satisfying 
\begin{gather}\label{eqn:autfms-E2xfm}
	E_2\left(\frac{a\tau+b}{c\tau+d}\right)\frac{1}{(c\tau+d)^2}+\frac{3}{\pi^2}\frac{2\pi \ii c}{(c\tau+d)}=E_2(\tau)
\end{gather}
for $\left(\begin{smallmatrix}a&b\\c&d\end{smallmatrix}\right)\in \SL_2(\ZZ)$. (Cf. e.g. Proposition 6 of \cite{MR2409678}.)

We make extensive use of the {\em Dedekind eta function}, a modular form of weight $\frac{1}{2}$ (with non-trivial multiplier), defined by
\be\label{eqn:autfms-eta}
\eta(\tau):=q^{1/24}\prod_{n>0}(1-q^n).
\ee
Related is Ramanujan's {\em Delta function}, a cusp form of weight 12 for $\SL_2(\ZZ)$, defined by
\be\label{eqn:autfms-Delta}
\Delta(\tau):=\eta(\tau)^{24}.
\ee

In this work, a modular form with {\em level} is a modular form for some $\Gamma_0(N)$,
\be
\Gamma_0(N):=\left\{\begin{pmatrix} a & b \\ c & d\end{pmatrix} \in SL_2(\ZZ)\mid c=0\mod N\right\}.
\ee
Note that if $f(\tau)$ is a modular form with level $N$ (i.e. a modular form for $\Gamma_0(N)$), then $f(h\tau)$ is a modular form with level $hN$.

The functions $\Lambda_N$ are a family of modular forms of weight 2 with level, defined by
\be\label{eqn:autfms-LambdaN}
\Lambda_N(\tau):=\frac{N}{2\pi i}
\frac{{\rm d}}{{\rm d}\tau}\log\left(\frac{\eta(N\tau)}{\eta(\tau)}\right)=\frac{N}{24}\left(NE_2(N\tau)-E_2(\tau)\right).
\ee
One easily checks using (\ref{eqn:autfms-E2xfm}) that $\Lambda_N\in M_2(\Gamma_0(N))$.

For later use we note the basic identities,
\begin{gather}
	\Lambda_4(\tau)=4\Lambda_2(2\tau)+2\Lambda_2(\tau),\quad
	\Lambda_4(\tau+1/2)=8\Lambda_2(2\tau)-2\Lambda_2(\tau).
\end{gather}
In particular, $\Lambda_4(\tau+1/2)$ is also a modular form for $\Gamma_0(4)$. Really, this is unsurprising because $\Gamma_0(4)$ is normalized by the matrix $\left(\begin{smallmatrix}1&1/2\\0&1\end{smallmatrix}\right)$. More generally, we have the following beautiful description of the full normalizer of $\Gamma_0(N)$ from \cite{MR554399}.

Given a positive integer $N$, let $h$ denote the largest divisor of $24$ such that $h^2$ divides $N$. Set $n=N/h$. Then the normalizer of $\Gamma_0(N)$ in $\SL_2(\RR)$ is composed of the matrices 
\begin{gather}\label{eqn:autfms-adebcnh}
	\frac{1}{\sqrt{e}}
	\begin{pmatrix}
	ae&b/h\\cn&de
	\end{pmatrix}
\end{gather}
where $e$ is an {\em exact} divisor of $n/h$ (i.e. $e|(n/h)$ and $(e,n/eh)=1$), and $a,b,c,d\in \ZZ$ are chosen so that $ade^2-bcn/h=e$.

We write $\Gamma_0(n|h)$ for the set of matrices (\ref{eqn:autfms-adebcnh}) with $e=1$. It is a subgroup of $\SL_2(\RR)$ that is conjugate to $\Gamma_0(n/h)$. For a fixed non-trivial exact divisor $e|(n/h)$, the matrices (\ref{eqn:autfms-adebcnh}) comprise an {\em Atkin--Lehner involution} of $\Gamma_0(n|h)$. (Really, an Atkin--Lehner involution is a coset of $\Gamma_0(n|h)$ in its normalizer, an involution in the sense that it defines an order $2$ element of the quotient group $N(\Gamma_0(n|h))/\Gamma_0(n|h)$.)

Assume now that $\Gamma$ is a subgroup of $\SL_2(\ZZ)$. We call a holomorphic function $\phi:\HH\times\CC\to \CC$ an {\em unrestricted Jacobi form} of weight $k$ and index $m$ for $\Gamma$ if it satisfies
\begin{gather}
	\phi\left(\frac{a\tau+b}{c\tau+d},\frac{z}{c\tau+d}\right)\frac{1}{(c\tau+d)^k}e\left(-m\frac{cz^2}{c\tau+d}\right)=\phi(\tau,z),
	\label{eqn:autfms-jacxfmmod}
	\\
	\phi(\tau,z+\lambda \tau+\mu)q^{m\lambda^2}y^{2m\lambda}=\phi(\tau,z),
	\label{eqn:autfms-jacxfmell}
\end{gather}
for $\left(\begin{smallmatrix}a&b\\c&d\end{smallmatrix}\right)\in\Gamma$ and $(\lambda,\mu)\in\ZZ^2$, where $q=e(\tau)$ and $y=e(z)$. For $\phi$ an unrestricted Jacobi form and $\sigma\in\SL_2(\ZZ)$ we have
\begin{gather}	\label{eqn:autfms-phicnr}
	\phi(\sigma\tau,z)=\sum_{n,r\in\ZZ}c_\sigma(n/w,r)q^{n/w}y^r
\end{gather}
for some $c_\sigma(n,r)\in\CC$. Say that $\phi$ is a {\em weak Jacobi form} if $c_\sigma(n/w,r)=0$ whenever $n/w<0$, for all $\sigma\in \SL_2(\ZZ)$. Note that $c_\sigma(n/w,r)$ differs from $c_{\sigma'}(n/w,r)$ only by a root of unity when $\sigma'\sigma^{-1}\in\Gamma$. So it suffices to check the $c_\sigma(n/w,r)$ for just one representative $\sigma$ of each right coset of $\Gamma$ in $\SL_2(\ZZ)$.

Good references for Jacobi forms include \cite{Dabholkar:2012nd} and \cite{eichler_zagier}. Jacobi forms occur naturally as Fourier coefficients of Siegel Modular forms (cf. \cite{eichler_zagier,feingold_frenkel}), but that manifestation will not play an explicit role here.

In this work, a weak Jacobi form with {\em level} $N$ is a weak Jacobi form for $\Gamma_0(N)$.

Particularly useful for writing Jacobi forms down explicitly are the four {\em Jacobi theta functions}, defined as
\begin{align}
\begin{split}\label{equation:theta1}
\vartheta_1(\tau,z) & :=-\ii\sum_{n\in \ZZ}(-1)^n y^{n+1/2} q^{(n+1/2)^2/2}, \\
\vartheta_2(\tau,z) & :=\sum_{n\in\ZZ}y^{n+1/2}q^{(n+1/2)^2/2}, \\
\vartheta_3(\tau,z) & :=\sum_{n\in\ZZ}y^n q^{n^2/2}, \\
\vartheta_4(\tau,z) & :=\sum_{n\in\ZZ}(-1)^n y^n q^{n^2/2}, \\
\end{split}
\end{align}
and admitting the product formulas
\begin{align}
\begin{split}\label{eqn:autfms-theta1prod}
\vartheta_1(\tau,z) 
& =-\ii q^{1/8}y^{1/2}(1-y^{-1})\prod_{n>0}(1-y^{-1}q^n)(1-yq^n)(1-q^n), \\
\vartheta_2(\tau,z)
& =q^{1/8}y^{1/2}(1+y^{-1})\prod_{n>0}(1+y^{-1}q^n)(1+yq^n)(1-q^n), \\
\vartheta_3(\tau,z) 
& =\prod_{n>0}(1+y^{-1}q^{n-1/2})(1+y q^{n-1/2})(1-q^n), \\
\vartheta_4(\tau,z)
& =\prod_{n>0}(1-y^{-1}q^{n-1/2})(1-yq^{n-1/2})(1-q^n), \\
\end{split}
\end{align}
according to the Jacobi triple product identity.

The first examples of weak Jacobi forms with level $1$ are $\phi_{0,1}$ and $\phi_{-2,1}$, defined by
\be\label{eqn:autfms:phi01}
\phi_{0,1}(\tau,z):=4\left(\frac{\vartheta_2(\tau,z)^2}{\vartheta_2(\tau,0)^2}+\frac{\vartheta_3(\tau,z)^2}{\vartheta_3(\tau,0)^2}+\frac{\vartheta_4(\tau,z)^2}{\vartheta_4(\tau,0)^2}\right)
\ee
and
\be\label{eqn:autfms:phi-21}
\phi_{-2,1}(\tau,z):=-\frac{\vartheta_1(\tau,z)^2}{\eta(\tau)^6}.
\ee
The subscripts indicate weight and index, respectively. Note that $\phi_{0,1}(\tau,0)=12$ and $\phi_{-2,1}(\tau,0)=0$, which is consistent with the facts that all modular forms of weight $0$ are constant, and all modular forms of negative weight are $0$.

Proposition 6.1 of \cite{Prop6.1} states that any weak Jacobi form of even weight can be written as a polynomial in $\phi_{0,1}$ and $\phi_{-2,1}$ with modular form coefficients. We record the following special cases of this for use later on.
\begin{proposition}\label{prop:autfms-phiCFs}
A holomorphic function $\phi:\HH\times\CC\to \CC$ is a weak Jacobi form of weight $0$ and index $m$ for $\Gamma$ if and only if there exists $C\in \CC$ and modular forms $F_{2j}\in M_{2j}(\Gamma)$, for $1\leq j\leq m$, such that
\begin{gather}\label{eqn:autfms-phiCFs}
	\phi(\tau,z)=C\phi_{0,1}(\tau,z)^m+\sum_{j=1}^m F_{2j}(\tau)\phi_{-2,1}(\tau,z)^j\phi_{0,1}(\tau,z)^{m-j}.
\end{gather}
\end{proposition}

We conclude this section with formulas that illustrate (\ref{eqn:autfms-phiCFs}) explicitly for the particular combinations of Jacobi theta functions appearing in (\ref{eqn:autfms:phi01}) and (\ref{eqn:autfms:phi-21}).
\begin{lemma}\label{lemma:theta function identities}
We have the following identities.
\begin{gather}\label{equation:theta1 identity}
	\phi_{-2,1}(\tau,z)
	=
	-\frac{\vartheta_1(\tau,z)^2}{\eta(\tau)^6}
\end{gather}
\be\label{equation:theta2 identity}
\frac{1}{12}\phi_{0,1}(\tau,z)+2\Lambda_2(\tau)\phi_{-2,1}(\tau,z)=\frac{\vartheta_2(\tau,z)^2}{\vartheta_2(\tau,0)^2}
\ee
\be\label{equation:theta3 identity}
\frac{1}{12}\phi_{0,1}(\tau,z)-\Lambda_2(\tau/2+1/2)\phi_{-2,1}(\tau,z)=\frac{\vartheta_3(\tau,z)^2}{\vartheta_3(\tau,0)^2}
\ee
\be\label{equation:theta4 identity}
\frac{1}{12}\phi_{0,1}(\tau,z)-\Lambda_2(\tau/2)\phi_{-2,1}(\tau,z)=\frac{\vartheta_4(\tau,z)^2}{\vartheta_4(\tau,0)^2}
\ee
\end{lemma}
Note that the first identity of Lemma \ref{lemma:theta function identities} is just the definition of $\phi_{-2,1}$ (cf. (\ref{eqn:autfms:phi-21})).

\begin{proof}
Let $\Gamma(N)$ denote the {\em principal congruence group of level $N$}, being the kernel of the natural map $\SL_2(\ZZ)\to \SL_2(\ZZ/N\ZZ)$.
First we will show that 
\begin{gather}\label{eqn:specfun-jactht2quot}
\frac{\vartheta_i(\tau,z)^2}{\vartheta_i(\tau,0)^2}
\end{gather} 
is a Jacobi form for $\Gamma(2)$ of weight 0 and index 1, for $i\in \{2,3,4\}$. The transformations
\be
\frac{\vartheta_2(\tau,z+1)^2}{\vartheta_2(\tau,0)^2}=\frac{\vartheta_2(\tau,z)^2}{\vartheta_2(\tau,0)^2}\quad\text{and}\quad\frac{\vartheta_2(\tau,z+\tau)^2}{\vartheta_2(\tau,0)^2}=q^{-1}y^{-2}\frac{\vartheta_2(\tau,z)^2}{\vartheta_2(\tau,0)^2}
\ee
can be seen by explicit computation using (\ref{equation:theta1}). Thus (\ref{eqn:specfun-jactht2quot}) transforms properly under $\ZZ^2$ in the case that $i=2$. If $S$ and $T$ are the standard generators for the modular group, then $\Gamma(2)$ is generated by $T^2:\tau\mapsto\tau+2$ and $ST^2S:\tau\mapsto\frac{-\tau}{2\tau-1}$ (see \S 6 of \cite{Frasch}), so the required transformations under $\Gamma(2)$ are
\be
\frac{\vartheta_2(\tau+2,z)^2}{\vartheta_2(\tau+2,0)^2}=\frac{\vartheta_2(\tau,z)^2}{\vartheta_2(\tau,0)^2}\quad\text{and}\quad\frac{\vartheta_2(\frac{-\tau}{2\tau-1},\frac{z}{2\tau-1})^2}{\vartheta_2(\frac{-\tau}{2\tau-1},0)^2}=e^{\frac{4\pi \ii z^2}{2\tau-1}}\frac{\vartheta_2(\tau,z)^2}{\vartheta_2(\tau,0)^2}.
\ee
The first can be seen explicitly from (\ref{equation:theta1}), and the second follows from acting successively with $T$ and $S$, using Jacobi's imaginary transformations
\be
\vartheta_2\left(-\frac{1}{\tau},\frac{z}{\tau}\right)=\sqrt{-\ii\tau}\,e^{\pi \ii z^2/\tau}\vartheta_4(\tau,z)\quad\text{and}\quad\vartheta_4\left(-\frac{1}{\tau},\frac{z}{\tau}\right)=\sqrt{-\ii\tau}\,e^{\pi \ii z^2/\tau}\vartheta_2(\tau,z).
\ee
The required form for the Fourier expansion can also be seen from (\ref{equation:theta1}). Thus (\ref{eqn:specfun-jactht2quot}) is a Jacobi form for $\Gamma(2)$ of weight 0 and index 1 for $i=2$.

Similar arguments handle the cases that $i=3$ and $i=4$. For $i=3$ we use another of Jacobi's imaginary transformations, 
\be
\vartheta_3\left(-\frac{1}{\tau},\frac{z}{\tau}\right)=\sqrt{-\ii\tau}\,e^{\pi \ii z^2/\tau}\vartheta_3(\tau,z).
\ee

As recorded in \cite{CP}, the group $\Gamma(2)$ is a genus $0$ congruence subgroup of $\SL_2(\ZZ)$ with 3 inequivalent cusps. By Theorem 2.23 of \cite{Shimura}, the dimension of the space of weight 2 modular forms on $\Gamma(2)$ is 2, and the dimension of the space of weight 0 modular forms on $\Gamma(2)$ is 1 (i.e. spanned by a constant function). Proposition 6.1 of \cite{Prop6.1} shows that the dimension of the space of weight 0, index 1 Jacobi forms on $\Gamma(2)$ must therefore be 3. Now the proof of the required identities is reduced to checking the agreement of Fourier coefficients up to $O(q)$.
\end{proof}


\section{Derived Equivalences}\label{sec:de}

Let $X$ be a complex K3 surface; i.e., a compact connected complex manifold of dimension $2$ with $\Omega_X^2\simeq \mathcal{O}_X$ and $H^1(X,\mathcal{O}_X)=0$. (Good references for K3 surfaces include \cite{MR2030225} and \cite{MR785216}.) Then the intersection form $(\;.\;)$ equips the integral singular cohomology group $H^2(X,\ZZ)$ with the structure of an even self-dual lattice of signature $(3,19)$, so we have $H^2(X,\ZZ)\simeq E_8(-1)^{\oplus 2}\oplus U^{\oplus 3}$ according to (\ref{eqn:lat-II319}).

Write $\mlz=(\mlz,\lab\;,\;\rab)$ for the {\em Mukai lattice} of $X$, being the lattice obtained from 
\begin{gather}
H^*(X,\ZZ)=H^0(X,\ZZ)\oplus H^2(X,\ZZ)\oplus H^4(X,\ZZ)
\end{gather} 
by reversing the sign of the pairings between $H^0(X,\ZZ)$ and $H^4(X,\ZZ)$, so that
\begin{gather}
\lab \lambda,\mu\rab=(\lambda_2.\mu_2)-(\lambda_0.\mu_4)-(\lambda_4.\mu_0)
\end{gather}	
for $\lambda=\lambda_{0}+\lambda_2+\lambda_4\in\mlz$ with $\lambda_k\in H^k(X,\ZZ)$, \&c. Then $\mlz$ is self-dual and even with signature $(4,20)$.
\begin{gather}\label{eqn:de:mlz}
\mlz\simeq E_8(-1)^{\oplus 2}\oplus U^{\oplus 4}
\end{gather}

A {\em Hodge structure} of weight $2$ on a lattice $L$ is a direct sum decomposition 
\begin{gather}
L\otimes_\ZZ{\CC}=L^{0,2}\oplus L^{1,1}\oplus L^{2,0}
\end{gather}
of the complex vector space enveloping $L$ into complex subspaces $L^{p,q}<L\otimes_\ZZ{\CC}$ such that the $\RR$-linear complex conjugation $v\mapsto \bar{v}$ on $L\otimes_\ZZ{\CC}$ that fixes the subset $L\otimes_\ZZ{\RR}$ induces $\RR$-linear isomorphisms $L^{p,q}\simeq L^{q,p}$. 

If $X$ is a complex K3 surface then we naturally obtain a weight $2$ Hodge structure 
\begin{gather}\label{eqn:de:mlzhod}
\mlc=\wt{H}^{2,0}(X)\oplus \wt{H}^{1,1}(X)\oplus \wt{H}^{0,2}(X)
\end{gather}
on the Mukai lattice of $X$, by setting 
\begin{gather}
\begin{split}\label{eqn:de:mlzhoddef}
\wt{H}^{2,0}(X)&:=H^{2,0}(X),\\
\wt{H}^{1,1}(X)&:=H^{0,0}(X)\oplus H^{1,1}(X)\oplus H^{2,2}(X),\\
\wt{H}^{0,2}(X)&:=H^{0,2}(X),
\end{split}
\end{gather}
where the $H^{p,q}(X)=H^{p,q}(X,\CC)$ are the Dolbeaut cohomology groups of $X$. Say that an automorphism $g$ of the lattice $\mlz$ is a {\em symplectic Hodge isometry} of $\mlz$ if the $\CC$-linear extension of $g$ to $\mlz\otimes_{\ZZ}\CC$ fixes $\wt{H}^{2,0}(X)$ (and hence also $\wt{H}^{0,2}(X)$) pointwise. Note that $\wt{H}^{2,0}(X)$ and $\wt{H}^{0,2}(X)$ are isotropic with respect to the bilinear form on $\mlc$ induced from $\mlz$. The intersection 
\begin{gather}\label{eqn:de:PX}
	P_X:=(\wt{H}^{2,0}(X)\oplus \wt{H}^{0,2}(X))\cap\mlr
\end{gather}
is a positive-definite $2$-dimensional subspace of $\mlr$.

Following \cite{2013arXiv1309.6528H} we write $\Aut_s(\mlz)$ for the group of symplectic Hodge isometries of the Mukai lattice $\mlz$ of a complex K3 surface $X$. Note that any symplectic automorphism of $X$ of finite order naturally induces a symplectic Hodge isometry of $\mlz$, via the induced action of $g$ on $H^*(X,\ZZ)$ (cf. \S1.2 of \cite{2013arXiv1309.6528H}), but for general $X$ not all symplectic Hodge isometries arise in this way (cf. \S1.4 of \cite{2013arXiv1309.6528H}).

Assume now that $X$ is projective, admitting an embedding in some complex projective space $\PP^n$. Write $\Db(X)$ for the bounded derived category of coherent sheaves on $X$ and let $\Aut(\Db(X))$ denote the group of isomorphism classes of exact $\CC$-linear autoequivalences of $\Db(X)$. (See \cite{MR2511017,MR2244106} for detailed expositions of this theory.) The induced action of an exact autoequivalence of $\Db(X)$ on $\mlz$---cf. the discussion in \S1.2 in \cite{MR2376815}---defines a morphism of groups from $\Aut(\Db(X))$ to the automorphism group (i.e. orthogonal group) of $\mlz$, and we write $\Aut_s(\Db(X))$ for the subgroup of {\em symplectic} autoequivalences, being those elements of $\Aut(\Db(X))$ that map to symplectic Hodge isometries of $\mlz$. 
\begin{gather}\label{eqn:de:autmor}
	\Aut_s(\Db(X))\to\Aut_s(\mlz)
\end{gather}

Let $\Stab(X)$ denote the space of stability conditions on $\Db(X)$. (See \cite{2011arXiv1111.1745H} for a nice introduction to stability conditions.) Write $\Stabo(X)$ for the distinguished connected component of $\Stab(X)$ introduced and first analyzed by Bridgeland in \cite{MR2376815}. Given $\sigma\in \Stabo(X)$ say that an autoequivalence in $\Aut_s(\Db(X))$ is {\em $\sigma$-positive} if its induced action on $\Stab(X)$ fixes $\sigma$, and write $\Aut_s(\Db(X),\sigma)$ for the group of all (isomorphism classes of) $\sigma$-positive exact $\CC$-linear autoequivalences of $\Db(X)$. 

Set $\mlzoo:=\wt{H}^{1,1}(X)\cap \mlz$ (cf. (\ref{eqn:de:mlzhoddef})). To each $\sigma\in \Stab(X)$ is attached a {\em central charge} $Z$, which may be regarded as a morphism of groups $\mlzoo\to\CC$, or equivalently, via Poincar\'e duality, as an element of $\mlcoo$. According to \S1.3 of \cite{2013arXiv1309.6528H}, the real subspace of $\mlroo$ spanned by the real and imaginary parts of $Z$,
\begin{gather}\label{eqn:de:PZ}
	P_Z:=\RR\Re(Z)\oplus \RR\Im(Z)<\mlroo,
\end{gather}
is positive with respect to the induced bilinear form from $\mlz$, when $Z$ is the central charge of a stability condition in $\Stabo(X)$. For such a $Z\in\mlcoo$ define $\Aut_s(\mlz,Z)$ to be the subgroup of symplectic Hodge isometries of $\mlz$ whose $\RR$-linear extensions to $\mlr$ fix the subspace $P_Z$ pointwise; such an isometry is called {\em $P_Z$-positive}. 

Recall the natural map (\ref{eqn:de:autmor}). Huybrechts has shown that this map induces an isomorphism between the group of $\sigma$-positive symplectic autoequivalences of $\Db(X)$ and the group of $P_Z$-positive symplectic Hodge isometries of $\mlz$ when $Z$ is the central charge of a stability condition $\sigma$ in $\Stabo(X)$.
\begin{proposition}[\cite{2013arXiv1309.6528H}]\label{prop:de:aetohi}
Let $X$ be a projective complex K3 surface, let $\sigma\in \Stabo(X)$ and let $Z$ be the central charge of $\sigma$. Then the natural map $\Aut_s(\Db(X))\to\Aut_s(\mlz)$ induces an isomorphism of groups
\begin{gather}\label{eqn:de:aetohi}
\Aut_s(\Db(X),\sigma)\xrightarrow{\sim}\Aut_s(\mlz,Z).
\end{gather}
\end{proposition}

We have mentioned that $P_X$ and $P_Z$ are positive-definite $2$-dimensional subspaces of $\mlr$. They are orthogonal (compare (\ref{eqn:de:PX}) with (\ref{eqn:de:PZ})), and $\mlz$ has signature $(4,20)$ (cf. (\ref{eqn:de:mlz})), so 
\begin{gather}\label{eqn:de-Pi}
\Pi:=P_X\oplus P_Z
\end{gather}
is a maximal positive-definite subspace of $\mlr$. As explained in \cite{2013arXiv1309.6528H}, the intersection $\Pi^{\perp}\cap \mlz$ contains no vectors $\delta$ with $\lab\delta,\delta\rab=-2$. For on the one hand, if $\lab\delta,\delta\rab=-2$ and $\delta\in P_X^\perp$, then $\delta\in\mlzoo$. On the other hand, it is proven in Proposition 13.2 of \cite{MR2376815} that if $\delta\in \mlzoo$ and $\lab\delta,\delta\rab=-2$, then $\lab Z,\delta\rab\neq 0$.

Note that any positive-definite $4$-dimensional subspace of $\mlr$, and in particular $\Pi=P_X\oplus P_Z$, is naturally oriented. For if $\varsigma$ is a non-zero element of $H^{2,0}(X)$ then the $4$-tuple $(\Re(\varsigma),\Im(\varsigma),\Re(Z),\Im(Z))$ defines an oriented basis, and the resulting orientation depends neither on $\varsigma$ nor $Z$ (cf. \S4.5 of \cite{MR2553878}). 

Given $X$ and $\sigma\in\Stabo(X)$ as above, define
\begin{gather}\label{eqn:de-GPi}
	G_\Pi:=\Aut_s(\Db(X),\sigma),
\end{gather}
and use the natural isomorphism (\ref{eqn:de:aetohi}) to identify $G_\Pi$ with $\Aut_s(\mlz,Z)$. Also define $\Gamma_\Pi$ to be the sublattice of $\mlz$ composed of vectors orthogonal to the sublattice of $\mlz$ that is fixed by $G_\Pi$, so that
\begin{gather}\label{eqn:de-GammaPi}
	\Gamma_\Pi:=\left(\mlz^{G_\Pi}\right)^\perp\cap\mlz.
\end{gather} 
Then $\Gamma_\Pi$ is an even, negative-definite lattice of rank at most $20$, naturally admitting a faithful action by $G_\Pi$. Moreover, according to the argument in \S\S B.1-2 of \cite{GHV} (cf. also \S2.2 of \cite{2013arXiv1309.6528H}), the lattice $\Gamma_\Pi$ admits a primitive embedding 
\begin{gather}\label{eqn:de-iota}
\iota:\Gamma_\Pi\to \LL(-1)
\end{gather}
in the negative-definite Leech lattice (cf. \S\ref{sec:lat}), and the action of $G_\Pi$ on $\Gamma_\Pi$ extends naturally to $\LL(-1)$, in such a way that all vectors in $\iota(\Gamma_\Pi)^\perp\cap\LL(-1)$ are fixed by $G_\Pi$. Thus the primitive embedding $\iota$ of (\ref{eqn:de-iota}) determines an embedding of groups, 
\begin{gather}\label{eqn:de-iotastar}
	\iota_*:G_\Pi\to \Aut(\LL), 
\end{gather}
which we may use to identify $G_\Pi$ with a subgroup of the Conway group, $\Co_0=\Aut(\LL)$ (cf. (\ref{eqn:lat-Co0})). The sublattice of $\LL$ fixed by this copy of $G_\Pi$ has rank at least $4$. 

Call a primitive embedding as in (\ref{eqn:de-iota}) a {\em Leech marking} of the data $(X,\sigma)$. We may summarize the previous paragraph by saying that the group $G_\Pi=\Aut_s(\Db(X),\sigma)$ is isomorphic to a subgroup of $\Co_0$ that fixes a rank $4$ sublattice of $\LL$, and the choice of Leech marking $\iota$ determines this subgroup completely. The main result of \cite{2013arXiv1309.6528H} states that the converse is also true.
\begin{theorem}[\cite{2013arXiv1309.6528H}]\label{thm:de:deco}
For $X$ a projective complex K3 surface and $\sigma\in \Stabo(X)$ the group $G_\Pi=\Aut_s(\Db(X),\sigma)$ is isomorphic to a subgroup of $\Co_0$ whose action on the Leech lattice fixes a sublattice of rank at least $4$. Conversely, if $G_*$ is a subgroup of $\Co_0$ that fixes a rank $4$ sublattice of the Leech lattice then there exists a projective complex K3 surface $X$, a stability condition $\sigma\in\Stabo(X)$, and a Leech marking $\iota$ for $(X,\sigma)$ such that $G_*$ is a subgroup of $\iota_*G_\Pi$.
\end{theorem}
Recall from \S\ref{sec:lat} that the center of $\Co_0$ is the group of order $2$ generated by $-\Id$, and $\Co_1$ denotes the sporadic simple quotient group $\Co_1=\Co_0/\{\pm\Id\}$. Observe that if $G_*$ is a subgroup of $\Co_0$ that has a fixed point in its action on $\Lambda$ then the natural map $\Co_0\to \Co_1$ induces an isomorphism between $G_*$ and its image in $\Co_1$. Thus one may replace $\Co_0$ with $\Co_1$ in the statement of Theorem \ref{thm:de:deco}.

\section{Vertex Algebra}\label{sec:va}

In this section we briefly recall super vertex operator algebras and their canonically-twisted modules. We refer to the texts \cite{MR2082709,MR1651389,MR2023933} for more background on vertex algebra.

A {\em super vector space} is a simply a vector space with a $\ZZ/2$-grading, $V=V_{\bar{0}}\oplus V_{\bar{1}}$. A linear operator $T:V\to V$ is called {\em even} if $T(V_{\bar{j}})\subset V_{\bar{j}}$, and {\em odd} if $T(V_{\bar{j}})\subset V_{\overline{j+1}}$. 

For $V$ a super vector space and $z$ a formal variable write $V((z)):=V[[z]][z^{-1}]$ for the space of Laurent series in $z$ with coefficients in $V$. Taking $z$ to be even we naturally obtain a super structure $V_{\bar 0}((z))\oplus V_{\bar 1}((z))$ on $V((z))$. Observe that the rational function $f(z,w)=(z-w)^{-1}$ naturally defines elements of $\CC((z))((w))$ and $\CC((w))((z))$ via formal power series expansions, for we have $f(z,w)=\sum_{n\geq 0} z^{-n-1}w^n$ in $\CC((z))((w))$ and $f(z,w)=-\sum_{n\geq 0}w^{-n-1}z^n$ in $\CC((w))((z))$. These rules extend naturally so as to define {\em formal expansion} maps
\begin{gather}
	\begin{split}\label{eqn:va-fmlxpn}
	V[[z,w]][z^{-1},w&^{-1},(z-w)^{-1}]\\
	\swarrow\quad\quad\quad\quad\downarrow&\quad\quad\quad\searrow\\
	V((z))((w))\quad\quad V((w))&((z))\quad\quad V((w))((z-w)).
	\end{split}
\end{gather}

A \emph{super vertex algebra} is a super vector space $V=V_{\bar{0}}\oplus V_{\bar{1}}$ equipped with a \emph{vacuum vector} $\mathbf{1} \in V_{\bar{0}}$, an even linear operator $T:V\to V$, and a linear map
\begin{align}
\begin{split}\label{eqn:va:funds-vopcorr}
V & \to\End(V)[[z^{\pm1}]] \\
a & \mapsto Y(a,z)=\sum_{n\in\ZZ} a_{(n)} z^{-n-1}
\end{split}
\end{align}
which associates to each $a\in V$ a \emph{vertex operator} $Y(a,z)$. This data should satisfy the following axioms for any $a,b, c\in V$.
\begin{enumerate}
\item $Y(a,z)b\in V((z))$ and if $a\in V_{\bar{0}}$ (resp. $a\in V_{\bar{1}}$) then $a_{(n)}$ is an even (resp. odd) operator for all $n$;
\item $Y(\mathbf{1},z)=\Id_V$ and $Y(a,z)\mathbf{1}\in a+zV[[z]]$;
\item $[T,Y(a,z)]=\partial_z Y(a,z)$ and $T\mathbf{1}=0$;
\item if $a\in V_{p(a)}$ and $b\in V_{p(b)}$ are $\ZZ/2$ homogenous, there exists an element
$$f\in V[[z,w]][z^{-1},w^{-1},(z-w)^{-1}]$$
depending on $a$, $b$, and $c$, such that
$$Y(a,z)Y(b,w)c,\quad (-1)^{p(a)p(b)}Y(b,w)Y(a,z)c,\quad\text{and}\quad Y(Y(a,z-w)b,w)c$$
are the formal expansions of $f$ in $V((z))((w))$, $V((w))((z))$, and $V((w))((z-w))$ respectively (cf. (\ref{eqn:va-fmlxpn})).
\end{enumerate}

For $V=V_{\bar{0}}\oplus V_{\bar{1}}$ a super vertex operator algebra let $\theta:V\to V$ denote the {\em parity involution}, acting as $(-1)^j$ on $V_{\bar{j}}$. A \emph{canonically-twisted module} for $V$ is a super vector space $M=M_{\bar{0}}\oplus M_{\bar{1}}$ equipped with a linear map
\begin{align}
\begin{split}\label{eqn:va-Ytw}
V & \to\End(M)[[z^{\pm{1/2}}]] \\
a & \mapsto Y_\tw(a,z^{1/2})=\sum_{n\in\frac{1}{2}\ZZ}a_{(n),\tw}z^{-n-1},
\end{split}
\end{align}
associating to each $a\in V$ a \emph{canonically-twisted vertex operator} $Y_\tw(a,z^{1/2})$, which satisfies the following axioms for any $a,b\in V, u\in M$:
\begin{enumerate}
\item $Y_\tw(a,z^{1/2})u\in M((z^{1/2}))$ and if $a\in V_{\bar{0}}$ (resp. $a\in V_{\bar{1}}$) then $a_{(n),\tw}$ is an even (resp. odd) operator for all $n$;
\item $Y_\tw(\mathbf{1},z^{1/2})=\Id_M$;
\item if $a\in V_{p(a)}$ and $b\in V_{p(b)}$, there exists an element
$$f\in M[[z^{1/2}, w^{1/2}]][z^{-{1/2}},w^{-{1/2}},(z-w)^{-1}]$$
depending on $a$, $b$, and $u$, such that
$$Y_\tw(a,z^{1/2})Y_\tw(b,w^{1/2})u, \quad (-1)^{p(a)p(b)}Y_\tw(b,w^{1/2})Y_\tw(a,z^{1/2})u,$$
$$\text{and}\qquad Y_\tw(Y(a,z-w)b,w^{1/2})u$$
are the expansions of $f$ in the spaces $M((z^{1/2}))((w^{1/2}))$, $M((w^{1/2}))((z^{1/2}))$, and $M((w^{1/2}))((z-w))$, respectively; and
\item if $\theta(a)=(-1)^ma$, then $a_{(n),\tw}=0$ for $n\notin\ZZ+\frac{m}{2}$.
\end{enumerate}
More details on twisted vertex operators can be found, e.g. in \cite{MR2074176,MR1372724}.

The \emph{Virasoro} algebra $\vir$ is the Lie algebra spanned by $L(m)$, for $m\in\ZZ$, and a central element ${\bf c}$, with Lie bracket
\begin{gather}\label{eqn:LmLn}
[L(m),L(n)]=(m-n)L(m+n)+\frac{m^3-m}{12}\delta_{m+n,0}{\bf c}.
\end{gather}
A representation $\vir\to \End(V)$ of the Virasoro algebra is said to have {\em central charge} $c$ if the central element ${\bf c}$ acts as multiplication by $c$ on $V$.

A {\em super vertex {operator} algebra} is a super vertex algebra $V=V_{\bar{0}}\oplus V_{\bar{1}}$ containing a \emph{Virasoro element} $\omega\in V_{\bar{0}}$ such that if $L(n):=\omega_{(n+1)}$ for $n\in \ZZ$ then
\begin{enumerate}\setcounter{enumi}{4}
\item $L({-1})=T$;
\item $[\Lm, \Ln]=(m-n)L(m+n)+\frac{m^3-m}{12}\delta_{m+n,0}c\Id_V$ for some $c\in \CC$;
\label{item:Vircond}
\item $\Lo$ is a diagonalizable operator on $V$, with rational eigenvalues bounded below, and finite-dimensional eigenspaces.
\end{enumerate}
According to item \ref{item:Vircond}, the components of $Y(\omega,z)$ generate a representation of the Virasoro algebra on $V$ with central charge $c$. 

In this work, all super vertex operator algebras will have rational central charges. For $V$ such a super vertex operator algebra let us write $V=\bigoplus_{n\in \QQ}V_n$ for the decomposition of $V$ into eigenspaces\footnote{Note that $V_n$ often denotes the $L(0)$-eigenspace with eigenvalue $n$, elsewhere in the literature, and in (2.5) of \cite{vacogm}, in particular.} for $L(0)-\frac{\bf c}{24}$.
\begin{gather}\label{eqn:va-Vn}
	V_n:=\left\{v\in V\mid \left(L(0)-\tfrac{\bf c}{24}\right)v=nv\right\}
\end{gather}
Similarly, if $V_\tw$ is a canonically-twisted module for $V$, we also write $L(n)$ for $\omega_{(n+1),\tw}$ (cf. (\ref{eqn:va-Ytw})), a linear operator on $V_\tw$, and we write $(V_\tw)_n$ for the eigenspace with eigenvalue $n$ for $L(0)-\frac{\bf c}{24}$.
\begin{gather}\label{eqn:va-Vtwn}
	(V_\tw)_n:=\left\{v\in V_\tw\mid \left(L(0)-\tfrac{\bf c}{24}\right)v=nv\right\}
\end{gather}

For $V$ a super vertex operator algebra, suppose to be given an element $\jmath\in V$ with $L(0)\jmath=\jmath$ such that if $J(n):=\jmath_{(n)}$ then 
\begin{gather}	\label{eqn:va-U1commrels}
	[J(m),J(n)]=k\delta_{m+n,0}\Id_V
\end{gather}
in $\End(V)$ for some $k\in\CC$. We call such an element $\jmath$ a {\em $\U(1)$ element} for $V$, and we call $k$ the {\em level} of $\jmath$. The action of the operator $J(0)=\jmath_{(0)}$ preserves the eigenspaces for $L(0)-\frac{\bf c}{24}$ by hypothesis, and may in addition be diagonalizable. In such a situation we write $V=\bigoplus_{n,r}V_{n,r}$ for the corresponding decomposition into bi-graded subspaces for $V$.
\begin{gather}\label{eqn:va-Vnr}
	V_{n,r}:=\left\{v\in V\mid \left(L(0)-\tfrac{\bf c}{24}\right)v=nv,\,J(0)v=rv\right\}
\end{gather}
Similarly, for $V_\tw$ a canonically-twisted $V$-module, we abuse notation slightly by writing $J(0)$ also for $\jmath_{(0),\tw}$, an operator on $V_\tw$, and define
\begin{gather}\label{eqn:va-Vtwnr}
	(V_\tw)_{n,r}:=\left\{v\in V_\tw\mid \left(L(0)-\tfrac{\bf c}{24}\right)v=nv,\, J(0)v=rv\right\}.
\end{gather}

We will use the bi-gradings arising from suitably chosen $\U(1)$ elements in a certain distinguished super vertex operator algebra to define weak Jacobi forms in \S\ref{sec:tg} and \S\ref{sec:um}.

\section{The Clifford Module Construction}\label{sec:cliff}

We now briefly review the standard construction that attaches a super vertex operator algebra, and a canonically-twisted module for it, to a vector space equipped with a non-degenerate symmetric bilinear form. We refer to \cite{MR1123265} for a very thorough treatment, and to \cite{vacogm} for a fuller description using the same notation that is employed here.

Let $\a$ be a finite dimensional complex vector space equipped with a non-degenerate symmetric bilinear form $\langle\cdot,\cdot\rangle$, and for each $n\in\ZZ$ let $\a(n+\tfrac{1}{2})$ be a vector space isomorphic to $\a$, with a chosen isomorphism $\a\to\a(n+\tfrac{1}{2})$, denoted $u\mapsto u(n+\tfrac{1}{2})$. Set 
\be
\aa^-:=\bigoplus_{n< 0}\a(n+\tfrac{1}{2})
\ee
and define $A(\a)$ to be a copy of the regular left-module for the exterior algebra of the vector space $\aa^-$,
\begin{gather}
	A(\a):= \bigwedge(\aa^-)\vv.
\end{gather}

For $u\in \a$ and $m\in \ZZ$ we regard $u(m+\frac12)$ as an operator on $A(\a)$ by letting $u(m+\frac12)$ act by left multiplication in case $m<0$. For $m\geq 0$ the action of $u(m+\frac12)$ is determined by the rules
\begin{gather}\label{eqn:va-Arels}
	u(m+\tfrac12)v(n+\tfrac12)a=-v(n+\tfrac12)u(m+\tfrac12)a-2\delta_{m+n+1,0}\langle u,v\rangle a,\quad u(m+\tfrac12)\vv=0,
\end{gather}
for $v\in \a$, $n<0$, and $a\in A(\a)$. Then for $u\in\a$ the vertex operator attached to $u(-\tfrac{1}{2})\vv$ is defined by setting
\be
Y(u(-\tfrac{1}{2})\vv,z):=\sum_{n\in\ZZ}u(n+\tfrac{1}{2})z^{-n-1}.
\ee
Theorem 4.4.1 of \cite{MR2082709} ensures that this specification extends uniquely to a super vertex algebra structure on $A(\a)$. 

Choose an orthonormal basis $\{e_i:1\le i\le \dim(\a)\}$ for $\a$. Then 
\begin{gather}\label{eqn:va-omega}
\omega:=-\frac{1}{4}\sum_{i=1}^{\dim(\a)}e_i(-\tfrac{3}{2})e_i(-\tfrac{1}{2})\vv
\end{gather}
serves as a Virasoro element for $A(\a)$, equipping it with the structure of a super vertex operator algebra with central charge $c=\frac12{\dim\a}$.

A similar construction produces a canonically-twisted module for $A(\a)$ (cf. (\ref{eqn:va-Ytw})), which we call $A(\a)_\tw$. We recall this now, assuming for the sake of simplicity that $\dim\a$ is even.

For each $n\in\ZZ$ let $\a(n)$ be a vector space isomorphic to $\a$, with a chosen isomorphism $\a\to\a(n)$ denoted $u\mapsto u(n)$. Suppose also to be given isotropic subspaces $\a^\pm<\a$ such that $\a=\a^-\oplus \a^+$. Such a decomposition is called a {\em polarization} of $\a$. Now set
\be\label{eqn:va-aatwminus}
\aa_\tw^-:=\a^-(0)\oplus \bigoplus_{n<0}\a(n),
\ee
where $\a^\pm(0)$ is the image of $\a^{\pm}$ under the isomorphism $u\mapsto u(0)$,
and define $A(\a)_\tw$ to be a regular left-module for the exterior algebra of the vector space $\aa^-_\tw$, so that
\begin{gather}
	A(\a)_\tw:= \bigwedge(\aa_\tw^-)\vv_\tw.
\end{gather}
Similar to before, we regard 
$u(m)$ as an operator on $A(\a)_\tw$ by letting $u(m)$ act by left multiplication in case $m<0$. For $m\geq 0$ the action of $u(m)$ is determined by the rules
\begin{gather}\label{eqn:va-Atwrels}
	u(m)v(n)a=v(n)u(m)a-2\delta_{m+n,0}\langle u,v\rangle a,\quad u(m)\vv_\tw=0 \Leftarrow \text{$m>0$ or $u\in \a^+$},
\end{gather}
for $v\in \a$, $n\leq 0$, and $a\in A(\a)_\tw$. For $u\in\a$ the twisted vertex operator attached to $u(-\tfrac{1}{2})\vv$ is defined by setting
\be
Y_\tw(u(-\tfrac{1}{2})\vv,z):=\sum_{n\in\ZZ}u(n)z^{-n-1/2}.
\ee
This specification extends uniquely to a canonically-twisted $A(\a)$-module structure on $A(\a)_\tw$, according to the discussion in \S2.2 of \cite{MR2074176}. In particular, the twisted vertex operator
\begin{gather}\label{eqn:va-Ytwomega}
Y_\tw(\omega,z^{1/2})=\sum_{n\in\ZZ} \Ln z^{-n-2}
\end{gather}
equips $A(\a)_\tw$ with a representation of the Virasoro algebra, and the action of $\Lo:=\omega_{(1),\tw}$ is diagonalizable. An explicit computation yields that the eigenvalues of $\Lo-\frac{\bf c}{24}$ on $A(\a)_\tw$ are contained in $\ZZ+\frac{1}{24}\dim\a$, so that
\begin{gather}\label{eqn:cliff-Atwgrading}
	(A(\a)_\tw)_n=0\Leftarrow n\notin \ZZ+\frac{1}{24}\dim\a.
\end{gather}
(Cf. (\ref{eqn:va-Vtwn}).)

The super vertex operator algebra $A(\a)$ admits various $\U(1)$ elements (cf. (\ref{eqn:va-U1commrels})). For example, if isotropic vectors $a_1^\pm,\ldots, a_d^\pm\in\a$ are chosen (for some $d\leq \frac12\dim\a$) such that $\langle a_i^\pm,a_j^\mp\rangle=\delta_{i,j}$, then 
\begin{gather}\label{eqn:cliff-jmath}
	\jmath:=\frac12\sum_{i=1}^d a_i^-(-1/2)a_i^+(-1/2)\vv\in A(\a)
\end{gather}
is a $\U(1)$ element with level $d$. Moreover, the action of $J(0):=\jmath_{(0)}$ on $A(\a)$ is diagonal, with integer eigenvalues, and similarly for $J(0):=\jmath_{(0),\tw}$ as an operator on $A(\a)_\tw$. Since it will be useful later in the article, we record a more detailed statement as follows for future use.

\begin{lemma}\label{lem:tg-JPirel}
Let $\jmath$ as in (\ref{eqn:cliff-jmath}). Then $\jmath$ is a $\U(1)$ element for $A(\a)^0$ with level $d$. We have $J(0)\vv=0$ and $J(0)\vv_\tw=\frac d4\vv_\tw$. Also, $[J(0),a_i^{\pm}(r)]=\pm a_i^{\pm}(r)$ for all $1\leq i\leq d$ and $r\in\frac{1}{2}\ZZ$, and if $u\in\a$ is orthogonal to $\Span\{a_1^\pm,\ldots,a_d^\pm\}$ then $J(0)$ commutes with $u(r)$ for all $r\in\frac{1}{2}\ZZ$.
\end{lemma}
\begin{proof}
The lemma follows from standard vertex algebra computations. For example, suppose $a=u(-\tfrac{1}{2})v(-\tfrac{1}{2})\vv$ for some $u,v\in\a$, and $a_{(n),\tw}\in\End(A(\a)_\tw)$ is the coefficient of $z^{-n-1}$ in $Y_\tw(a,z^{1/2}):A(\a)_\tw\to A(\a)_\tw((z))$. Then we have 
\begin{gather}\label{eqn:tg-a0tw}
	a_{(0),\tw}=\sum_{r\in\ZZ} :u(-r)v(r):
\end{gather}
where $:\;\;:$ denotes the {\em fermonic normal ordering}, defined so that 
\begin{gather}
	:u(-r)v(r):\;=
	\begin{cases}
		u(-r)v(r)&\text{ for $r>0$,}\\
		\frac{1}{2}(u(0)v(0)-v(0)u(0))&\text{ for $r=0$,}\\
		-v(r)u(-r)&\text{ for $r<0$.}
	\end{cases}
\end{gather}
For $a_{(0)}:A(\a)\to A(\a)$ we should replace $\ZZ$ with $\ZZ+\frac{1}{2}$ in the summation in (\ref{eqn:tg-a0tw}). The required relations follow from (\ref{eqn:va-Arels}) and (\ref{eqn:va-Atwrels}).
\end{proof}

Define the {\em Clifford algebra} associated to $\a$ and $\langle\cdot\,,\cdot\rangle$ by setting 
\begin{gather}
\Cliff(\a):=T(\a)/I
\end{gather}
where $T(\a)$ is the tensor algebra of $\a$, and $I$ is the ideal generated by the $u\otimes u+\langle u,u\rangle{\bf 1}$ for $u\in \a$.
Given $u_i\in \a$, write $u_1\cdots u_k$ for the image of $u_1\otimes\cdots\otimes u_k\in T(\a)$ in $\Cliff(\a)$.
Then the relations (\ref{eqn:va-Atwrels}) ensure that $\Cliff(\a)$ acts naturally on $A(\a)_\tw$, via $u_1\cdots u_k\mapsto u_1(0)\cdots u_k(0)$, for $u_i\in \a$. 

The $\Cliff(\a)$-submodule of $A(\a)_\tw$ generated by $\vv_\tw$ is the unique (up to isomorphism) non-trivial irreducible representation of $\Cliff(\a)$. We denote this subspace of $A(\a)_\tw$ by $\CM$. We have 
\begin{gather}\label{eqn:va-AtwwedgeCM}
	A(\a)_\tw\simeq\bigwedge\left(\bigoplus_{n<0}\a(n)\right)\otimes \CM,
	\quad
	\CM\simeq \bigwedge(\a^-(0))\vv_\tw.
\end{gather}

\section{Lifting to the Spin Group}\label{sec:spin}

Let $\a$ be a complex vector space equipped with a non-degenerate symmetric bilinear form $\langle\cdot\,,\cdot\rangle$, as in \S\ref{sec:va}. To recall the definition of the {\em spin group} of $\a$, denoted $\Spin(\a)$, we remind that the {\em main anti-automorphism} $\alpha$ of $\Cliff(\a)$ is defined by setting $\alpha(u_1\cdots u_k):=u_k\cdots u_1$ for $u_i\in\a$. The group $\Spin(\a)$ is composed of the even, invertible elements $x\in\Cliff(\a)$ with $\alpha(x)x={\bf 1}$ such that $xux^{-1}\in \a$ whenever $u\in\a$.

Set 
\begin{gather}\label{eqn:spin-xbrackets}
x(u):=xux^{-1}
\end{gather} 
for $x\in \Spin(\a)$ and $u\in \a$. Then $u\mapsto x(u)$ is a linear transformation on $\a$ belonging to $\SO(\a)$ and the assignment $x\mapsto x(\cdot)$ defines a surjective map $\Spin(\a)\to\SO(\a)$
with kernel $\{\pm \mathbf{1}\}$. 
\begin{gather}
\begin{split}\label{eqn:spin-spintoso}
	1\to\{\pm \mathbf{1}\}\to\Spin(\a)&\to \SO(\a)\to1\\
		x&\mapsto x(\cdot)
\end{split}
\end{gather}

The group $\Spin(\a)$ acts naturally on $A(\a)$ and $A(\a)_\tw$. Explicitly, if $a\in A(\a)$ has the form $a=u_1(-n_1+\tfrac12)\cdots u_k(-n_k+\tfrac12)\vv$ for some $u_i\in\a$ and $n_i\in\ZZ^+$, then 
\begin{gather}\label{eqn:gps:spin-defnxa}
xa=u_1'(-n_1+\tfrac12)\cdots u_k'(-n_k+\tfrac12)\vv,
\end{gather}
for $x\in \Spin(\a)$, where $u_i':=x(u_i)$ (cf. (\ref{eqn:spin-xbrackets})). Evidently $-\mathbf{1}$ is in the kernel of this assignment $\Spin(\a)\to\Aut(A(\a))$, so the action factors through $\SO(\a)$. 

For $A(\a)_\tw$ we use (\ref{eqn:va-AtwwedgeCM}) to identify the elements of the form 
\begin{gather}
u_1(-n_1)\cdots u_k(-n_k)\otimes y
\end{gather} 
as a spanning set, where $u_i\in\a$ and $n_i\in\ZZ^+$ as above, and $y\in \CM$ (cf. (\ref{eqn:va-AtwwedgeCM})). The image of such an element under $x\in \Spin(\a)$ is given by $u_1'(-n_1)\cdots u_k'(-n_k)\otimes xy$, where $u_i':=x(u_i)$ as before. Since $\CM$ is a faithful $\Spin(\a)$-module, so too is $A(\a)_\tw$. 

In terms of the vertex operator correspondences we have
\begin{gather}\label{eqn:gps:spin-spinactsvops}
\begin{split}
Y(xa,z)xb&=xY(a,z)b=\sum_{n\in\ZZ}x(a_{(n)}b)z^{-n-1},\\
Y_\tw(xa,z^{1/2})xc&=xY_\tw(a,z^{1/2})c=\sum_{n\in\tfrac12\ZZ}x(a_{(n),\tw}c)z^{-n-1},
\end{split}
\end{gather}
for $x\in \Spin(\a)$, $a,b\in A(\a)$ and $c\in A_{\tw}(\a)$.

Say that $\wh{g}\in \Spin(\a)$ is a {\em lift} of an element $g\in \SO(\a)$ if $\wh{g}$ has the same order as $g$, and $\wh{g}(\cdot)=g$ (cf. (\ref{eqn:spin-xbrackets})). More generally, say that $\wh{G}<\Spin(\a)$ is a {\em lift} of a subgroup $G<\SO(\a)$ if the natural map (\ref{eqn:spin-spintoso}) induces an isomorphism $\wh{G}\xrightarrow\sim G$.

Suppose we are given an identification $\a=\LL\otimes_\ZZ\CC$ where $\LL$ is the Leech lattice (cf. \S\ref{sec:lat}). In this situation we set $G:=\Aut(\LL)$, a copy of the Conway group $\Co_0$ which we may naturally regard as a subgroup of $\SO(\a)$. Proposition 3.1 in \cite{vacogm} demonstrates that there is a unique lift of $G<\SO(\a)$ to $\Spin(\a)$.
\begin{proposition}[\cite{vacogm}]\label{prop:spin-whG}
If $\a=\LL\otimes_\ZZ\CC$ and $G=\Aut(\LL)<\SO(\a)$ then there is a unique subgroup $\wh{G}<\Spin(\a)$ such that the natural map $\Spin(\a)\to \SO(\a)$ induces an isomorphism $\wh{G}\xrightarrow{\sim} G$.
\end{proposition}
With $\wh{G}\xrightarrow{\sim}G=\Aut(\LL)\simeq \Co_0$
as in Proposition \ref{prop:spin-whG}, write
\begin{gather}\label{eqn:spin-gtowhg}
	g\mapsto \wh{g}
\end{gather}
for the inverse isomorphism, ${G}\xrightarrow{\sim}\wh{G}$.

Assuming an identification $\a=\LL\otimes_\ZZ\CC$, we now construct some elements of $\wh{G}\simeq \Co_0$ explicitly. Let $g\in G=\Aut(\LL)$ and choose a basis $\{a_i^\pm\}$ for $\a$, consisting of eigenvectors for $g$, such that the $\a^{\pm}:=\Span_\CC\{a_i^\pm\}$ are isotropic subspaces of $\a$, and $\langle a_i^\pm,a_j^\mp\rangle=\delta_{i,j}$. Write $\lambda_i$ for the eigenvalue of $g$ attached to $a_i^+$. 
\begin{gather}\label{eqn:spin-lambdai}
	g(a_i^{\pm})=\lambda_i^{\pm 1}a_i^\pm
\end{gather}
Then $\a=\a^-\oplus\a^+$ is a $g$-invariant polarization of $\a$, and we may assume that 
\begin{gather}\label{eqn:spin-zz}
\zz:=\prod_{i=1}^{12}e^{\frac{\pi}{2}X_i}
\end{gather}
belongs to $\wh{G}$, where
\begin{gather}\label{eqn:spin-Xi}	
	X_i:=\frac{\ii}{2}\left(a_i^-a_i^+-a_i^+a_i^-\right)\in\Cliff(\a).
\end{gather}
For if $\zz\notin\wh{G}$ for our first choice of basis $\{a_i^\pm\}$, then $-\zz\in \wh{G}$, and $\zz$ gets replaced with $-\zz$ in (\ref{eqn:spin-zz}) once we swap $a_i^-$ with $a_i^+$ for some $i$.
We call $\zz$ as in (\ref{eqn:spin-zz}) the lift of $-\Id_\a$ {\em associated} to the polarization $\a=\a^-\oplus \a^+$.

Note that $X_i^2=-{\bf 1}$, so $e^{\alpha X_i}=(\cos \alpha){\bf 1}+(\sin \alpha)X_i$ in $\Cliff(\a)$ for $\alpha\in \RR$. Also, $X_ia^{\pm}_i=\pm \ii a^{\pm}_i=-a^{\pm}_iX_i$, and $X_i$ commutes with $X_j$ and $a_j^\pm$ when $i\neq j$. This entails (cf. e.g. \S3.1 of \cite{vacogm}) that the lift $\wh{g}$ of $g$ to $\wh{G}<\Spin(\a)$ is given explicitly by
\begin{gather}\label{eqn:spin-whgexp}
	\wh{g}=\prod_{i=1}^{12}e^{\alpha_iX_i},
\end{gather}
for some $\alpha_i\in 2\pi\QQ$ such that $\lambda_i=e^{2\alpha_i\ii}$ (cf. (\ref{eqn:spin-lambdai})).

We can now compute the trace of $\wh{g}$ as an operator on the $\Spin(\a)$-module $\CM$ (cf. (\ref{eqn:va-AtwwedgeCM})).
Indeed, since $X_i\vv_\tw=\ii\vv_\tw$, and the $2^{12}$ monomials 
\begin{gather}
	a_{i_1}^-(0)\cdots a_{i_k}^-(0)\vv_\tw
\end{gather}
with $1\leq i_1<\cdots<i_k\leq 12$ furnish a basis for $\CM$, we have
\begin{gather}\label{eqn:spin-trCMwhg}
	\tr_{\CM}\wh{g}=\nu\prod_{i=1}^{12}(1+\lambda_i^{-1})=\prod_{i=1}^{12}(\nu_i+\nu_i^{-1})
\end{gather}
for $\nu_i:=e^{\alpha_i\ii}$ and $\nu:=\prod_{i=1}^{12}\nu_i$.

Suppose that $V$ is a real vector space contained in $\a$, such that $\a=V\otimes_{\RR}\CC$, and such that $\lab\cdot\,,\cdot\rab$ restricts to an $\RR$-valued bilinear form on $V$. (E.g. $\a=\LL\otimes_\ZZ\CC$ and $V=\LL\otimes_\ZZ\RR$.) Then a choice of orientation $\RR^+\w\subset\bigwedge^{24}(V)$ on $V$ also determines a lift of $-\Id_\a$ to $\Spin(\a)$, for given an ordered basis $\{e_i\}$ of $V$ satisfying $\lab e_i,e_j\rab=\pm \delta_{i,j}$ and 
\begin{gather}
e_1\wedge \cdots \wedge e_{24 }\in\RR^+\w,
\end{gather} 
we obtain one of the two lifts of $-\Id_\a$ to $\Spin(\a)$ by setting
\begin{gather}\label{eqn:spin-zzee}
	\zz':=e_1\cdots e_{24}\in\Cliff(\a).
\end{gather}
We call $\zz'$ the lift of $-\Id_\a$ {\em associated} to the orientation $\RR^+\omega$. Evidently, a change in orientation replaces $\zz'$ with $-\zz'$.

We see now from Proposition \ref{prop:spin-whG} that $\LL$ is naturally oriented. For setting $\a=\LL\otimes_\ZZ\CC$ and $V=\LL\otimes_\ZZ\RR\subset \a$, and taking $G=\Aut(\LL)<\SO(\a)$ and $\wh{G}<\Spin(\a)$ as in Proposition \ref{prop:spin-whG}, we may choose the preferred orientation on $V$ to be the one for which the associated lift $\zz'$ of $-\Id_\a$ (cf. (\ref{eqn:spin-zzee})) belongs to $\wh{G}$. By the same token there is a preferred $\SO(\a)$-orbit of polarizations $\a=\a^-\oplus\a^+$ of $\a=\LL\otimes_\ZZ\CC$, being the one for which an associated lift $\zz$ (cf. (\ref{eqn:spin-zz})) belongs to $\wh{G}$.

Recall that the construction of $A(\a)_\tw$ depends upon a choice of polarization $\a=\a^-\oplus \a^+$ (cf. (\ref{eqn:va-aatwminus})). If $\zz$ is the lift associated to $\a=\a^-\oplus \a^+$ (cf. (\ref{eqn:spin-zz})) then we have
\begin{gather}
	\zz\vv_\tw=
	\vv_\tw
\end{gather}
(since $\dim\a=0\mod 4$). Thus $\zz$ acts with order two on $A(\a)_\tw$. We write
\begin{gather}\label{eqn:spin-Atw01}
A(\a)_\tw=A(\a)_\tw^0\oplus A(\a)_\tw^1
\end{gather}
for the decomposition into eigenspaces for $\zz$, where $\zz$ acts as $(-1)^j\Id$ on $A(\a)_\tw^j$. The element $\zz$ is central so the action of $\Spin(\a)$ on $A(\a)_\tw$ preserves the decomposition (\ref{eqn:spin-Atw01}). 

From the description (\ref{eqn:gps:spin-defnxa}), we see that writing $A(\a)^j$ for the $(-1)^j$ eigenspace of either $\zz$ or $-\zz$ recovers the super space decomposition 
of $A(\a)$.
\begin{gather}\label{eqn:spin-A01}
A(\a)=A(\a)^0\oplus A(\a)^1
\end{gather}

\section{The Conway Moonshine Module}\label{sec:dist}

We now recall the main construction from \cite{vacogm}.

Assume henceforth that $\a$ is a $24$-dimensional vector space over $\CC$, equipped with a bilinear form $\langle\cdot\,,\cdot\rangle:\a\otimes\a\to \CC$ that is symmetric and non-degenerate. Suppose also to be chosen a lift $\zz'\in \Spin(\a)$ of $-\Id_\a$. (In practice, $\zz'$ will be the lift of $-\Id_\a$ associated to an orientation on some real vector space $V\subset\a$, as in (\ref{eqn:spin-zzee}).) Then for $\a=\a^-\oplus\a^+$ a polarization such that $\zz=\zz'$ (cf. (\ref{eqn:spin-zz})) we set
\begin{gather}\label{eqn:dist-vsnvsnt}
	\vsn:=A(\a)^0\oplus A(\a)_{\tw}^1,\quad
	\vsnt:=A(\a)^1\oplus A(\a)_{\tw}^0,
\end{gather} 
where $A(\a)$ and $A(\a)_\tw$ are constructed as in \S\ref{sec:va}, and the subspaces $A(\a)^j$ and $A(\a)_\tw^j$ are as in (\ref{eqn:spin-A01}) and (\ref{eqn:spin-Atw01}), respectively. According to \cite{vacogm}, The $A(\a)^0$-module $\vsn$ is naturally a super vertex operator algebra, and $\vsnt$ is naturally a canonically-twisted module for $\vsn$. \begin{proposition}[\cite{vacogm}]\label{prop:dist-vsnstruc}
The $A(\a)^0$-module structure on $\vsn$ extends uniquely to a super vertex operator 
algebra structure on $\vsn$, and the $A(\a)^0$-module structure on $\vsnt$ extends uniquely to a canonically-twisted $\vsn$-module structure.
\end{proposition}

The super vertex operator algebra $\vsn$ is distinguished. The following abstract characterization of $\vsn$ has been established in \cite{vacogm}. (Cf. Theorem 5.15 of \cite{SM}.)
\begin{theorem}[\cite{vacogm}]\label{thm:de:vsnuniq}
The super vertex operator algebra $\VU$ is the unique self-dual $C_2$-cofinite rational super vertex operator algebra of CFT type with central charge $12$ such that $\Lo u=\tfrac{1}{2}u$ for $u\in \VU$ implies $u=0$.
\end{theorem}
We refer to \cite{vacogm} for the precise meanings of the terms self-dual, $C_2$-cofinite, rational, and CFT type. Briefly, a super vertex operator algebra $V$ is rational if any $V$-module can be written as a direct sum of irreducible $V$-modules. We say that $V$ is self-dual if it is irreducible as a module over itself, and if $V$ is the only irreducible $V$-module up to isomorphism. As explained in \cite{vacogm}, Theorem \ref{thm:de:vsnuniq} identifies $\VU$ as an analogue for super vertex operator algebras of the extended binary Golay code, of the Leech lattice $\LL$ (cf. \S\ref{sec:lat}), and (conjecturally) of the moonshine module vertex operator algebra $\vn$ (cf. \S\ref{sec:intro:mm}).

As explained in \S\ref{sec:spin}, the spin group $\Spin(\a)$ acts naturally on the $A(\a)^j$ and $A(\a)_\tw^j$, so it acts naturally on $\vsn$ and $\vsnt$. In particular, given an identification $\a=\LL\otimes_\ZZ\CC$, we naturally obtain actions of the Conway group $\Co_0$ on $\vsn$ and $\vsnt$, because $G=\Aut(\LL)<\SO(\a)$ admits a unique lift $\wh{G}<\Spin(\a)$, according to Proposition \ref{prop:spin-whG}.

Since it will be useful in the sequel we now recall (cf. \S4.3 of \cite{vacogm}) explicit expressions for the graded traces of elements of $\wh{G}\simeq \Co_0$ on $\vsn$ and $\vsnt$. In preparation for this, define $\eta_g(\tau)$ for $g\in G=\Aut(\LL)$ by setting
\begin{gather}\label{eqn:dist-etag}
	\eta_g(\tau):=q\prod_{n>0}\prod_{i=1}^{12}(1-\lambda_i^{-1} q^n)(1-\lambda_iq^n),
\end{gather}
where the $\lambda_i^{\pm 1}$ are the eigenvalues for $g$ acting on $\a$, as in (\ref{eqn:spin-lambdai}), and define $C_g$ by setting\footnote{Note that $C_g$ is denoted $C_{\wh{g}}$ in \cite{vacogm}.}
\begin{gather}\label{eqn:dist-Cg}
	C_g:=\tr_{\CM}\zz\wh{g}.
\end{gather}
(Cf. (\ref{eqn:spin-gtowhg}).) Note that
\begin{gather}\label{eqn:dist-fracetag}
\frac{\eta_g(\tau/2)}{\eta_g(\tau)}=q^{-1/2}\prod_{n>0}\prod_{i=1}^{12}(1-\lambda_i^{-1}q^{n-1/2})(1-\lambda_iq^{n-1/2}).
\end{gather}
Also,
\begin{gather}\label{eqn:spin-Cgnu}
	C_g=\nu\prod_{i=1}^{12}(1-\lambda_i^{-1})=\prod_{i=1}^{12}(\nu_i-\nu_i^{-1})
\end{gather}
according to (\ref{eqn:spin-trCMwhg}), where the $\nu$ and $\nu_i$ are as in (\ref{eqn:spin-trCMwhg}), the $\nu_i$ being square roots of the $\lambda_i$. So in particular, $C_g$ is determined up to sign by the eigenvalues of $g$, and $C_g=0$ exactly when $g$ has a non-zero fixed point in $\a$.

For $g\in G$ define
\begin{align}
	T^s_g(\tau)&:=\tr_{\vsn}\zz\wh{g}q^{L(0)-c/24},\label{eqn:dist-Tsg}\\
	T^s_{g,\tw}(\tau)&:=\tr_{\vsnt}\zz\wh{g}q^{L(0)-c/24}.\label{eqn:dist-Tsgtw}
\end{align}
We obtain the explicit formulas
\begin{align}
	T^s_g(\tau)&=\frac12\left(\frac{\eta_g(\tau/2)}{\eta(\tau)}+\frac{\eta_{-g}(\tau/2)}{\eta_{-g}(\tau)}+C_g\eta_g(\tau)-C_{-g}\eta_{-g}(\tau)\right),\label{eqn:dist-Tsgdir}\\
	T^s_{g,\tw}(\tau)&=\frac12\left(\frac{\eta_g(\tau/2)}{\eta(\tau)}-\frac{\eta_{-g}(\tau/2)}{\eta_{-g}(\tau)}+C_g\eta_g(\tau)+C_{-g}\eta_{-g}(\tau)\right),\label{eqn:dist-Tsgtwdir}
\end{align}
from Lemma 4.6 of \cite{vacogm} (or by direct calculation using (\ref{eqn:dist-vsnvsnt}), (\ref{eqn:dist-etag}) and (\ref{eqn:dist-Cg})). 

Define also
\begin{gather}\label{eqn:dist-chig}
	\chi_g:=\tr_\a g
\end{gather}
so that $\chi_g=\sum_{i=1}^{12}(\lambda_i+\lambda_i^{-1})$ for $\lambda_i$ as in (\ref{eqn:spin-lambdai}).
In \cite{vacogm} the following alternative identities are proved, 
\begin{align}
	T^s_g(\tau)&=\frac{\eta_g(\tau/2)}{\eta_g(\tau)}+\chi_g,\label{eqn:dist-Tsgchig}\\
	T^s_{g,\tw}(\tau)&=C_g\eta_g(\tau)-\chi_g.\label{eqn:dist-Tsgtwchig}
\end{align}
Both the equivalence of (\ref{eqn:dist-Tsgdir}) with (\ref{eqn:dist-Tsgchig}), and of (\ref{eqn:dist-Tsgtwdir}) with (\ref{eqn:dist-Tsgtwchig}) follow from the following non-trivial identity, which is the content of the main technical lemma (Lemma 4.8) of \cite{vacogm}.
\begin{lemma}[\cite{vacogm}]\label{lemma:eta identity}
For $g\in G=\Aut(\LL)$ we have 
\be\label{equation:eta identity}
2\chi_g
-\frac{\eta_{-g}(\tau/2)}{\eta_{-g}(\tau)}+\frac{\eta_g(\tau/2)}{\eta_g(\tau)}+C_{{-g}}\eta_{-g}(\tau)-C_{{g}}\eta_g(\tau)=0.
\ee
\end{lemma}

The main result of \cite{vacogm} is that $T^s_g$ is the normalized principal modulus for a genus zero group $\Gamma_g<\SL_2(\RR)$, and $T^s_{g,\tw}$ is also a principal modulus, so long as $C_g\neq 0$. From the explicit descriptions of the $\Gamma_g$ in Table 1 of \cite{vacogm} we see that $T^s_g(2\tau)$ is invariant for some $\Gamma_0(N)$, with $N$ depending on $g$, for every $g\in \Co_0$.
\begin{theorem}[\cite{vacogm}]\label{thm:dist-Tsgpm}
Let $g\in \Co_0$. Then $T^s_g(2\tau)$ is the normalized principal modulus for a genus zero group $\Gamma_g<\SL_2(\RR)$ that contains some $\Gamma_0(N)$. If $g$ has a fixed point in its action on the Leech lattice then the function $T^s_{g,\tw}(\tau)$ is constant, with constant value $-\chi_g$. If $g$ has no fixed points then $T^s_{g,\tw}(\tau)$ is a principal modulus for a genus zero group $\Gamma_{g,\tw}<\SL_2(\RR)$.
\end{theorem}

The groups $\Gamma_g$ and $\Gamma_{g,\tw}$ are described explicitly in \cite{vacogm}.

Note that the characteristic polynomial for the action of an element $g\in \coa$ on $\a$ can always be written in the form $\prod_{m>0}(1-x^m)^{k_m}$ for some non-negative integers $k_m$, all but finitely many being zero. (Cf. \S4.3 of \cite{vacogm}.) It follows from (\ref{eqn:dist-etag}) that 
$\eta_g(\tau)=\prod_{m>0}\eta(m\tau)^{k_m}$. The formal product 
\begin{gather}\label{eqn:spin-pig}
\pi_g:=\prod_{m>0}m^{k_m}
\end{gather}
is called the {\em Frame shape} of $g$.

\section{Twining Genera}\label{sec:tg}

In this section we establish the main results of the paper.

Let $X$ be a projective complex K3 surface and let $\sigma=(\mathcal{P},Z)$ be a stability condition in Bridegland's space $\Stabo(X)$ (cf. \S\ref{sec:de}). Presently we will attach a formal series $\phi_g\in \CC[y^{\pm 1}][[q]]$ to any $g\in G_{\Pi}=\Aut_s(\Db(X),\sigma)$ by computing the graded trace of a suitable automorphism of the canonically-twisted module $\vsnt$ for the distinguished super vertex algebra $\vsn$ that was reviewed in \S\ref{sec:dist} (and studied in detail in \cite{vacogm}). It will develop (see Theorem \ref{theorem:jacobi form}) that $\phi_g$ is a weak Jacobi form of weight zero and index one, with some level (cf. \S\ref{sec:autfms}).

In order to define $\phi_g$ we require explicit realizations of $\vsn$ and $\vsnt$. In preparation for this we take $\a=\mlc$ to be the complex vector space enveloping the Mukai lattice. Once and for all we choose an orientation on $\mlr\subset \a$,
\begin{gather}\label{eqn:tg-orientmlr}
	\RR^+\omega\subset \bigwedge^{24}\left(\mlr\right),
\end{gather}
and we let $\zz'$ (cf. (\ref{eqn:spin-zzee})) denote the corresponding lift of $-\Id_\a$ to $\Spin(\a)$. Then, for a polarization $\a=\a^-\oplus \a^+$ such that $\zz=\zz'$ (cf. (\ref{eqn:spin-zz})), we identify
\begin{gather}
	\vsn=A(\a)^0\oplus A(\a)_\tw^1,\quad
	\vsnt=A(\a)^1\oplus A(\a)_\tw^0,
\end{gather} 
as in \S\ref{sec:dist}.

As in (\ref{eqn:va-Ytwomega}) we write $Y_{\tw}(\omega,z^{1/2})=\sum_{n\in\ZZ} L(n)z^{-n-2}$ for the twisted module vertex operator $\VV\to\VV((z))$ attached to the Virasoro element $\omega\in \vsn$ (cf. (\ref{eqn:va-omega})). Then $L(0)$ acts diagonalizably on $\VV$ with eigenvalues in $\ZZ+\tfrac{1}{2}$, thus $L(0)-\frac{\bf c}{24}$ defines an integer grading on $\vsnt$, since the central charge of $\vsn$ is $c=\frac{1}{2}\dim(\a)=12$.

The data of $X$ and $Z$ enable us to define a $\U(1)$ element (cf. \S\ref{sec:va}), and hence a second integer grading on $\vsnt$. To see this, first recall the spaces $P_X$ (cf. (\ref{eqn:de:PX})) and $P_Z$ (cf. (\ref{eqn:de:PZ})) from \S\ref{sec:de}. Let $\varsigma$ be a non-zero element of $H^{2,0}(X)$, and choose vectors 
\begin{gather}
	x_X\in\RR\Re(\varsigma),\quad
	y_X\in\RR\Im(\varsigma),\quad
	x_Z\in\RR\Re(Z),\quad
	y_Z\in\RR\Im(Z),
\end{gather}
of norm one with respect to $\langle\cdot\,,\cdot\rangle$, so that $\{x_X,y_X,x_Z,y_Z\}$ is an orthonormal basis for $\Pi=P_X\oplus P_Z$. Now set 
\begin{gather}
	a_X^\pm:=\frac{1}{\sqrt{2}}(x_X\pm \ii y_X),\quad
	a_Z^\pm:=\frac{1}{\sqrt{2}}(x_Z\pm \ii y_Z),
\end{gather}
so that the $a_X^\pm$ and $a_Z^\pm$ are isotropic, satisfying $\langle a_X^\pm,a_X^\mp\rangle=\langle a_Z^\pm,a_Z^\mp\rangle=1$. Then
\begin{gather}
\jo
:=\frac{1}{2}a_X^-(-\tfrac{1}{2})a_X^+(-\tfrac{1}{2})\vv
+\frac{1}{2}a_Z^-(-\tfrac{1}{2})a_Z^+(-\tfrac{1}{2})\vv
\end{gather}
is a $\U(1)$ element of level $4$ for $\vsn$ (cf. (\ref{eqn:cliff-jmath})). Write $J(n)\in\End(\VV)$ for the coefficient of $z^{-n-1}$ in the twisted vertex operator attached to $\jo$,
\begin{gather}
	Y_\tw(\jo,z^{1/2})=\sum_{n\in\ZZ} J(n)z^{-n-1}.
\end{gather}
According to (\ref{eqn:cliff-Atwgrading}), Lemma \ref{lem:tg-JPirel} and the fact that $\dim\a=24$, the operators $\Lo-\frac{\bf c}{24}$ and $\Jo$ equip $\VV$ with an integral bi-grading
\be
\VV=\bigoplus_{\substack{n,r\in\ZZ\\n\geq 0}} (\VV)_{n,r},
\ee
with finite-dimensional homogenous subspaces,
\begin{gather}
(\VV)_{n,r}:=\left\{v\in\VV\mid (\Lo-\tfrac{\bf c}{24})v=nv,\,\Jo v=rv\right\}.
\end{gather}

Now recall $\Gamma_\Pi<\mlz$ (cf. (\ref{eqn:de-GammaPi})) and let $\iota:\Gamma_\Pi\to \LL(-1)$ be a Leech marking of $(X,\sigma)$ (cf. (\ref{eqn:de-iota})). Choose a copy of the negative-definite Leech lattice $\LL(-1)$ in $\a$ such that $\a=\LL(-1)\otimes_\ZZ\CC$ and $\Gamma_\Pi\subset\LL(-1)$, and assume also that $\iota(\gamma)=\gamma$. Set 
\begin{gather}
	G:=\Aut(\LL(-1)),
\end{gather} 
a copy of the Conway group $\Co_0$ in $\SO(\a)$, and let $\wh{G}$ be the lift of $G$ to $\Spin(\a)$ whose existence and uniqueness is guaranteed by Proposition \ref{prop:spin-whG}. Recall that we write $g\mapsto \wh{g}$ for the isomorphism $G\to \wh{G}$. We may assume that $\zz'\in \wh{G}$ (cf. (\ref{eqn:tg-orientmlr})), for if this is not true for our first choice of $\LL(-1)$, then it becomes true once we replace $\LL(-1)$ with its image under the reflection in the hyperplane defined by a non-zero vector in $\Pi$.

As explained in \S\ref{sec:de}, the Leech marking $\iota$ induces an embedding of groups $\iota_*:G_\Pi\to G$ (cf. (\ref{eqn:de-iotastar})). Using this map to regard $G_\Pi$ as a subgroup of $G$, we suppress it from notation. Thus to each $g\in G_{\Pi}\subset G$ is associated a corresponding element $\wh{g}\in \wh{G}$. We now define $\phi_g\in \CC[y^{\pm 1}][[q]]$ by setting
\be\label{eqn:tg-phig}
\phi_g:=-\tr_\VV \zz\wh{g}y^{\Jo}q^{\Lo-c/24},
\ee
where $\zz=\zz'$ is the central element of $\wh{G}$.

Our notation $\phi_g$ obscures the choice of Leech marking for $(X,\sigma)$. We now show that this convention entails no ambiguity.
\begin{proposition}\label{prop:tg-phigindep}
The series $\phi_g$ is independent of the choice of Leech marking $\iota$.
\end{proposition}
\begin{proof}
Suppose that $\LL(-1)\subset\a$ is chosen as above, having full rank $\a=\LL(-1)\otimes_\ZZ\CC$ in $\a$, containing $\Gamma_\Pi$ as a primitive sublattice, and such that $\zz'\in\wh{G}$, for $\wh{G}$ the unique lift of $G:=\Aut(\LL(-1))\simeq \Co_0$ to $\Spin(\a)$, and $\zz'$ the lift of $-\Id_\a$ associated to the chosen orientation (\ref{eqn:tg-orientmlr}) on $\mlr$. A second choice of Leech marking leads to a second copy of the negative-definite Leech lattice, $\LL'(-1)\subset\a$, with $\a=\LL'(-1)\otimes_\ZZ\CC$ and $\Gamma_\Pi<\LL(-1)\cap\LL'(-1)$. Set $G':=\Aut(\LL'(-1))$ and write $\wh{G}'$ for the unique lift of $G'\simeq\Co_0$ to $\Spin(\a)$, and assume, as we may, that $\zz'\in \wh{G}'$.

We have $g\in G\cap G'$. Write $\wh{g}$ and $\wh{g}'$ for the respective lifts to $\Spin(\a)$, determined by $\wh{G}$ and $\wh{G}'$. We have $\wh{g}=\pm\wh{g}'$, and we require to show that, in fact, $\wh{g}=\wh{g}'$. 

Let $h$ be an orthogonal transformation of $\a$ that restricts to an isomorphism $h:\LL(-1)\xrightarrow{\sim}\LL'(-1)$. By our hypothesis that $\zz'\in\wh{G}\cap\wh{G}'$, we have $h\in \SO(\a)$. Since $\Gamma_\Pi$ is a primitive sublattice of $\LL(-1)\cap\LL'(-1)$ we may choose $h$ so that it restricts to the identity on $\Gamma_\Pi$. Then $h$ commutes with $g$, because $g$ acts trivially on $\Gamma_\Pi^\perp$. More than this, any lift $\wh{h}$ of $h$ to $\Spin(\a)$ commutes with $\wh{g}$, because we have $\wh{g}=\prod_{i=1}^{12}e^{\alpha_iX_i}$ (cf. (\ref{eqn:spin-whgexp})), for some basis $\{a_i^\pm\}$ of eigenvectors for $g$, as in (\ref{eqn:spin-lambdai}), with $X_i$ as in (\ref{eqn:spin-Xi}), and we may assume that $\alpha_i\neq 0$ only when $a_i^\pm\in \Gamma_\Pi\otimes_\ZZ\CC$. Then $\wh{h}X_i=X_i\wh{h}$ whenever $\alpha_i\neq 0$, and so $\wh{h}\wh{g}=\wh{g}\wh{h}$. Now $\wh{h}\wh{G}\wh{h}^{-1}$ is a lift of $G'$ to $\Spin(\a)$, so it must be $\wh{G}'$ by Proposition \ref{prop:spin-whG}. So $\wh{h}\wh{g}\wh{h}^{-1}$ is the lift $\wh{g}'$ of $hgh^{-1}=g$ to $\wh{G}'$, so $\wh{g}'=\wh{h}\wh{g}\wh{h}^{-1}=\wh{g}$, as we required to show.
\end{proof}

The coefficients of the $\phi_g$ may be computed explicitly, in direct analogy with (\ref{eqn:dist-Tsgtwdir}). With this purpose in mind we define constants $D_{{g}}$ as follows.

Given $g\in G_\Pi$, choose a polarization $\a=\a^-\oplus \a^+$ for $\a=\mlc$ such that $\a^{\pm}$ is spanned by (isotropic) eigenvectors $a^{\pm}_i$ for $g$, constituting a pair of dual bases in the sense that $\lab a^-_i,a^+_j\rab=\delta_{i,j}$. (Cf. the discussion in \S\ref{sec:dist}.) We may assume that 
\begin{gather}\label{eqn:tg-a11a12}
a_{11}^\pm =a_X^\pm,\quad 
a_{12}^\pm=a_Z^\pm,
\end{gather}
since $a_X^\pm$ and $a_Z^\pm$ are fixed by $g$, by hypothesis. We may also assume that the lift of $-\Id_\a$ associated to the polarization $\a=\a^-\oplus\a^+$coincides with $\zz'$ (cf. (\ref{eqn:tg-orientmlr})), for if not, then replace $a_i^\pm $ with $a_i^\mp$, for some $i\in\{1,\ldots,10\}$.

Write $\lambda^{\pm 1}_i$ for the eigenvalue of $g$ attached to $a^{\pm}_i$. Set $X_i:=\frac{\ii}{2}(a_i^-a_i^+-a_i^+a_i^-)$ as in (\ref{eqn:spin-Xi}). Then, according to the discussion in \S\ref{sec:spin}, we have
\begin{gather}
	\wh{g}=\prod_{i=1}^{10}e^{\alpha_iX_i},
\end{gather}
for some $\alpha_i\in 2\pi\QQ$ satisfying $\lambda_i^{\pm 1}=e^{\pm 2 \alpha_i\ii }$.
We now set $\nu_i:=e^{\alpha_i \ii}$, for $1\leq i\leq 10$, and define
\begin{gather}\label{eqn:tg:Dg}
	D_{g}:=\prod_{i=1}^{10}(\nu_i-\nu_i^{-1}).
\end{gather}

Observe that if $\nu':=\prod_{i=1}^{10}\nu_i$ then $D_{g}=\nu'\prod_{i=1}^{10}(1-\lambda_i^{-1})$. So $D_{g}$ vanishes if and only if $g$ has a fixed point in its action on $\Gamma_\Pi$. In particular, $D_{g}$ vanishes whenever the sublattice of $\LL$ fixed by $g$ has rank larger than $4$.

We are now prepared to present an explicit expression for $\phi_g$.

\begin{proposition}\label{proposition:phi_g}
Let $X$ be a projective complex K3 surface and let $\sigma\in \Stabo(X)$. Then for $g\in G_\Pi=\Aut_s(\Db(X),\sigma)$ and $\phi_g$ defined by (\ref{eqn:tg-phig}), we have
\begin{align}\label{equation:phi_g}
\begin{split}
\phi_g
= & -\frac{1}{2} \left(\frac{\vartheta_4(\tau,z)^2}{\vartheta_4(\tau,0)^2}\frac{\eta_g(\tau/2)}{\eta_g(\tau)} - \frac{\vartheta_3(\tau,z)^2}{\vartheta_3(\tau,0)^2}\frac{\eta_{-g}(\tau/2)}{\eta_{-g}(\tau)}\right) \\
& + \frac{1}{2}\left(
\frac{\vartheta_1(\tau,z)^2}{\eta(\tau)^6}D_{g}\eta_g(\tau)
- 
\frac{\vartheta_2(\tau,z)^2}{\vartheta_2(\tau,0)^2}C_{-g}\eta_{-g}(\tau) 
\right)
\end{split}
\end{align}
after substituting $q=e^{2\pi \ii \tau}$ and $y=e^{2\pi \ii z}$. In particular, $\phi_g$ is the Fourier expansion of a holomorphic function $\phi_g(\tau,z)$ on $\HH\times \CC$, invariant under $(\tau,z)\mapsto(\tau+m,z+n)$, for $m,n\in\ZZ$.
\end{proposition}
\begin{proof}
The required identity (\ref{equation:phi_g}) may be obtained via direct calculation. We use the decomposition $\vsn=A(\a)^1\oplus A(\a)^0_\tw$ along with the formulas (\ref{eqn:dist-etag}) and (\ref{eqn:dist-Cg}). We also use Lemma \ref{lem:tg-JPirel}, and the product formulas (\ref{eqn:autfms-theta1prod}) for the Jacobi theta functions 
$\vartheta_i$.

For the contribution of $A(\a)^1$ to $\phi_g$ note that
\begin{gather}
\begin{split}
&\tr_{A(\a)}\wh{g}y^{J(0)}q^{L(0)-c/24}\\
&=
q^{-1/2}\prod_{n>0}(1+y^{-1}q^{n-1/2})^2(1+yq^{n-1/2})^2\prod_{i=1}^{10}(1+\lambda_i^{-1}q^{n-1/2})(1+\lambda_iq^{n-1/2})\\
&=\frac{\vartheta_3(\tau,z)^2}{\vartheta_3(\tau,0)^2}\frac{\eta_{-g}(\tau/2)}{\eta_{-g}(\tau)}
\end{split}
\end{gather}
since $\lambda_{11}=\lambda_{12}=1$ according to the convention (\ref{eqn:tg-a11a12}). Similarly, 
\begin{gather}
\tr_{A(\a)}\zz\wh{g}y^{J(0)}q^{L(0)-c/24}
=\frac{\vartheta_4(\tau,z)^2}{\vartheta_4(\tau,0)^2}\frac{\eta_{-g}(\tau/2)}{\eta_{-g}(\tau)}.
\end{gather}
Thus, recalling the definition (\ref{eqn:tg-phig}) of $\phi_g$, we see that the first line of the right hand side of (\ref{equation:phi_g}) is precisely the contribution of $A(\a)^1$ to $\phi_g$.

The contribution of $A(\a)^0_\tw$ is computed similarly. We have
\begin{gather}
\begin{split}
&\tr_{A(\a)}\wh{g}y^{J(0)}q^{L(0)-c/24}\\
&=
qy\nu\prod_{n>0}(1+y^{-1}q^{n-1})^2(1+yq^{n})^2\prod_{i=1}^{10}(1+\lambda_i^{-1}q^{n-1})(1+\lambda_iq^{n})\\
&=\frac{\vartheta_2(\tau,z)^2}{\vartheta_2(\tau,0)^2}C_{-g}{\eta_{-g}(\tau)},
\end{split}\\
\begin{split}
&\tr_{A(\a)}\zz\wh{g}y^{J(0)}q^{L(0)-c/24}\\
&=
qy\nu\prod_{n>0}(1-y^{-1}q^{n-1})^2(1-yq^{n})^2\prod_{i=1}^{10}(1-\lambda_i^{-1}q^{n-1})(1-\lambda_iq^{n})\\
&=-\frac{\vartheta_1(\tau,z)^2}{\eta(\tau)^6}D_{g}{\eta_{g}(\tau)},
\end{split}
\end{gather}
assuming, as we may, that $\nu'=\nu$ in (\ref{eqn:tg:Dg}). This shows that the second line of the right hand side of (\ref{equation:phi_g}) represents the contribution of $A(\a)^0_\tw$ to $\phi_g$.
The identity is proved.
\end{proof}

Armed with Proposition \ref{proposition:phi_g}, we henceforth regard $\phi_g=\phi_g(\tau,z)$ as a holomorphic function on $\HH\times \CC$. We would like to show that $\phi_g$ is a weak Jacobi form. This is accomplished by giving an expression in terms of the standard weak Jacobi forms $\phi_{0,1}$ and $\phi_{-2,1}$ (cf. \S\ref{sec:autfms}). With this in mind, define 
\begin{gather}\label{equation:F_g}
	\begin{split}
F_g(\tau):= & 
\frac{1}{2}\Lambda_2(\tau/2)\frac{\eta_{g}(\tau/2)}{\eta_g(\tau)} 
-\frac{1}{2}\Lambda_2(\tau/2+1/2)\frac{\eta_{-g}(\tau/2)}{\eta_{-g}(\tau)}
\\
& 
+\frac{1}{2}D_{g}\eta_{g}(\tau)
- \Lambda_2(\tau)C_{-g}\eta_{-g}(\tau)
	\end{split}
\end{gather}
for $g$ as in Proposition \ref{proposition:phi_g}, and recall the definition (\ref{eqn:dist-chig}) of $\chi_g$.

\begin{proposition}\label{proposition:phi_g and F_g}
Let $X$ be a projective complex K3 surface and let $\sigma\in \Stabo(X)$. Then for $g\in \Aut_s(\Db(X),\sigma)$ we have
\be\label{equation:phi_g and F_g}
\phi_g(\tau,z)=\frac1{12}
	\chi_g\phi_{0,1}(\tau,z)+F_g(\tau)\phi_{-2,1}(\tau,z).
\ee
\end{proposition}
\begin{proof}
Replacing $F_g(\tau)$ with the right hand side of (\ref{equation:F_g}), the right hand side of (\ref{equation:phi_g and F_g}) becomes
\begin{gather}
\begin{split}
\frac1{12}
\chi_g\phi_{0,1}(\tau,z)&
+\frac{1}{2}\Lambda_2(\tau/2)\frac{\eta_{g}(\tau/2)}{\eta_g(\tau)}\phi_{-2,1}(\tau,z) 
-\frac{1}{2}\Lambda_2(\tau/2+1/2)\frac{\eta_{-g}(\tau/2)}{\eta_{-g}(\tau)}\phi_{-2,1}(\tau,z)
\\
& 
+\frac{1}{2}D_{g}\eta_{g}(\tau)\phi_{-2,1}(\tau,z)
-\Lambda_2(\tau)C_{-g}\eta_{-g}(\tau)\phi_{-2,1}(\tau,z)
.
\end{split}
\end{gather}
Now subtract $\frac{1}{24}0\phi_{0,1}$, where $0$ is written as in Lemma \ref{lemma:eta identity} (i.e. the left hand side of (\ref{equation:eta identity})). After some rearrangement we obtain
\begin{gather}
	\begin{split}
&
-\left(\frac{1}{24}\phi_{0,1}(\tau,z)-\frac{1}{2}\Lambda_2(\tau/2)\phi_{-2,1}(\tau,z)\right)\frac{\eta_{g}(\tau/2)}{\eta_g(\tau)}
	\\
&
+\left(\frac{1}{24}\phi_{0,1}(\tau,z)-\frac{1}{2}\Lambda_2(\tau/2+1/2)\phi_{-2,1}(\tau,z)\right)\frac{\eta_{-g}(\tau/2)}{\eta_{-g}(\tau)}
	\\
&
+\frac{1}{2}\phi_{-2,1}(\tau,z)D_{g}\eta_{g}(\tau)
-\left(\frac{1}{24}\phi_{0,1}(\tau,z)+\Lambda_2(\tau)\phi_{-2,1}(\tau,z)\right)C_{-g}\eta_{-g}(\tau)
,
	\end{split}
\end{gather}
and the identities of Lemma \ref{lemma:theta function identities} show that this is exactly (\ref{equation:phi_g}).
\end{proof}

Applying Proposition \ref{prop:autfms-phiCFs} with $m=1$ to (\ref{equation:phi_g and F_g}) we see that $\phi_g$ is a weak Jacobi form of level $N$ so long as $F_g$ is a modular form of weight 2 for $\Gamma_0(N)$.

\begin{proposition}\label{proposition:F_g modular}
Let $X$ be a projective complex K3 surface and let $\sigma\in \Stabo(X)$. Then for $g\in \Aut_s(\Db(X),\sigma)$ the function $F_g$ is a modular form of weight $2$ for $\Gamma_0(N_g)$, for some positive integer $N_g$.
\end{proposition}
\begin{proof}
As in the definition (\ref{eqn:tg-phig}) of $\phi_g$, and the proof of Proposition \ref{prop:tg-phigindep}, we use Proposition \ref{prop:de:aetohi} to identify $G_\Pi=\Aut_s(\Db(X),\sigma)$ with a subgroup of $\SO(\a)$ (recall that $\a=\mlc$), and we choose a Leech marking $\iota$ for $(X,\sigma)$, in order to identify $g$ as an element of $G=\Aut(\LL(-1))$, for a suitable copy of $\LL(-1)$ in $\a$. Then $\eta_{\pm g}(\tau)$ and $C_{-g}$ depend only on the conjugacy class of $g$ in $G\simeq \Co_0$, and $D_g$ is determined by the conjugacy class $[g]\subset G$ up to sign (cf. (\ref{eqn:tg:Dg})). The values $C_{-g}$ and $D_{g}$, and the Frame shapes $\pi_{\pm g}$ that determine the $\eta_{\pm g}(\tau)$ (cf. (\ref{eqn:spin-pig})) may be read off from Table \ref{table:4space}. 

Now the proof is essentially a case by case check of the relevant classes of $\Co_0$ (rather than, say, all the groups $G_{\Pi}$ in $\Aut(\mlz)=O(\II_{4,20})$), but we can use the results of \cite{vacogm} to simplify this further, replacing the explicit calculation of modular forms with simple checks on properties of the invariance groups $\Gamma_g$ of the functions $T^s_g(2\tau)$ (cf. (\ref{eqn:dist-Tsg}), Theorem \ref{thm:dist-Tsgpm}).

As a first step towards this goal, observe that if $D_{g}\neq 0$ then $\eta_g(\tau)$ is an eta product of weight $2$, meaning that $\sum_{m>0} k_m=4$ for $\pi_g=\prod_{m>0}m^{k_m}$. Indeed, $\sum_{m>0}k_m$ is exactly the rank of $\Lambda^g$, and it was pointed out in the sentence following (\ref{eqn:tg:Dg}) that $D_g$ vanishes when $\Lambda^g$ has rank larger than $4$.

It follows that the third summand in the definition (\ref{equation:F_g}) of $F_g(\tau)$ is a modular form of weight $2$, and some level, for each $g$. So we may consider $F'_g(\tau):=F_g(\tau)-\frac{1}{2}D_{g}\eta_g(\tau)$. (The prime here does not denote differentiation.) 

Next we apply Lemma \ref{lemma:eta identity} to rewrite $C_{-g}\eta_{-g}(\tau)$ in terms of the functions 
\begin{gather}\label{eqn:tg-tpmg}
	t_{\pm g}(\tau):=\frac{\eta_{\pm g}(\tau)}{\eta_{\pm g}(2\tau)}.
\end{gather}
(Cf. (\ref{eqn:dist-Tsgchig}).) Since $g$ has fixed points in $\a$ by hypothesis, $C_g=0$ (cf. (\ref{eqn:spin-Cgnu})). Thus we obtain
\begin{gather}\label{eqn:tg-etaidCgzero}
2\chi_g
+t_g\left(\frac{\tau}{2}\right)
-t_{-g}\left(\frac{\tau}{2}\right)
+C_{{-g}}\eta_{-g}(\tau)
=0
\end{gather}
from Lemma \ref{lemma:eta identity}. 
Solving for $C_{-g}\eta_{-g}(\tau)$ in (\ref{eqn:tg-etaidCgzero}), substituting the result into (\ref{equation:F_g}), and noting the identities $t_g(\tau+1/2)=-t_{-g}(\tau)$ and $\Lambda_4(\tau)=4\Lambda_2(2\tau)+2\Lambda_2(\tau)$, we see that if $F''_g(\tau):=F'_g(\tau)-2\chi_g\Lambda_2(\tau)$, then
\begin{gather}\label{eqn:tg-Fgpp}
	F''_g(\tau)=
	\frac14\left(
	t_{g}\left(\frac{\tau}{2}\right)
	\Lambda_4\left(\frac{\tau}2\right)
	+
	t_{g}\left(\frac{\tau+1}{2}\right)
	\Lambda_4\left(\frac{\tau+1}2\right)\right).
\end{gather}
The function $2\chi_g\Lambda_2(\tau)$ is a modular form for $\Gamma_0(M)$ whenever $M$ is even, so we may focus on $F_g''$. 

Set $f_g(\tau):=\frac12 t_g(\tau)\Lambda_4(\tau)$. Then $f_g(\tau)$ is a weakly holomorphic modular form of weight $2$ and some level, since $t_g(\tau)$ is a principal modulus for a genus zero group containing some $\Gamma_0(N)$, according to (\ref{eqn:dist-Tsgchig})  and Theorem \ref{thm:dist-Tsgpm}. Precisely, for $\Gamma_g$ the invariance group of $t_g(\tau)$ (i.e., as in Theorem \ref{thm:dist-Tsgpm}), the function $f_g$ is a weakly holomorphic modular form of weight $2$ for $\Gamma_g\cap \Gamma_0(4)$. 

The relevance of $f_g$, apparent from (\ref{eqn:tg-Fgpp}), is that we have $F_g''=T_2(2)f_g$, where $T_k(2)$ denotes the second order Hecke operator on modular forms of weight $k$ for $\Gamma_0(M)$, for any even $M$,
\begin{gather}\label{eqn:tg-Tk2}
(T_k(2)f)(\tau):=\frac{1}{2}\left(f\left(\frac{\tau}{2}\right)+f\left(\frac{\tau+1}{2}\right)\right).
\end{gather}
(Cf. e.g. \S IX.6 of \cite{MR1193029}.)
From this we conclude that $F''_g$ is a weakly holomorphic modular form of weight $2$ for $\Gamma_0(N')$, where $N'$ is the least common multiple of $4$ and the level of $t_g$. 

We require to show that $F''_g$ is actually a (holomorphic) modular form, i.e. has no poles at cusps. There is no loss in considering
\begin{gather}\label{eqn:tg-4Fgpp2}
	4F''_g(2\tau)=
	t_{g}(\tau)
	\Lambda_4({\tau})
	-
	t_{-g}({\tau})
	\Lambda_4\left(\tau+\frac{1}2\right)
\end{gather}
instead. Note that $\Lambda_4(\tau)$ and $\Lambda_4(\tau+1/2)$ are both modular forms of weight $2$ for $\Gamma_0(4)$, since $\Gamma_0(4)$ is normalized by $\left(\begin{smallmatrix}1&1/2\\0&1\end{smallmatrix}\right)$. Recall that $\Gamma_0(4)$ has three orbits on $\widehat{\QQ}=\QQ\cup\{\infty\}$, represented by $1$, $1/2$ and $1/4$. (The infinite cusp is represented by $1/4$.) Table \ref{tab:Lambda4cusps} presents the asymptotic behavior of $\Lambda_4(\tau)$ and $\Lambda_4(\tau+1/2)$ at these three cusps of $\Gamma_0(4)$.
\begin{table}[ht]
\centering
\caption{Modular forms of level four at cusps}
\begin{tabular}{ccc}\label{tab:Lambda4cusps}
&$\Lambda_4(\tau)$&$\Lambda_4(\tau+1/2)$\\
\hline
$1$&$1/2+O(q)$&$1/2+O(q)$\\
$1/2$&$O(q)$&$-1/2+O(q)$\\
$1/4$&$-1/8+O(q^{1/4})$&$O(q^{1/4})$
\end{tabular}
\end{table}

The function $F''_g$ cannot have a pole at $\alpha\in \widehat{\QQ}$ unless one or both of $t_{\pm g}$ do. The functions $t_{\pm g}$ are principal moduli for groups $\Gamma_{\pm g}$ according to Theorem \ref{thm:dist-Tsgpm}, so they can only have poles at points $\alpha\in\widehat{\QQ}$ such that $\alpha\in\Gamma_{\pm g}\infty$. Comparing (\ref{eqn:tg-4Fgpp2}) with Table \ref{tab:Lambda4cusps} we see that our task has been reduced to verifying, for arbitrary $\alpha\in\widehat{\QQ}$, that
\begin{enumerate}
\item
if $\alpha\in\Gamma_g\infty$ and $\Gamma_{-g}\infty$ then $\alpha=1/4\mod \Gamma_0(4)$,
\item
if $\alpha\in \Gamma_g\infty$ and $\alpha\notin\Gamma_{-g}\infty$ then $\alpha=1/2\mod \Gamma_0(4)$, and
\item
if $\alpha\notin \Gamma_{g}\infty$ and $\alpha\in\Gamma_{-g}\infty$ then $\alpha=1\mod \Gamma_0(4)$.
\end{enumerate}

The verification of these statements can now be handled directly using the descriptions of the groups $\Gamma_{\pm g}$ appearing in \S\ref{sec:tables-comps}. 

Observe that the verification of the statements 1, 2 and 3 is generally quite easy. For example, if $C_{-g}=0$ (which is the case for most conjugacy classes) then we necessarily have $\Gamma_g=\Gamma_{-g}$, since (\ref{eqn:tg-etaidCgzero}) implies that $t_g$ and $t_{-g}$ coincide, up to an additive constant. Then the conditions 2 and 3 become vacuous, and we require only to check that if $\gamma\in \Gamma_g$ and $\gamma\infty=\frac ac$ for $a,c\in\ZZ$ with $(a,c)=1$, then $c=0\mod 4$. The remaining cases are handled similarly.
\end{proof}

Together, Propositions \ref{proposition:phi_g and F_g} and \ref{proposition:F_g modular} prove our main theorem.

\begin{theorem}\label{theorem:jacobi form}
Let $X$ be a projective complex K3 surface and let $\sigma\in \Stabo(X)$. Then for $g\in \Aut_s(\Db(X),\sigma)$ the function $\phi_g$ is a Jacobi form of weight $0$, index $1$, and some level.
\end{theorem}

Taking $g$ to be the identity in Theorem \ref{theorem:jacobi form} produces a Jacobi form $\phi_e$ of weight 0, index 1, and level 1; it must be the K3 elliptic genus, up to a constant. The constant can be determined by setting $z=0$. Taking $z=0$ in (\ref{eqn:autfms:phi01}) and (\ref{eqn:autfms:phi-21}), and applying (\ref{equation:phi_g and F_g}),
we see that in fact it is exactly the K3 elliptic genus, but expressed in a rather non-standard way:
\be\label{eqn:tg-K3EG}
Z_{K3}(\tau,z)=\frac{1}{2}\frac{\vartheta_3(\tau,z)^2}{\vartheta_3(\tau,0)^2}\frac{\Delta(\tau)^2}{\Delta(2\tau)\Delta(\tau/2)}-\frac{1}{2}\frac{\vartheta_4(\tau,z)^2}{\vartheta_4(\tau,0)^2}\frac{\Delta(\tau/2)}{\Delta(\tau)}-2^{11}\frac{\vartheta_2(\tau,z)^2}{\vartheta_2(\tau,0)^2}\frac{\Delta(2\tau)}{\Delta(\tau)}.
\ee

Explicit expressions for the $\phi_g$ are recorded in Table \ref{table:4space}. Coincidences with the weight zero weak Jacobi forms of Mathieu moonshine, and with the K3 sigma model twining genera computed in \cite{GHV} are recorded in Table \ref{table:4spacecoin}. Observe that every twining genus appearing in \cite{GHV} appears also in Table \ref{table:4spacecoin}. This leads us to the following conjecture.

\begin{conjecture}\label{conj:sigmod}
The twined elliptic genus attached to a supersymmetry preserving automorphism $g\in G_\Pi$ of the supersymmetric non-linear K3 sigma model determined by $\Pi=P_X\oplus P_Z$ coincides with $\phi_g$.
\end{conjecture}

\section{Umbral Moonshine}\label{sec:um}

In addition to twined K3 elliptic genera, a number of which coincide with weak Jacobi forms of Mathieu moonshine (cf. Table \ref{table:4spacecoin}), graded traces on $\VV$ defined by $\U(1)$ elements corresponding to higher dimensional subspaces of $\a$ recover functions arising from umbral moonshine, as we will now explain. 

In \cite{UM}, to each lambency $\ell\in\{2,3,4,5,7,13\}$ is associated a Jacobi form $Z^{(\ell)}$ of weight 0 and index $\ell-1$. In \cite{UMNL} this is expanded to a correspondence associating to each Niemeier lattice with root system $X$ a lambency $\ell$ and a meromorphic Jacobi form $\psi^X$ of weight 1, with index given by the Coxeter number of $X$. For lambencies $\ell$ occurring in \cite{UM}, we recover $Z^{(\ell)}$ from $\psi^X$ according to the rule
\begin{gather}
	Z^{(\ell)}(\tau,z)=-\ii\frac{\vartheta_1(\tau,z)^2}{\vartheta_1(\tau,2z)\eta(\tau)^3}\psi^X(\tau,z),
\end{gather}
upon taking $X=A_{\ell-1}^{n}$ for $n=24/(\ell-1)$.

For $d\in\{2,4,6,8,10,12\}$ define a corresponding $\ell$ by $\ell=\frac{d}{2}+1$. Identify $\a=\LL\otimes_\ZZ\CC$. Choose a $2d$-dimensional real vector space $\Pi<\LL\otimes_\ZZ\RR\subset\a$ and let $\{a^\pm_i\}$ be bases for isotropic subspaces $\a^\pm<\a$ constituting a polarization $\a=\a^-\oplus\a^+$. Assume that $\lab a_i^-,a_j^+\rab=\delta_{i,j}$ and 
\begin{gather}
\Pi\otimes_\RR\CC=
\Span\left\{a_i^-,a_i^+\mid 1\leq i\leq d\right\}.
\end{gather}
Define an associated $\U(1)$ element $\jmath\in \vsn$, with $\vsn$ as in (\ref{eqn:dist-vsnvsnt}), realized using $\a=\a^-\oplus\a^+$, by setting
\begin{gather}
\jmath:=\frac12\sum_{i=1}^{d} a_i^-(-1/2)a_i^+(-1/2)\vv.
\end{gather}
Then $\jmath$ has level $d$ according to Lemma \ref{lem:tg-JPirel}.

Recall from \S\ref{sec:spin} that $g\mapsto \wh{g}$ denotes the natural isomorphism relating $G=\Aut(\LL)$ to the naturally corresponding copy $\wh{G}$ of $\Co_0$ in $\Spin(\a)$. If $g\in G$ is chosen so that $g$ restricts to the identity on $\Pi$ then the action of $\wh{g}$ on $\vsnt$ commutes with that of $J(0):=\jmath_{(0),\tw}$, and we may define
\begin{gather}
	\phi_g^{(\ell)}(\tau,z):=-\tr_{\vsnt}\zz\wh{g}y^{J(0)}q^{L(0)-c/24}.
\end{gather}
Cf. (\ref{eqn:tg-phig}). 

To describe the $\phi_g^{(\ell)}$ explicitly choose a $g$ in $G$ with $g|_{\Pi}=\Id$. Assume that the $a_i^\pm$ are eigenvectors for $g$ and write 
\begin{gather}
	\wh{g}=\prod_{i=1}^{10}e^{\alpha_iX_i},
\end{gather}
where $X_i=\frac{\ii}{2}(a_i^-a_i^+-a_i^+a_i^-)$ as in (\ref{eqn:spin-Xi}), and the $\alpha_i\in 2\pi\QQ$ are chosen so that $\lambda_i^{\pm 1}=e^{\pm 2 \alpha_i\ii }$ is the eigenvalue for the action of $g$ on $a_i^\pm$. We may assume that $\alpha_i=0$ for $1\leq i\leq d$ since $g$ fixes the corresponding $a^\pm_i$ by hypothesis. 
Set $\nu_i:=e^{\alpha_i \ii}$, and define
\begin{gather}\label{eqn:um:Dellg}
	D^{(\ell)}_{g}:=\prod_{i=d+1}^{12}(\nu_i-\nu_i^{-1}).
\end{gather}
Similar to the discussion around (\ref{eqn:tg:Dg}), if $\nu':=\prod_{i=d+1}^{12}\nu_i$ then
\begin{gather}
D^{(\ell)}_{g}=\nu'\prod_{i=d+1}^{12}(1-\lambda_i^{-1}),
\end{gather}
so 
$D^{(\ell)}_{g}$ vanishes if and only if $g$ has a fixed point in its action on the orthogonal complement of $\Pi$ in $\LL$. 
In particular, $D^{(\ell)}_{g}$ vanishes whenever the sublattice of $\LL$ fixed by $g$ has rank larger than $2d$.

By a method directly similar to the proof of Proposition \ref{proposition:phi_g}, we obtain the following explicit expression for $\phi^{(\ell)}_g$.

\begin{proposition}\label{proposition:phi_g^ell}
Let $d\in\{2,4,6,8,10,12\}$ and $\ell=\frac{d}{2}+1$. Let $\Pi$ be a $2d$-dimensional subspace of $\LL\otimes_\ZZ\RR$. If $g\in G$ and $g|_\Pi=\Id$ then
\begin{align}
\begin{split}\label{eqn:phi_g^ell}
\phi_g^{(\ell)}(\tau,z)
= & -\frac{1}{2} \left(\frac{\vartheta_4(\tau,z)^d}{\vartheta_4(\tau,0)^d}\frac{\eta_g(\tau/2)}{\eta_g(\tau)} - \frac{\vartheta_3(\tau,z)^d}{\vartheta_3(\tau,0)^d}\frac{\eta_{-g}(\tau/2)}{\eta_{-g}(\tau)}\right) \\
& + \frac{1}{2}\left((-1)^\ell
\frac{\vartheta_1(\tau,z)^d}{\eta(\tau)^{3d}}D_{g}^{(\ell)}\eta_g(\tau)
- 
\frac{\vartheta_2(\tau,z)^d}{\vartheta_2(\tau,0)^d}C_{-g}\eta_{-g}(\tau) 
\right)
\end{split}
\end{align}
after substituting $q=e^{2\pi \ii \tau}$ and $y=e^{2\pi \ii z}$. In particular, $\phi_g^{(\ell)}$ is the Fourier expansion of a holomorphic function $\phi_g^{(\ell)}(\tau,z)$ on $\HH\times \CC$, invariant under $(\tau,z)\mapsto(\tau+m,z+n)$, for $m,n\in\ZZ$.
\end{proposition}

On the strength of Proposition \ref{proposition:phi_g^ell} we may henceforth regard $\phi_g^{(\ell)}=\phi_g^{(\ell)}(\tau,z)$ as a holomorphic function on $\HH\times \CC$. As in \S\ref{sec:tg}, we would like to show that $\phi_g^{(\ell)}$ is a weak Jacobi form, and this is accomplished by giving an expression in terms of the standard weak Jacobi forms $\phi_{0,1}$ and $\phi_{-2,1}$ (cf. \S\ref{sec:autfms}), and some particular modular forms depending on $g$. With this in mind, define 
\begin{gather}\label{equation:F_2k,g}
	\begin{split}
F_{2j,g}(\tau):=-\Lambda_2(\tau/2)^j\frac{\eta_g(\tau/2)}{\eta_g(\tau)}+\Lambda_2(\tau/2+1/2)^j\frac{\eta_{-g}(\tau/2)}{\eta_{-g}(\tau)}-
(-2\Lambda_2(\tau))^jC_{-g}\eta_{-g}(\tau)
	\end{split}
\end{gather}
for $j\ge0$ and $g$ as in Proposition \ref{proposition:phi_g^ell}.

\begin{proposition}\label{proposition:F_2k,g modular}
Let $\ell$ and $g$ be as in Proposition \ref{proposition:phi_g^ell}. Then $F_{0,g}$ is constant, and there exists a positive integer $N$ such that $F_{2j,g}\in M_{2j}(\Gamma_0(N))$ for $0<j<\ell$.
\end{proposition}
\begin{proof}
The proof is very similar to that of Proposition \ref{proposition:F_g modular}.
That is to say, it is ultimately a case by case check, but we use the results of \cite{vacogm} to replace the explicit calculation of modular forms with simple checks on properties of the invariance groups $\Gamma_g$ of the functions $T^s_g(2\tau)$ (cf. (\ref{eqn:dist-Tsg}), Theorem \ref{thm:dist-Tsgpm}).

To begin, consider $F_{0,g}$. With $t_g$ as defined by (\ref{eqn:tg-tpmg}) we have 
\begin{gather}\label{eqn:um-F0gtpmg}
F_{0,g}(\tau)=-t_g(\tau/2)+t_{-g}(\tau/2)-C_{-g}\eta_{-g}(\tau).
\end{gather} 
Note that $g$ has fixed points in its action on $\a$, by hypothesis, so $C_g=0$. So (\ref{eqn:tg-etaidCgzero}) holds, according to Lemma \ref{lemma:eta identity}. Comparing (\ref{eqn:tg-etaidCgzero}) to (\ref{eqn:um-F0gtpmg}) we see that $F_{0,g}(\tau)=2\chi_g$. In particular, $F_{0,g}$ is constant, as required.

Now let $0<j<\ell$. Applying Lemma \ref{lemma:eta identity} to rewrite $C_{-g}\eta_{-g}(\tau)$ in terms of $t_{\pm g}$, and also using $t_g(\tau+1/2)=-t_{-g}(\tau)$, we have $F_{2j,g}(\tau)=F_{2j,g}'(\tau)+2\chi_g(-2\Lambda_2(\tau))^j$, where
\begin{gather}
	F'_{2j,g}(\tau):=t_g\left(\frac{\tau}2\right)G_{2j}\left(\frac\tau 2\right)+t_g\left(\frac{\tau+1}2\right)G_{2j}\left(\frac{\tau+1}2\right),\\
	G_{2j}(\tau):=-\Lambda_2(\tau)^j+(-2\Lambda_2(2\tau))^j.
\end{gather}
We require to show that $F'_{2j,g}(\tau)$ is a modular form for some $\Gamma_0(N)$.

So set 
\begin{gather}
f_{2j,g}(\tau):=2 t_g(\tau)G_{2j}(\tau)
\end{gather}
and observe that $f_{2j,g}(\tau)$ is a weakly holomorphic modular form of weight $2j$ for some $\Gamma_0(N)$, since the invariance group of $t_g$ contains some $\Gamma_0(N)$. Also, $F_{2j,g}'=T_{2j}(2)f_{2j,g}$, where $T_k(2)$ denotes the second order Hecke operator on modular forms of weight $k$ for $\Gamma_0(M)$, for any even $M$. (Cf. (\ref{eqn:tg-Tk2}).)
So $F'_{2j,g}$ is a weakly holomorphic modular form of weight $2j$ for some $\Gamma_0(N)$, and it remains to verify that $F'_{2j,g}$ has no poles at cusps.

To this end, observe
\begin{gather}\label{eqn:tg-F2jgp2}
	F'_g(2\tau)=
	t_{g}(\tau)
	G_{2j}({\tau})
	-
	t_{-g}({\tau})
	G_{2j}\left(\tau+\frac{1}2\right),
\end{gather}
and note that both $G_{2j}(\tau)$ and $G_{2j}(\tau+1/2)$ are modular forms of weight $2j$ for $\Gamma_0(4)$. Table \ref{tab:G2jcusps} is the appropriate analogue of Table \ref{tab:Lambda4cusps}, presenting the asymptotic behavior of $G_{2j}(\tau)$ and $G_{2j}(\tau+1/2)$ at the three cusps of $\Gamma_0(4)$.
\begin{table}[ht]
\centering
\caption{Modular forms of level four at cusps}
\begin{tabular}{ccc}\label{tab:G2jcusps}
&$G_{2j}(\tau)$&$G_{2j}(\tau+1/2)$\\
\hline
$1$&$((-2)^j-1){12^{-j}}+O(q)$&${((-2)^j-1)}{12^{-j}}+O(q)$\\
$1/2$&$O(q)$&$(1-(-2)^j)12^{-j}+O(q)$\\
$1/4$&$(1-(-2)^j)48^{-j}+O(q^{1/4})$&$O(q^{1/4})$
\end{tabular}
\end{table}

Just as for $F''_g$ in the proof of Proposition \ref{proposition:F_g modular}, the function $F'_{2j,g}$ cannot have a pole at $\alpha\in \widehat{\QQ}$ unless one or both of $t_{\pm g}$ do, and the $t_{\pm g}$ can only have poles at points $\alpha\in\widehat{\QQ}$ such that $\alpha\in\Gamma_{\pm g}\infty$. Comparing (\ref{eqn:tg-F2jgp2}) with Table \ref{tab:G2jcusps} we see that we require to verify, for arbitrary $\alpha\in\widehat{\QQ}$, that
\begin{enumerate}
\item
if $\alpha\in\Gamma_g\infty$ and $\Gamma_{-g}\infty$ then $\alpha=1/4\mod \Gamma_0(4)$,
\item
if $\alpha\in \Gamma_g\infty$ and $\alpha\notin\Gamma_{-g}\infty$ then $\alpha=1/2\mod \Gamma_0(4)$, and
\item
if $\alpha\notin \Gamma_{g}\infty$ and $\alpha\in\Gamma_{-g}\infty$ then $\alpha=1\mod \Gamma_0(4)$.
\end{enumerate}
But these statements have been verified already, for any $g\in G$ such that the rank of $\LL^g$ is at least $4$, in the course of the proof of Proposition \ref{proposition:F_g modular}. Since $\Pi$ has dimension $2d=4(\ell-1)\geq 4$, the proof of the proposition is complete.
\end{proof}

\begin{proposition}\label{proposition:phi_g^ell and F_2k,g}
Let $\ell$ and $g$ be as in Proposition \ref{proposition:phi_g^ell}. Then we have
\be\label{equation:phi_g and F_2k,g}
	\begin{split}
\phi_g^{(\ell)}(\tau,z)
=&-\frac12D_g^{(\ell)}\eta_g(\tau)\phi_{-2,1}(\tau,z)^{\ell-1}\\
&+\sum_{j=0}^{\ell-1}
(-1)^j\frac12\binom{\ell-1}{j}\frac1{12^{\ell-j-1}}
F_{2j,g}(\tau)\phi_{0,1}(\tau,z)^{\ell-j-1}\phi_{-2,1}(\tau,z)^j.
	\end{split}
\ee
\end{proposition}
\begin{proof}
Recall $d=2(\ell-1)$ and apply the identities of Lemma \ref{lemma:theta function identities} to the expression \eqref{eqn:phi_g^ell} for $\phi_g^{(\ell)}$, replacing the theta quotients with the left-hand sides of \eqref{equation:theta1 identity}-\eqref{equation:theta4 identity}, to find
\begin{align}
\begin{split}
\phi_g^{(\ell)}(\tau,z)
=  -\frac{1}{2} \left(\vphantom{\left(\frac1{12}\right)^{\ell-1}}\right. & 
\left
(\frac{1}{12}\phi_{0,1}(\tau,z)-\Lambda_2(\tau/2)\phi_{-2,1}(\tau,z)\right
)^{\ell-1}\frac{\eta_g(\tau/2)}{\eta_g(\tau)}
\\ 
& - \left
(\frac{1}{12}\phi_{0,1}(\tau,z)-\Lambda_2(\tau/2+1/2)\phi_{-2,1}(\tau,z)\right
)^{\ell-1}\frac{\eta_{-g}(\tau/2)}{\eta_{-g}(\tau)}
\left.\vphantom{\left(\frac1{12}\right)^{\ell-1}}\right) \\
 + \frac{1}{2}\left(\vphantom{\left(\frac1{12}\right)^{\ell-1}}\right. & 
-\phi_{-2,1}(\tau,z)^{\ell-1}D_{g}^{(\ell)}\eta_g(\tau)
\\
& - 
\left
(\frac{1}{12}\phi_{0,1}(\tau,z)+2\Lambda_2(\tau)\phi_{-2,1}(\tau,z)\right
)^{\ell-1}C_{-g}\eta_{-g}(\tau) 
\left.\vphantom{\left(\frac1{12}\right)^{\ell-1}}\right).
\end{split}
\end{align}
Applying the binomial theorem to each term, we find that the coefficient of $\phi_{0,1}^{\ell-j-1}\phi_{-2,1}^j$ for $j<\ell-1$ in the resulting expansion is
\begin{align}
\begin{split}
&-\frac12\binom{\ell-1}{j}\frac1{12^{\ell-j-1}}(-\Lambda_2(\tau/2))^j\frac{\eta_g(\tau/2)}{\eta_g(\tau)}
\\
& \hspace{2cm} +\frac12\binom{\ell-1}{j}\frac1{12^{\ell-j-1}}(-\Lambda_2(\tau/2+1/2))^j\frac{\eta_{-g}(\tau/2)}{\eta_{-g}(\tau)}
\\
& \hspace{4cm} -\frac12\binom{\ell-1}{j}\frac1{12^{\ell-j-1}}(2\Lambda_2(\tau))^jC_{-g}\eta_{-g}(\tau) \\
&=(-1)^j\frac12\binom{\ell-1}{j}\frac1{12^{\ell-j-1}}F_{2j,g}(\tau).
\end{split}
\end{align}
The coefficient of $\phi_{-2,1}^{\ell-1}$ is similar, with the additional term $-\frac12D_g^{(\ell)}\eta_g(\tau)$.
\end{proof}

\begin{theorem}\label{thm:um-wkjac}
Let $d\in\{2,4,6,8,10,12\}$ and $\ell=\frac{d}{2}+1$. Let $\Pi$ be a $2d$-dimensional subspace of $\LL\otimes_\ZZ\RR$. If $g\in G$ and $g|_\Pi=\Id$ then $\phi^{(\ell)}_g$ is a weak Jacobi form of weight $0$ and index $\ell-1$, with some level depending on $g$.
\end{theorem}
\begin{proof}
Proposition \ref{proposition:phi_g^ell and F_2k,g} shows that $\phi^{(\ell)}_g$ is a homogeneous polynomial in $\phi_{0,1}$ and $\phi_{-2,1}$ of the form required by Proposition \ref{prop:autfms-phiCFs}. Proposition \ref{proposition:F_2k,g modular} verifies that all the coefficients in this expression have the correct modular properties, except for $-\frac12D^{(\ell)}_g\eta_g(\tau)$, which appears as the coefficient of $\phi_{-2,1}^{\ell-1}$. So the required result follows from Proposition \ref{prop:autfms-phiCFs}, as soon as we verify that $D^{(\ell)}_g\eta_g(\tau)$ is a modular form of weight $d=2(\ell-1)$ for some $\Gamma_0(N)$, but this follows from the definition of $D^{(\ell)}_g$. For it is apparent from (\ref{eqn:um:Dellg}) that $D^{(\ell)}_g$ can only by non-zero when the rank of $\LL^g$ is precisely $2d=4(\ell-1)$. On the other hand, if the Frame shape of $g$ takes the form $\pi_g=\prod_{m>0}m^{k_m}$ then $\rank(\LL^g)=\sum_{m>0}k_m$ which is exactly $\frac12$ times the weight of $\eta_g(\tau)=\prod_{m>0}\eta(m\tau)^{k_m}$. So $\eta_g(\tau)$ has weight $d$, or $D^{(\ell)}_g=0$. In either case $D^{(\ell)}_g\eta_g(\tau)\in M_{2(\ell-1)}(\Gamma_0(N))$ for some $N$. This completes the proof.
\end{proof}

Taking $g$ to be the identity in Theorem \ref{theorem:jacobi form} produces a Jacobi form $\phi^{(\ell)}_e$ of weight 0, index $\ell-1$, and level 1. This construction recovers the extremal weak Jacobi forms $Z^{(\ell)}$ of \cite{UM} for $\ell\in\{2,3,4,5,7\}$.
\begin{proposition}
If $d\in\{2,4,6,8,12\}$ then $\phi_e^{(\ell)}=\frac{d}{2}Z^{(\ell)}$.
\end{proposition}

Explicit expressions for the $\phi^{(\ell)}_g$ are recorded in Tables \ref{table:4space}-\ref{table:24space}. Coincidences with the weight zero weak Jacobi forms of umbral moonshine are recorded in Tables \ref{table:4spacecoin}-\ref{table:24spacecoin}. 

Note that $\phi^{(6)}_e$, corresponding to $d=10$, does not correspond to a weak Jacobi form arising in \cite{UM}. However, $\phi^{(6)}_e$ maps naturally to the meromorphic Jacobi form $\psi^X$ for $X=A_5^4D_4$ via the construction in \S 4.3 of \cite{UMNL}. Note that the $\ell$ for which $\phi^{(\ell)}_e$ recovers the weight 0 Jacobi form of umbral moonshine correspond to pure $A$-type root systems $X$. It is natural to ask if some modification of our methods can recover the $Z^{(\ell)}$ corresponding to the remaining pure $A$-type root systems (at $\ell=9, 13, 25$).

\section{Sigma Models}\label{sec:sigmod}

In this section we describe an isomorphism of graded vector spaces relating $\vsn$ and its canonically-twisted module $\vsnt$ to the vector spaces underlying the NS-NS and R-R sectors of an explicitly constructed super conformal field theory arising from a particular, distinguished supersymmetric non-linear K3 sigma model. This model was constructed by Wendland in \cite{Wen2002}. Its automorphism group is exceptionally large, as is demonstrated in \cite{2013arXiv1309.4127G}.

In preparation for a description of the relevant sigma model, let $\Gamma<\CC^2$ be a lattice of rank $4$ that spans $\CC^2$ (i.e. $\Gamma\simeq \ZZ^4$ and $\Gamma\otimes_\ZZ{\RR}=\CC^2$). Then the quotient $T=\CC^2/\Gamma$ is a complex $2$-torus. The {\em Kummer involution} of $T$ is the automorphism induced by the map $\kappa:x\mapsto -x$ on $\CC^2$. A minimal resolution $X\to T/\lab\kappa\rab$ of the quotient (there are $16$ points of $T$ fixed by $\kappa$) is a complex K3 surface and is projective exactly when $T$ is. 

We consider the special case that $\Gamma$ is the $D_4$ lattice (cf. (\ref{eqn:lat-Dn})). More precisely, write $V$ for $\CC^2$ regarded as a real vector space of dimension $4$. Equip $V$ with the symmetric $\RR$-bilinear form $\lab\;,\;\rab$, such that $e_1=(1,0)$, $e_2=(i,0)$, $e_3=(0,1)$ and $e_4=(0,i)$ form an orthonormal basis, and set
\begin{gather}\label{eqn:sigmod-Gamma}
\Gamma:=\left\{\sum_{i=1}^4 n_ie_i\mid n_i\in \ZZ,\,\sum n_i=0\mod{2}\right\}.
\end{gather}
Then $\Gamma$ is a copy of the $D_4$ root lattice in $V$.

Following \cite{2013arXiv1309.4127G}, the Neveu--Schwarz (NS) sector (or rather, the NS-NS sector) of the supersymmetric torus model attached to $T=V/\Gamma$ may be described as
\begin{gather}\label{eqn:sigmod:tormodnsns}
	\mathcal{H}_{T,\text{NS-NS}}=
	\bigoplus_{i\in\{0,1,\w,\wc\}}
	A(\Gamma\otimes_\ZZ{\CC})^{\mathcal{L}}\otimes V_{\Gamma+\gamma_i}^{\mathcal{L}}\otimes 
	A(\Gamma\otimes_\ZZ{\CC})^{\mathcal{R}}\otimes V_{\Gamma+\gamma_i}^{\mathcal{R}},
\end{gather}
when the $\gamma_i$ are chosen so that $\Gamma^*=\bigcup_{i\in\{0,1,\w,\wc\}}\Gamma+\gamma_i$ is the dual lattice to $\Gamma$. 
\begin{gather}\label{eqn:sigmod-Gammastar}
	\Gamma^*=
	\left\{\sum_{i=1}^4 n_ie_i\mid n_i\in \tfrac 12\ZZ,\,n_1=n_2=n_3=n_4 \mod 1\right\}.
\end{gather}
Note that $\Gamma^*/\Gamma$ is a copy of the Klein four-group. 

In (\ref{eqn:sigmod:tormodnsns}) we write $A(\Gamma\otimes_\ZZ\CC)$ for the Clifford module super vertex operator algebra attached to the ($4$ dimensional) complex vector space $\Gamma\otimes_\ZZ\CC$ via the construction recalled in \S\ref{sec:va}, we write $V_{\Gamma+\gamma_i}$ for the module over the lattice vertex operator algebra $V_{\Gamma}$ determined by the coset $\Gamma+\gamma_i$, and we use superscripts $\mathcal{L}$ and $\mathcal{R}$ to distinguish the {\em left-movers} and {\em right-movers}, respectively. The complex structure on $V$ arising from the identification $V=\CC^2$ reflects the choice of B-field made in \cite{2013arXiv1309.4127G}.
In the Ramond-Ramond sector we have
\begin{gather}\label{eqn:sigmod:tormodrr}
	\mathcal{H}_{T,\text{R-R}}=
	\bigoplus_{i\in\{0,1,\w,\wc\}}
	A(\Gamma\otimes_\ZZ\CC)_\tw^{\mathcal{L}}\otimes V_{\Gamma+\gamma_i}^{\mathcal{L}}\otimes A(\Gamma\otimes_\ZZ\CC)_\tw^{\mathcal{R}}\otimes V_{\Gamma+\gamma_i}^{\mathcal{R}}.
\end{gather}

In order to obtain the vector space underlying the minimal resolution $X\to T/\lab \kappa\rab$ we should construct the $\ZZ/2$-orbifold of $\mathcal{H}_{T}$ corresponding to a lift $\hat{\kappa}$ of $\kappa$ to $\Aut(\mathcal{H}_T)$, which means taking $\hat{\kappa}$-fixed points of $\mathcal{H}_T$ together with $\hat{\kappa}$-fixed points of a suitable $\hat{\kappa}$-twisted module for $\mathcal{H}_T$. This leads to
\begin{gather}
\begin{split}
	\mathcal{H}_{X,\text{NS-NS}}=
	&\bigoplus_{i\in\{0,1,\w,\wc\}}
	\left(A(\Gamma\otimes_\ZZ\CC)^{\mathcal{L}}\otimes V_{\Gamma+\gamma_i}^{\mathcal{L}}\otimes A(\Gamma\otimes_\ZZ\CC)^{\mathcal{R}}\otimes V_{\Gamma+\gamma_i}^{\mathcal{R}}\right)^+\\
		\oplus	&\bigoplus_{i\in\{0,1,\w,\wc\}}
	\left(A(\Gamma\otimes_\ZZ\CC)_\tw^{\mathcal{L}}\otimes V_{\Gamma+\gamma_i,\tw}^{\mathcal{L}}\otimes A(\Gamma\otimes_\ZZ\CC)_\tw^{\mathcal{R}}\otimes V_{\Gamma+\gamma_i,\tw}^{\mathcal{R}}\right)^+
\end{split}\\
\begin{split}
	\mathcal{H}_{X,\text{R-R}}=
	&\bigoplus_{i\in\{0,1,\w,\wc\}}
	\left(A(\Gamma\otimes_\ZZ\CC)_\tw^{\mathcal{L}}\otimes V_{\Gamma+\gamma_i}^{\mathcal{L}}\otimes A(\Gamma\otimes_\ZZ\CC)_\tw^{\mathcal{R}}\otimes V_{\Gamma+\gamma_i}^{\mathcal{R}}\right)^+\\
		\oplus	&\bigoplus_{i\in\{0,1,\w,\wc\}}
	\left(A(\Gamma\otimes_\ZZ\CC)^{\mathcal{L}}\otimes V_{\Gamma+\gamma_i,\tw}^{\mathcal{L}}\otimes A(\Gamma\otimes_\ZZ\CC)^{\mathcal{R}}\otimes V_{\Gamma+\gamma_i,\tw}^{\mathcal{R}}\right)^+
\end{split}
\end{gather}
for the NS-NS and R-R sectors of $\mathcal{H}_X$, where the $V_{\Gamma+\gamma_i,\tw}$ are certain twisted modules for $V_{\Gamma}$ and the superscript $+$ denotes $\hat{\kappa}$-fixed points 

At first glance it now appears that a detailed investigation of the structure of $\mathcal{H}_X$ will require a review of the construction of lattice vertex algebras and their twisted modules, but we will refrain from doing that here in favor of using an equivalent description in terms of Clifford modules. 

For this reformulation let $\e$ be a complex vector space of dimension $8$ equipped with a non-degenerate bilinear form. Then we have the Clifford module super vertex operator algebra $A(\e)$ and it's canonically-twisted module $A(\e)_\tw$ as described in \S\ref{sec:va}. According to the boson-fermion correspondence (see \cite{MR643037,MR1284796}, and also \cite{MR1123265} for the particular case of relevance here) we have an isomorphism of vertex operator algebras $A(\e)^0\simeq V_{\Gamma}$ which extends to isomorphisms between the irreducible $A(\e)^0$-modules and the $V_{\Gamma+\gamma_i}$. After relabeling the $\gamma_i$ if necessary, we may assume 
\begin{gather}
	\begin{split}
	A(\e)^0&\simeq V_{\Gamma+\gamma_0},\\
	A(\e)^1&\simeq V_{\Gamma+\gamma_1},\\
	A(\e)_\tw^0&\simeq V_{\Gamma+\gamma_{\w}},\\
	A(\e)_\tw^1&\simeq V_{\Gamma+\gamma_{\wc}}.	
	\end{split}
\end{gather}

We seek some resonance with the notation of \S6 of \cite{SM} for it will develop that the discussion there is very closely related to our present situation. So let us define 
\begin{gather}
	\begin{split}\label{eqn:sigmod-Ui}
	U_0&:=A(\e)^0,\\
	U_1&:=A(\e)^1,\\
	U_\w&:=A(\e)_\tw^0,\\
	U_{\wc}&:=A(\e)_\tw^1,
	\end{split}
\end{gather} 
so that $U_i\simeq V_{\Gamma+\gamma_i}$ as $U_0$-modules. Note the isomorphisms 
\begin{gather}
	A(\e)=U_0\oplus U_1
	\simeq A(\Gamma\otimes_\ZZ\CC)\otimes A(\Gamma\otimes_\ZZ\CC),\\
	A(\e)_\tw=U_{\w}\oplus U_{\wc}
	\simeq A(\Gamma\otimes_\ZZ\CC)_\tw\otimes A(\Gamma\otimes_\ZZ\CC)_\tw.
\end{gather}	
Thus we have isomorphisms 
\begin{gather}
\mathcal{H}_{T,\text{NS-NS}}\label{eqn:sigmod:HTNSNSU}
	\simeq
	\bigoplus_{i\in\{0,1,\w,\wc\}}
	(U_0\oplus U_1)\otimes U_i\otimes U_i,\\
\mathcal{H}_{T,\text{R-R}}\label{eqn:sigmod:HTRRU}
	\simeq
	\bigoplus_{i\in\{0,1,\w,\wc\}}
	(U_{\w} \oplus U_{\wc})\otimes U_i\otimes U_i,
\end{gather}
of $U_0\otimes U_0\otimes U_0$-modules, for the supersymmetric torus model attached to $T$. Now the $U_0\otimes U_0$-module structure on $\bigoplus_iU_i\otimes U_i$ naturally extends to a vertex operator algebra structure as is explained in detail in \cite{MR1123265}. In fact, this vertex operator algebra is isomorphic to the lattice vertex algebra $V_{L}$ for $L$ a copy of the $E_8$ lattice (cf. \S\ref{sec:lat}), and the vertex operator algebra isomorphism $V_L\simeq \bigoplus_i U_i\otimes U_i$ reflects the coincidence 
\begin{gather}
	L=\bigcup_i (\Gamma+\gamma_i)\oplus(\Gamma+\gamma_i),
\end{gather}
expressing the $E_8$ root lattice as a union of cosets for $D_4\oplus D_4$. 

Thus we may interpret (\ref{eqn:sigmod:HTNSNSU}) as an isomorphism of super vertex operator algebras, with each side isomorphic to $A(\e)\otimes V_L$, once we equip $\mathcal{H}_{T,\text{NS-NS}}$, as defined in (\ref{eqn:sigmod:tormodnsns}), with the {\em diagonal} Virasoro element 
\begin{gather}\label{eqn:sigmod-virdiag}
\omega^{\mathcal{D}}:=\omega^{\mathcal{L}}\otimes 1 +1\otimes \omega^{\mathcal{R}},
\end{gather} 
writing here $\omega^{\mathcal{L}}$ for the Virasoro element of $A(\Gamma\otimes_\ZZ\CC)^{\mathcal{L}}\otimes V_{\Gamma}^{\mathcal{L}}$, and similarly for $\omega^{\mathcal{R}}$. With this understanding we may regard (\ref{eqn:sigmod:HTRRU}) as an isomorphism of the corresponding canonically-twisted modules. 

Observe that $A(\e)\otimes V_L$ is precisely the super vertex operator algebra denoted $_{\CC}V_L^f$ in \cite{SM} (the symbols $V_L^f$ denote a real form of $_\CC V_L^f$) and used there to construct an $N=1$ super vertex operator algebra whose automorphism group is the largest simple Conway group, $\Co_1=\Co_0/\{ \pm\Id\}$ (cf. (\ref{eqn:lat-Co1})).  	

The construction of $\mathcal{H}_X$ involves a lift of the Kummer involution $\lambda\mapsto -\lambda$ from $L$ to $V_{L}$ but according to \cite{MR1123265} we may realize such an automorphism explicitly in the $U_0\otimes U_0$-module description as $1\otimes \theta$, where $\theta$ denotes the parity involution on $A(\e)\oplus A(\e)_\tw$, fixing $U_0$ and $U_{\w}$, and negating $U_1$ and $U_{\wc}$. Now we may replace $V_{\Gamma+\gamma_i,\tw}\otimes V_{\Gamma+\gamma_i,\tw}$ with $U_{i}\otimes U_{i+\w}$ in the description of $\mathcal{H}_{X}$, where $\{0,1,\w,\wc\}$ is equipped with the obvious $4$-group structure. Comparing with \cite{2013arXiv1309.4127G} we see that the orbifolding symmetry $\hat{\kappa}$, lifting the Kummer involution on $T$, should act as $\theta\otimes 1\otimes \theta$ on $\mathcal{H}_{T}$ and as $\theta\otimes \theta\otimes 1$ on its $\hat{\kappa}$-twisted module, and in this way we arrive at the isomorphisms
\begin{gather}
	\mathcal{H}_{X,\text{NS-NS}}\simeq U_{000}\oplus U_{0\w\w}\oplus U_{111}\oplus U_{1\wc\wc}\label{eqn:sigmod:HXNSNSU}
		\oplus U_{\w0\w}\oplus U_{\w\w0}\oplus U_{\wc1\wc}\oplus U_{\wc\wc1},\\
	\mathcal{H}_{X,\text{R-R}}\simeq U_{\w00}\oplus U_{\w\w\w}\oplus U_{\wc11}\oplus U_{\wc\wc\wc}\label{eqn:sigmod:HXRRU}
		\oplus U_{00\w}\oplus U_{0\w0}\oplus U_{11\wc}\oplus U_{1\wc1},
\end{gather}
where $U_{ijk}$ is a shorthand for $U_i\otimes U_j\otimes U_k$ (cf. (\ref{eqn:sigmod-Ui})).

Observe that the right hand side of (\ref{eqn:sigmod:HXNSNSU}) is precisely the $U_0\otimes U_0\otimes U_0$-module description given in (6.4.7) of \cite{SM} for the super vertex operator algebra denoted there by $_\CC V^{f\natural}$. A suitably chosen vector $\tau\in U_{111}$ equips $_\CC V^{f\natural}$ with a representation of the Neveu--Schwarz super Lie algebra with central charge $12$, and a main result of \cite{SM} is that the subgroup of $\Aut(_\CC V^{f\natural})$ composed of elements that fix $\tau$ is exactly $\Co_1$. 

We intend to use (\ref{eqn:sigmod:HXRRU}) to relate $\mathcal{H}_{X,\text{R-R}}$ to $\vsnt$. To this purpose, recall the $D_4$ {\em triality}, which, at the level of lattices, is the fact that $\Gamma^*$ (cf. (\ref{eqn:sigmod-Gammastar})) admits an automorphism of order $3$ that stabilizes the type $D_4$ sublattice $\Gamma$ (cf. (\ref{eqn:sigmod-Gamma})), and cyclically permutes its three non-trivial cosets $\Gamma+\gamma_i$. At the level of vertex operator algebras and their modules, this translates to the existence of an automorphism $\sigma$ of $U_0$, and invertible maps $\sigma:U_i\mapsto U_{\w i}$ for $i\in \{1,\w,\wc\}$, such that
\begin{gather}\label{eqn:sigmod-trialitymaps}
	\sigma Y(a,z)c=Y(\sigma a,z)\sigma c
\end{gather}
for $a\in U_0$ and $c\in U_i$, for $i\in \{1,\w,\wc\}$. See \cite{MR1123265} for full details on this. Since $\sigma$ must fix the Virasoro element of $U_0$, the identity (\ref{eqn:sigmod-trialitymaps}) implies that the maps $\sigma:U_i\mapsto U_{\w i}$ are isomorphisms of Virasoro modules. So in particular, the graded dimensions of the $U_i$ coincide, for $i\in \{1,\w,\wc\}$. 

Choosing a decomposition $\a=\e_1\oplus\e_2\oplus \e_3$, with each $\e_i$ a copy of $\CC^8$, non-degenerate with respect to the bilinear form on $\a$, leads to identifications 
\begin{gather}
A(\a)
=A(\e_1)\otimes A(\e_2)\otimes A(\e_3)\simeq \bigoplus_{i,j,k\in\{0,1\}}U_{ijk},\\
A(\a)_\tw
=A(\e_1)_\tw\otimes A(\e_2)_\tw\otimes A(\e_3)_\tw\simeq \bigoplus_{i,j,k\in\{\w,\wc\}}U_{ijk},
\end{gather}
where $U_{ijk}=U_i\otimes U_j\otimes U_k$, as in (\ref{eqn:sigmod:HXNSNSU}) and (\ref{eqn:sigmod:HXRRU}). Consequently, applying (\ref{eqn:dist-vsnvsnt}), we obtain isomorphisms of $U_{000}$-modules,
\begin{gather}
	\vsn\simeq U_{000}\oplus U_{011}\oplus U_{101}\oplus U_{110}\label{eqn:sigmod-vsnUijk}
		\oplus U_{\wc\w\w}\oplus U_{\w\wc\w}\oplus U_{\w\w\wc}\oplus U_{\wc\wc\wc},\\
	\vsnt\simeq U_{100}\oplus U_{010}\oplus U_{001}\oplus U_{111}\label{eqn:sigmod-vsntUijk}
		\oplus U_{\w\w\w}\oplus U_{\w\wc\wc}\oplus U_{\wc\w\wc}\oplus U_{\w\wc\wc}.
\end{gather}

Now consider the images of (\ref{eqn:sigmod-vsnUijk}) and (\ref{eqn:sigmod-vsnUijk}) under $\sigma\otimes\sigma\otimes\sigma$, where $\sigma$ denotes the triality maps of (\ref{eqn:sigmod-trialitymaps}). The result is the right hand sides of (\ref{eqn:sigmod:HXNSNSU}) and (\ref{eqn:sigmod:HXRRU}). 

Thus we have proven the following, final result of the paper.

\begin{proposition}\label{prop:sigmod-Virmodisom}
The distinguished super vertex operator algebra $\vsn$ is isomorphic to $\mathcal{H}_{X,\text{NS-NS}}$ as a Virasoro module, when the latter is equipped with the diagonal Virasoro element, $\omega^{\mathcal{D}}$. Similarly, $\vsnt$ is isomorphic as a Virasoro module to $\mathcal{H}_{X,\text{R-R}}$.
\end{proposition}

\section*{Acknowledgement}

We are grateful to Miranda Cheng, Thomas Creutzig, Xi Dong, Tohru Eguchi, Igor Frenkel, Matthias Gaberdiel, Terry Gannon, Sarah Harrison, Jeff Harvey, Yi-Zhi Huang, Gerald Hoehn, Shamit Kachru, Ching Hung Lam, Atsushi Matsuo, Nils Scheithauer, Yuji Tachikawa, Roberto Volpato, Katrin Wendland and Timm Wrase for helpful discussions on related topics. We are particularly grateful to Miranda Cheng, Jeff Harvey, Shamit Kachru, Katrin Wendland and the anonymous referees, for comments on an earlier draft. The first author gratefully acknowledges support from the U.S. National Science Foundation (DMS 1203162), and from the Simons Foundation (\#316779). The first author thanks the University of Tokyo for hospitality during the completion of this project.

\appendix


\clearpage

\section{Computations}\label{sec:tables-comps}

Table \ref{table:4space} records all the necessary information to compute $\phi_g$ and $F_g$ for all conjugacy classes of $\Co_0$ fixing a rank 4 sublattice of the Leech lattice, including, in particular, the Frame shapes $\pi_{\pm g}$ and the traces $C_{-g}$ and $D_{g}$. The trace $\chi_g$ can be read off from the Frame shape $\pi_g$ as the exponent of $1$, and the rank $\rk\LL^g$ of the sublattice fixed by $g$ is the sum of the exponents (counting signs) in $\pi_g$.

Also recorded in Table \ref{table:4space} are the invariance groups $\Gamma_{\pm g}$ of the $t_{\pm g}$ arising in Proposition \ref{proposition:F_g modular}. Our notation for the $\Gamma_{g}$ is the same as in \cite{vacogm}, and may be described as follows.  We follow the conventions of \cite{MR1291027}, so that when $h$ is the largest divisor of $24$ such that $h^2$ divides $nh$, the symbol $n|h-$ denotes the subgroup of index $h$ in $\Gamma_0(n/h)$ defined in \cite{MR554399}. (See \cite{Fer_Genus0prob} for an analysis of the groups $n|h-$, and their extensions by Atkin--Lehner involutions.) So $12+3$, for example, denotes the group obtained by adjoining an Atkin--Lehner involution $W_3=\frac{1}{\sqrt{3}}\left(\begin{smallmatrix} 3a&b\\12c&3d\end{smallmatrix}\right)$ to $\Gamma_0(12)$, where $9ad-12bc=3$. In addition to this, we use $\up\tfrac1h$ and $n\dn$ to denote upper and lower triangular matrices, respectively, 
\begin{gather}
\up\tfrac1h :=\left(\begin{matrix} 1&\tfrac1h\\0&1\end{matrix}\right),\quad
n\dn :=\left(\begin{matrix} 1&0\\n&1\end{matrix}\right).
\end{gather}

We then write $12+3\up\tfrac12$, for example, for the group generated by $\Gamma_0(12)$ and the product of $W_3$ with $\up\tfrac12$, where $W_3$ is an Atkin--Lehner involution for $\Gamma_0(12)$, as in the previous paragraph.

Note that $\Gamma_g$ and $\Gamma_{-g}$ are related by conjugation by $\up\tfrac12$, for every $g\in \Co_0$.

The data in tables \ref{table:8space} through \ref{table:24space} enable the computation of $\phi^{(\ell)}_g$, for $\ell\in\{3,4,5,7\}$, where the relevant conjugacy classes in $\Co_0$ are those fixing sublattices of the Leech lattice with rank at least $8$, $12$, $16$, or $24$, respectively.

\def\arraystretch{1.3}

\begin{center}
\begingroup
\begin{longtable}{rrccrrll}
\caption{Data for the computation of $\phi_g=\phi_g^{(2)}$}
\label{table:4space} \\

\multicolumn{1}{r}{$\Co_0$} & 
\multicolumn{1}{r}{$\Co_1$} & 
\multicolumn{1}{c}{$\pi_g$} & 
\multicolumn{1}{c}{$\pi_{-g}$} & 
\multicolumn{1}{r}{$C_{-g}$} & 
\multicolumn{1}{r}{$D_{g}$} & 
\multicolumn{1}{l}{$\Gamma_{g}$} & 
\multicolumn{1}{l}{$\Gamma_{-g}$} \\ \hline
\endfirsthead

\multicolumn{8}{c}
{{\tablename\ \thetable{}, continued from previous page}} \\
\multicolumn{1}{r}{$\Co_0$} & 
\multicolumn{1}{r}{$\Co_1$} & 
\multicolumn{1}{c}{$\pi_g$} & 
\multicolumn{1}{c}{$\pi_{-g}$} & 
\multicolumn{1}{r}{$C_{-g}$} & 
\multicolumn{1}{r}{$D_{g}$} & 
\multicolumn{1}{l}{$\Gamma_{g}$} & 
\multicolumn{1}{l}{$\Gamma_{-g}$} \\ \hline
\endhead

\hline \multicolumn{8}{r}{{Continued on next page}} \\
\endfoot

\hline
\endlastfoot

1A & 1A & $1^{24}$ & $\frac{2^{24}}{1^{24}}$ & $4096$ & 0 & $2-$ & $4+$ \\

2B & 2A & $1^8 2^8$ & $\frac{2^{16}}{1^8}$ & 0 & 0 & $4-$ & $4-$ \\
2C & 2A & $\frac{2^{16}}{1^8}$ & $1^8 2^8$ & 0 & 0 & $4-$ & $4-$ \\

2D & 2C & $2^{12}$ & $2^{12}$ & 0 & 0 & $4|2-$ & $4|2-$ \\

3B & 3B & $1^6 3^6$ & $\frac{2^6 6^6}{1^6 3^6}$ & $64$ & 0 & $6+3$ & $12+$ \\

3C & 3C & $\frac{3^9}{1^3}$ & $\frac{1^3 6^9}{2^3 3^9}$ & $-8$ & 0 & $6-$ & $12+4$ \\

3D & 3D & $3^8$ & $\frac{6^8}{3^8}$ & 16 & 0 & $6|3$ & $12|3+$ \\

4B & 4A & $\frac{1^8 4^8}{2^8}$ & $\frac{4^8}{1^8}$ & $256$ & 0 & $(8+)^{\up\tfrac12}$ & $8+$ \\

4D & 4B & $\frac{4^8}{2^4}$ & $\frac{4^8}{2^4}$ & 0 & $\pm 64$ & $8-$ & $8-$ \\

4E & 4C & $1^4 2^2 4^4$ & $\frac{2^6 4^4}{1^4}$ & 0 & 0 & $8-$ & $8-$ \\
4F & 4C & $\frac{2^6 4^4}{1^4}$ & $1^4 2^2 4^4$ & 0 & 0 & $8-$ & $8-$ \\

4G & 4D & $2^4 4^4$ & $2^4 4^4$ & 0 & 0 & $8|2-$ & $8|2-$ \\

4H & 4F & $4^6$ & $4^6$ & 0 & 0 & $8|4-$ & $8|4-$ \\

5B & 5B & $1^4 5^4$ & $\frac{2^4 10^4}{1^4 5^4}$ & $16$ & 0 & $10+5$ & $20+$ \\

5C & 5C & $\frac{5^5}{1^1}$ & $\frac{1^1 10^5}{2^1 5^5}$ & $-4$ & $\pm25\sqrt{5}$ & $10-$ & $20+4$ \\

6G & 6C & $\frac{2^5 3^4 6^1 }{1^4}$ & $\frac{1^4 2^1 6^5}{3^4}$ & 0 & 0 & $12+3\up\tfrac12$ & $12+3\up\tfrac12$ \\
6H & 6C & $\frac{1^4 2^1 6^5}{3^4}$ & $\frac{2^5 3^4 6^1}{1^4}$ & 0 & 0 & $12+3\up\tfrac12$ & $12+3\up\tfrac12$ \\

6I & 6D & $\frac{1^5 3^1 6^4}{2^4}$ & $\frac{2^1 6^5}{1^5 3^1}$ & $72$ & 0 & $(12+12)^{\up\tfrac12}$ & $12+12$ \\

6K & 6E & $1^2 2^2 3^2 6^2$ & $\frac{2^4 6^4}{1^2 3^2}$ & 0 & 0 & $12+3$ & $12+3$ \\
6L & 6E & $\frac{2^4 6^4}{1^2 3^2}$ & $1^2 2^2 3^2 6^2$ & 0 & $\pm48$ & $12+3$ & $12+3$ \\

6M & 6F & $\frac{3^3 6^3}{1^1 2^1}$ & $\frac{1^1 6^6}{2^2 3^3}$ & 0 & $\pm54$ & $12-$ & $12-$ \\

6O & 6G & $2^3 6^3$ & $2^3 6^3$ & 0 & 0 & $12\vert2+3\up\tfrac12$ & $12\vert2+3\up\tfrac12$ \\

6P & 6I & $6^4$ & $6^4$ & 0 & $\pm36$ & $12|6-$ & $12|6-$ \\

7B & 7B & $1^3 7^3$ & $\frac{2^3 14^3}{1^3 7^3}$ & $8$ & 0 & $14+7$ & $28+$ \\

8C & 8B & $\frac{2^4 8^4}{4^4}$ & $\frac{2^4 8^4}{4^4}$ & 0 & $\pm16$ & $(16\vert2+)^{\up\tfrac14}$ & $(16\vert2+)^{\up\tfrac14}$ \\

8D & 8C & $\frac{1^4 8^4}{2^2 4^2}$ & $\frac{2^2 8^4}{1^4 4^2}$ & $32$ & $\pm8$ & $(16+)^{\up\tfrac12}$ & $16+$ \\

8G & 8E & $1^2 2^1 4^1 8^2$ & $\frac{2^3 4^1 8^2}{1^2}$ & 0 & 0 & $16-$ & $16-$ \\
8H & 8E & $\frac{2^34^1 8^2}{1^2}$ & $1^2 2^1 4^1 8^2$ & 0 & $\pm32\sqrt{2}$& $16-$ & $16-$ \\

8I & 8F & $4^2 8^2$ & $4^2 8^2$ & 0 & $\pm 16$ & $16\vert4-$ & $16\vert4-$ \\

9C & 9C & $\frac{1^3 9^3}{3^2}$ & $\frac{2^3 3^2 18^3}{1^3 6^2 9^3}$ & $4$ & $\pm9$ & $18+9$ & $36+$ \\

10F & 10D & $\frac{2^3 5^2 10^1 }{1^2}$ & $\frac{1^2 2^1 10^3}{5^2}$ & 0 & $\pm20\sqrt{5}$ & $20+5\up\tfrac12$ & $20+5\up\tfrac12$ \\
10G & 10D & $\frac{1^2 2^1 10^3}{5^2}$ & $\frac{2^3 5^2 10^1 }{1^2}$ & 0 & $\pm4\sqrt{5}$ & $20+5\up\tfrac12$ & $20+5\up\tfrac12$ \\

10H & 10E & $\frac{1^3 5^1 10^2}{2^2}$ & $\frac{2^1 10^3}{1^3 5^1}$ & 20 & $\pm5\sqrt{5}$ & $(20+20)^{\up\tfrac12}$ & $20+20$ \\

10J & 10F & $2^2 10^2$ & $2^2 10^2$ & 0 & $\pm20$ & $20\vert2+5$ & $20\vert2+5$ \\

11A & 11A & $1^2 11^2$ & $\frac{2^2 22^2}{1^2 11^2}$ & $4$ & $\pm11$ & $22+11$ & $44+$ \\

12I & 12E & $\frac{1^2 3^2 4^2 12^2}{2^2 6^2}$ & $\frac{4^2 12^2}{1^2 3^2}$ & $16$ & $\pm12$ & $(24+)^{\up\tfrac12}$ & $24+$ \\

12L & 12H & $\frac{1^1 2^2 3^1 12^2}{4^2}$ & $\frac{2^3 6^1 12^2}{1^1 3^1 4^2}$ & 0 & $\pm6\sqrt{3}$ & $(24\vert2+12)^{\up\tfrac14}$ & $(24\vert2+12)^{\up\tfrac14}$ \\

12N & 12I & $\frac{2^2 3^2 4^1 12^1}{1^2}$ & $\frac{1^2 4^1 6^2 12^1}{3^2}$ & 0 & $\pm 24\sqrt{3}$ & $24+3\up\tfrac12$ & $24+3\up\tfrac12$ \\
12O & 12I & $\frac{1^2 4^1 6^2 12^1}{3^2}$ & $\frac{2^2 3^2 4^1 12^1}{1^2}$ & 0 & $\pm 8\sqrt{3}$ & $24+3\up\tfrac12$ & $24+3\up\tfrac12$ \\

12P & 12J & $2^1 4^1 6^1 12^1$ & $2^1 4^1 6^1 12^1$ & 0 & $\pm24$ & $24\vert2+3$ & $24\vert2+3$ \\

14C & 14B & $1^1 2^1 7^1 14^1$ & $\frac{2^2 14^2}{1^1 7^1}$ & 0 & $\pm14$ & $28+7$ & $28+7$ \\

15D & 15D & $1^1 3^1 5^1 15^1$ & $\frac{2^1 6^1 10^1 30^1}{1^1 3^1 5^1 15^1}$ & $4$ & $\pm15$ & $30+3,5,15$ & $60+$ \\

\end{longtable}
\end{center}
\endgroup

\begin{table}[ht]
\centering
\caption{Data for the computation of $\phi^{(3)}_g$}
\label{table:8space}
\def\arraystretch{1.3}
\begin{tabular}{cccccccc}

$\Co_0$ & $\Co_1$ & $\pi_{g}$ & $\pi_{-g}$ & $C_{-g}$ & $D^{(3)}_{g}$ & $\Gamma_g$ & $\Gamma_{-g}$ \\

\hline

1A & 1A & $1^{24}$ & $\frac{2^{24}}{1^{24}}$ & $4096$ & 0 & $2-$ & $4+$ \\
2B & 2A & $1^8 2^8$ & $\frac{2^{16}}{1^8}$ & 0 & 0 & $4-$ & $4-$ \\
2C & 2A & $\frac{2^{16}}{1^8}$ & $1^8 2^8$ & 0 & $\pm256$ & $4-$ & $4-$ \\
2D & 2C & $2^{12}$ & $2^{12}$ & 0 & 0 & $4|2-$ & $4|2-$ \\
3B & 3B & $1^6 3^6$ & $\frac{2^6 6^6}{1^6 3^6}$ & 64 & 0 & $6+3$ & $12+$ \\
3D & 3D & $3^8$ & $\frac{6^8}{3^8}$ & 16 & $\pm81$ & $6|3$ & $12|3+$ \\
4B & 4A & $\frac{1^8 4^8}{2^8}$ & $\frac{4^8}{1^8}$ & 256 & $\pm16$ & $(8+)^{\up\tfrac12}$ & $8+$ \\
4E & 4C & $1^4 2^2 4^4$ & $\frac{2^6 4^4}{1^4}$ & 0 & 0 & $8-$ & $8-$ \\
4G & 4D & $2^4 4^4$ & $2^4 4^4$ & 0 & $\pm64$ & $8|2-$ & $8|2-$ \\
5B & 5B & $1^4 5^4$ & $\frac{2^4 10^4}{1^4 5^4}$ & 16 & $\pm25$ & $10+5$ & $20+$ \\
6K & 6E & $1^2 2^2 3^2 6^2$ & $\frac{2^4 6^4}{1^2 3^2}$ & 0 & $\pm36$ & $12+3$ & $12+3$ \\

\hline
\end{tabular}
\end{table}

\begin{table}
\centering
\caption{Data for the computation of $\phi^{(4)}_g$}
\label{table:12space}
\def\arraystretch{1.3}
\begin{tabular}{ccccccccp{5.25cm}}

$\Co_0$ & $\Co_1$ & $\pi_{g}$ & $\pi_{-g}$ & $C_{-g}$ & $D^{(4)}_{g}$ & $\Gamma_g$ & $\Gamma_{-g}$ \\

\hline

1A & 1A & $1^{24}$ & $\frac{2^{24}}{1^{24}}$ & $4096$ & 0 & $2-$ & $4+$ \\
2B & 2A & $1^8 2^8$ & $\frac{2^{16}}{1^8}$ & 0 & 0 & $4-$ & $4-$ \\
2D & 2C & $2^{12}$ & $2^{12}$ & 0 & $\pm64$ & $4|2-$ & $4|2-$ \\
3B & 3B & $1^6 3^6$ & $\frac{2^6 6^6}{1^6 3^6}$ & 64 & $\pm27$ & $6+3$ & $12+$ \\

\hline
\end{tabular}
\end{table}

\begin{table}
\centering
\caption{Data for the computation of $\phi^{(5)}_g$}
\label{table:16space}
\def\arraystretch{1.3}
\begin{tabular}{ccccccccc}

$\Co_0$ & $\Co_1$ & $\pi_{g}$ & $\pi_{-g}$ & $C_{-g}$ & $D^{(5)}_{g}$ & $\Gamma_g$ & $\Gamma_{-g}$ \\

\hline

1A & 1A & $1^{24}$ & $\frac{2^{24}}{1^{24}}$ & $4096$ & 0 & $2-$ & $4+$ \\
2B & 2A & $1^8 2^8$ & $\frac{2^{16}}{1^8}$ & 0 & $\pm16$ & $4-$ & $4-$ \\

\hline
\end{tabular}
\end{table}
 
\begin{table}
\centering
\caption{Data for the computation of $\phi^{(7)}_g$}
\label{table:24space}
\def\arraystretch{1.3}
\begin{tabular}{ccccccccc}

$\Co_0$ & $\Co_1$ & $\pi_{g}$ & $\pi_{-g}$ & $C_{-g}$ & $D^{(7)}_{g}$ & $\Gamma_g$ & $\Gamma_{-g}$ \\

\hline

1A & 1A & $1^{24}$ & $\frac{2^{24}}{1^{24}}$ & $4096$ & $\pm1$ & $2-$ & $4+$ \\

\hline
\end{tabular}
\end{table}

\clearpage

\section{Coincidences}\label{sec:tables-coins}

Table \ref{table:4spacecoin} records instances in which $\phi_g$ coincides with (or is a simple linear combination of) weight zero (weak) Jacobi forms $Z^{(2)}_g$ attached to elements $g\in M_{24}$ by Mathieu moonshine, being the $\ell=2$ case of umbral moonshine (cf. \cite{UM,UMNL}). The functions $Z^{(2)}_g$ are as defined in \cite{UM}. We also indicate when $\phi_g$ recovers one of the twined K3 sigma model elliptic genera that is computed explicitly in \cite{GHV}. The notations $\phi_{nZ}$ and $\hat{\phi}_{nz}$ are as in \cite{GHV}. Since there is a choice, we specify which $D_{g}$ produces the function in question. Observe that every twined K3 sigma model elliptic genus appearing in \cite{GHV} appears also in Table \ref{table:4spacecoin}.

Tables \ref{table:8spacecoin} through \ref{table:24spacecoin} present coincidences between the $\phi^{(\ell)}_g$ and the functions $Z^{(\ell)}_g$ of umbral moonshine, for $\ell\in\{3,4,5,7\}$. As in the case of $\ell=2$, the functions $Z^{(\ell)}_g$ are as defined in \cite{UM}.

\def\arraystretch{1.3}

\begin{center}
\begingroup
\begin{longtable}{rrrcc}
\caption{Coincidences with Mathieu moonshine and sigma model twining genera}
\label{table:4spacecoin} \\

\multicolumn{1}{r}{$\Co_0$} & 
\multicolumn{1}{r}{$\Co_1$} & 
\multicolumn{1}{r}{$D_{g}$} & 
\multicolumn{1}{c}{$\phi_g$} & 
\multicolumn{1}{c}{$N_g$} \\ \hline
\endfirsthead

\multicolumn{5}{c}
{{\tablename\ \thetable{}, continued from previous page}} \\
\multicolumn{1}{r}{$\Co_0$} & 
\multicolumn{1}{r}{$\Co_1$} & 
\multicolumn{1}{r}{$D_{g}$} & 
\multicolumn{1}{c}{$\phi_g$} & 
\multicolumn{1}{c}{$N_g$} \\ \hline
\endhead

\hline \multicolumn{5}{r}{{Continued on next page}} \\
\endfoot

\hline
\endlastfoot

1A & 1A & 0 & $Z^{(2)}_{1A}=\phi_{1A}$ & 1 \\
2B & 2A & 0 & $Z^{(2)}_{2A}=\phi_{2A}$ & 2 \\
2C & 2A & 0 & $-Z^{(2)}_{1A}+2Z^{(2)}_{2A}=\phi_{\mathcal{Q}}$ & 2 \\
2D & 2C & 0 & $Z^{(2)}_{2B}=\phi_{2B}$ & 4 \\
3B & 3B & 0 & $Z^{(2)}_{3A}=\phi_{3A}$ & 3 \\
3C & 3C & 0 & $-\frac{1}{2}Z^{(2)}_{1A}+\frac{3}{2}Z^{(2)}_{3A}=\hat{\phi}_{3a}$ & 3 \\
4B & 4A & 0 & $Z^{(2)}_{2A}=\phi_{2A}$ & 2 \\
4D & 4B & $64$ & $Z^{(2)}_{2B}=\phi_{2B}$ & 4 \\
4D & 4B & $-64$ & $-\frac{1}{2}Z^{(2)}_{1A}+\frac{3}{2}Z^{(2)}_{2A}$ & 2 \\
4E & 4C & 0 & $Z^{(2)}_{4B}=\phi_{4B}$ & 4 \\
4F & 4C & 0 & $-\frac{1}{2}Z^{(2)}_{1A}+\frac{1}{2}Z^{(2)}_{2A}+Z^{(2)}_{4B}=\hat{\phi}_{4a}$ & 4 \\
4G & 4D & 0 & $Z^{(2)}_{4A}=\phi_{4A}$ & 8 \\
5B & 5B & 0 & $Z^{(2)}_{5A}=\phi_{5A}$ & 5 \\
6G & 6C & 0 & $-\frac{1}{2}Z^{(2)}_{1A}+\frac{1}{2}Z^{(2)}_{2A}+\frac{1}{2}Z^{(2)}_{3A}+\frac{1}{2}Z^{(2)}_{6A}$ & 6 \\
6H & 6C & 0 & $\frac{1}{2}Z^{(2)}_{3A}+\frac{1}{2}Z^{(2)}_{6A}$ & 6 \\
6I & 6D & 0 & $\frac{1}{2}Z^{(2)}_{2A}+\frac{1}{2}Z^{(2)}_{6A}=\hat{\phi}_{6a}$ & 6 \\
6K & 6E & 0 & $Z^{(2)}_{6A}=\phi_{6A}$ & 6 \\
6L & 6E & $48$ & $-Z^{(2)}_{3A}+2Z^{(2)}_{6A}$ & 6 \\
6L & 6E & $-48$ & $-\frac{1}{2}Z^{(2)}_{1A}+\frac{1}{2}Z^{(2)}_{2A}+Z^{(2)}_{3A}$ & 6 \\
6M & 6F & $54$ & $-\frac{1}{2}Z^{(2)}_{2A}+\frac{3}{2}Z^{(2)}_{6A}$ & 6 \\
6M & 6F & $-54$ & $-\frac{1}{2}Z^{(2)}_{1A}+Z^{(2)}_{2A}+\frac{1}{2}Z^{(2)}_{3A}$ & 6 \\
7B & 7B & 0 & $Z^{(2)}_{7AB}$ & 7 \\
8C & 8B & 16 & $Z^{(2)}_{4C}=\phi_{4C}$ & 16 \\
8D & 8C & $8$ & $Z^{(2)}_{4B}=\phi_{4B}$& 4 \\
8D & 8C & $-8$ & $\frac{1}{2}Z^{(2)}_{2A}+\frac{1}{2}Z^{(2)}_{4A}$ & 8 \\
8G & 8E & 0 & $Z^{(2)}_{8A}=\phi_{8A}$ & 8 \\
9C & 9C & $9$ & $\frac{1}{2}Z^{(2)}_{3A}+\frac{1}{2}Z^{(2)}_{3B}=\hat{\phi}_{9a}$ & 9 \\
9C & 9C & $-9$ & $\hat{\phi}_{9b}$ & 9 \\
10J & 10F & 20 & $Z^{(2)}_{10A}$ & 20 \\
11A & 11A & $11$ & $Z^{(2)}_{11A}$ & 11 \\
12I & 12E & $12$ & $Z^{(2)}_{6A}=\phi_{6A}$ & 6 \\
12P & 12J & 24 & $Z^{(2)}_{12A}$ & 24 \\
14C & 14B & $14$ & $Z^{(2)}_{14AB}$ & 14 \\
15D & 15D & $15$ & $Z^{(2)}_{15AB}$ & 15 \\

\end{longtable}
\end{center}
\endgroup

\clearpage

\begin{table}
\centering
\caption{Coincidences with umbral moonshine at $\ell=3$}
\label{table:8spacecoin}
\def\arraystretch{1.3}
\begin{tabular}{cccccccp{4cm}}

$\Co_0$&$\Co_1$ & $D^{(3)}_{g}$ & $\phi^{(3)}_g$ & $N_g$ \\

\hline

1A & 1A & 0 & $2Z^{(3)}_{1A}$ \\
2B & 2A & 0 & $2Z^{(3)}_{2B}$ \\
2C & 2A & $-256$ & $-2Z^{(3)}_{2A}+4Z^{(3)}_{2B}$ \\
3B & 3B & 0 & $2Z^{(3)}_{3A}$ \\
3D & 3D & $-81$ & $2Z^{(3)}_{3B}$ \\
4B & 4A & $-16$ & $2Z^{(3)}_{2B}$ \\
4G & 4D & $-64$ & $2Z^{(3)}_{4B}$ \\
5B & 5B & $-25$ & $2Z^{(3)}_{5A}$ \\
6K & 6E & $-36$ & $2Z^{(3)}_{6C}$ \\

\hline
\end{tabular}
\end{table}

\begin{table}
\centering
\caption{Coincidences with umbral moonshine at $\ell=4$}
\label{table:12spacecoin}
\def\arraystretch{1.3}
\begin{tabular}{cccccccp{5.25cm}}

$\Co_0$ & $\Co_1$ & $D^{(4)}_{g}$ & $\phi^{(4)}_g$ & $N_g$ \\

\hline

1A & 1A & 0 & $3Z^{(4)}_{1A}$ \\
2D & 2C & 64 & $3Z^{(4)}_{2B}$ \\
3B & 3B & 27 & $3Z^{(4)}_{3A}$ \\

\hline
\end{tabular}
\end{table}

\begin{table}
\centering
\caption{Coincidences with umbral moonshine at $\ell=5$}
\label{table:16spacecoin}
\def\arraystretch{1.3}
\begin{tabular}{cccccccc}

$\Co_0$ & $\Co_1$ & $D^{(5)}_{g}$ & $\phi^{(5)}_g$ & $N_g$ \\

\hline

1A & 1A & 0 & $4Z^{(5)}_{1A}$ \\
2B & 2A & $-16$ & $4Z^{(5)}_{2B}$ \\

\hline
\end{tabular}
\end{table}

\begin{table}
\centering
\caption{Coincidences with umbral moonshine at $\ell=7$}
\label{table:24spacecoin}
\def\arraystretch{1.3}
\begin{tabular}{cccccccc}

$\Co_0$ & $\Co_1$ & $D^{(7)}_{g}$ & $\phi^{(7)}_g$ & $N_g$ \\

\hline

1A & 1A & $-1$ & $6Z^{(7)}_{1A}$ \\

\hline
\end{tabular}
\end{table}


\end{document}